\documentclass[11pt,reqno,tbtags]{amsart}
\usepackage{amssymb}
\usepackage{ifdraft}
\usepackage[square,numbers]{natbib}

\numberwithin{equation}{section}

\allowdisplaybreaks

\newtheorem{theorem}{Theorem}[section]

\newtheorem{corollary}[theorem]{Corollary}

\theoremstyle{definition}
\newtheorem{example}[theorem]{Example}
\newtheorem{definition}[theorem]{Definition}

\newtheorem{remark}[theorem]{Remark}
\newtheorem*{warning}{Warning}

\theoremstyle{remark}

\newenvironment{romenumerate}[1][0pt]{
\addtolength{\leftmargini}{#1}\begin{enumerate}
 }{\end{enumerate}}

\newcounter{oldenumi}
{\setcounter{oldenumi}{\value{enumi}}
\begin{romenumerate} \setcounter{enumi}{\value{oldenumi}}}
{\end{romenumerate}}

\newcounter{thmenumerate}

\newcounter{romxenumerate}   

\newcounter{xenumerate}   

\newcommand\pfitemx[1]{\par#1:}
\newcommand\pfitemref[1]{\pfitemx{\ref{#1}}}

\newcommand{\refT}[1]{Theorem~\ref{#1}}
\newcommand{\refC}[1]{Corollary~\ref{#1}}

\newcommand{\refR}[1]{Remark~\ref{#1}}
\newcommand{\refS}[1]{Section~\ref{#1}}
\newcommand{\refSS}[1]{Subsection~\ref{#1}}

\newcommand{\refE}[1]{Example~\ref{#1}}

\newcommand{\refand}[2]{\ref{#1} and~\ref{#2}}

\newcommand\marginal[1]{\ifdraft
{\marginpar[\raggedleft\tiny #1]{\raggedright\tiny #1}}
{\message{ERROR marginal requires draft option}}}

\begingroup
  \divide\count255 by 60
  \multiply\count255 by -60
  \advance\count255 by \time
  \ifnum \count255 < 10 \xdef\klockan{\the\count1.0\the\count255}
  \else\xdef\klockan{\the\count1.\the\count255}\fi
\endgroup

\DeclareMathOperator*{\sumx}{\sum\nolimits^{*}}

\newcommand{\sumj}{\sum_{j=0}^\infty}
\newcommand{\sumk}{\sum_{k=0}^\infty}

\newcommand{\sumin}{\sum_{i=0}^{\nn}}
\newcommand{\sumiin}{\sum_{i=1}^{\nn}}
\newcommand{\sumini}{\sum_{i=0}^{\nn_1}}
\newcommand{\suminii}{\sum_{i=0}^{\nn_2}}
\newcommand{\prodin}{\prod_{i=1}^{\nn}}
\newcommand{\prodkn}{\prod_{k=1}^{\nn}}
\newcommand{\prodjm}{\prod_{j=1}^{m}}

\newcommand\set[1]{\ensuremath{\{#1\}}}

\newcommand\xpar[1]{(#1)}
\newcommand\bigpar[1]{\bigl(#1\bigr)}
\newcommand\Bigpar[1]{\Bigl(#1\Bigr)}
\newcommand\biggpar[1]{\biggl(#1\biggr)}
\newcommand\lrpar[1]{\left(#1\right)}

\newcommand\bigcpar[1]{\bigl\{#1\bigr\}}
\newcommand\Bigcpar[1]{\Bigl\{#1\Bigr\}}

\newcommand\lrcpar[1]{\left\{#1\right\}}

\newcommand\bigabs[1]{\bigl|#1\bigr|}

\newcommand\lrabs[1]{\left|#1\right|}
\def\rompar(#1){\textup(#1\textup)}    
\newcommand\xfrac[2]{#1/#2}

\newcommand\parfrac[2]{\lrpar{\frac{#1}{#2}}}
\newcommand\bigparfrac[2]{\bigpar{\frac{#1}{#2}}}
\newcommand\Bigparfrac[2]{\Bigpar{\frac{#1}{#2}}}

\def\xexp(#1){e^{#1}}

\newcommand\punkt[1]{\if.#1\else.\spacefactor1000\fi{#1}}

\newcommand\ie{i.e\punkt}
\newcommand\eg{e.g\punkt}
\newcommand\viz{viz\punkt}
\newcommand\cf{cf\punkt}

\newcommand\etc{etc\punkt}

\newcommand\ii{\mathrm{i}}

\newcommand\bbR{\mathbb R}
\newcommand\bbC{\mathbb C}

\newcommand\bbQ{\mathbb Q}

\newcommand\bbF{\mathbb F}

\newcounter{CC}
\newcounter{cc}

\newcommand\E{\operatorname{\mathbb E{}}}

\newcommand\sign{\operatorname{sign}}

\newcommand\ga{\alpha}
\newcommand\gb{\beta}
\newcommand\gd{\delta}
\newcommand\gD{\Delta}
\newcommand\gf{\varphi}
\newcommand\gam{\gamma}

\newcommand\gl{\lambda}
\newcommand\gL{\Lambda}
\newcommand\go{\omega}
\newcommand\gO{\Omega}

\newcommand\eps{\varepsilon}

\newcommand\cB{\mathcal B}

\newcommand\cD{\mathcal D}

\newcommand\cH{\mathcal H}

\newcommand\cP{\mathcal P}

\newcommand\smatrixx[1]{\left(\begin{smallmatrix}#1\end{smallmatrix}\right)}
\newcommand\matrixx[1]{\begin{pmatrix}#1\end{pmatrix}}

\newcommand\qw{^{-1}}
\newcommand\qww{^{-2}}

\newcommand\qqc{^{3/2}}

\renewcommand{\=}{:=}

\newcommand\xx{\mathbf{x}}
\newcommand\yy{\mathbf{y}}
\newcommand\ax{\mathbf{a}}
\newcommand\ha{\check a}
\newcommand\hb{\check b}
\newcommand\tf{\widetilde f}
\newcommand\tPsi{\widetilde \Psi}
\newcommand\tPhi{\widetilde \Phi}
\newcommand\refl{^\dag}
\newcommand\xij{\xi^{(j)}}
\newcommand\xin{\xi_1,\dots,\xi_{\nn}}
\newcommand\xijn{\xij_1,\dots,\xij_{\nn}}
\newcommand\xiln{\xi^{(1)}_1,\dots,\xi^{(\ell)}_\nn}
\newcommand\xini{\xi_1,\dots,\xi_{\nn-1}}
\newcommand\Phix{\Phi^*}
\newcommand\red{\widehat}
\newcommand\ar{\red a}
\newcommand\fr{\red f}
\newcommand\Han{\operatorname{Han}}
\newcommand\scite[1]{_{\text{\cite{#1}}}}
\newcommand\scitex[2]{_{#1\,\text{\cite{#2}}}}
\newcommand\tH{\widetilde H}
\newcommand\hx{\widetilde x}
\newcommand\hy{\widetilde y}
\newcommand\JH{G_6}
\newcommand\tJH{\widetilde{\JH}}
\newcommand\gOx{\Omega^*}
\newcommand\Res{\operatorname{Res}}
\newcommand\tRes{\red{\Res}}
\newcommand\Resx{\widetilde{\Res}}
\newcommand\Resy{{\Res^*}}
\newcommand\tp{\tilde p}
\newcommand\tq{\tilde q}
\newcommand\tr{\tilde r}
\newcommand\dg[1]{\langle#1\rangle}
\newcommand\dgx[1]{_{\dg{#1}}}
\newcommand\dgxx[2]{_{#1\,\dg{#2}}}
\newcommand\restr[1]{|\dgx{#1}}
\newcommand\nn{n}
\newcommand\pd{\partial}
\newcommand\pdx{\pd_x}
\newcommand\pdy{\pd_y}
\newcommand\Fx{F^*}
\newcommand\crx[1]{[#1_1,#1_2;#1_3,#1_4]}
\newcommand\crxx[4]{[x_#1,x_#2;x_#3,x_#4]}
\newcommand\mm{\mu}
\newcommand\tv[1]{\{#1\}}
\newcommand\bigtv[1]{\bigcpar{#1}}
\newcommand\mud{d}
\newcommand\bcP{\overline{\cP}}
\newcommand\GQ[2]{\genfrac{[}{]}{0pt}{}{#1}{#2}_q}
\newcommand\Satze{S\"atze }
\newcommand\bxi{\overline\xi}
\newcommand\xibxi{\ensuremath{\xi_1-\bxi,\dots,\xi_n-\bxi}}
\newcommand\orr{\emph{or} }

\newcommand{\maple}{\texttt{Maple}}

\newcommand\REM[1]{{\raggedright\texttt{[#1]}\par\marginal{XXX}}}

\newcommand\urladdrx[1]{{\urladdr{\def~{{\tiny$\sim$}}#1}}}

\begin{document}
\title
{Invariants of polynomials and binary forms}

\date{14 February, 2011} 

\author{Svante Janson}
\address{Department of Mathematics, Uppsala University, PO Box 480,
SE-751~06 Uppsala, Sweden}
\email{svante.janson@math.uu.se}
\urladdrx{http://www.math.uu.se/~svante/}

\subjclass[2000]{} 

\begin{abstract} 
We survey various classical
results on invariants of polynomials, or equivalently, of binary forms,
focussing on
explicit calculations for invariants of polynomials
of degrees 2, 3, 4.
\end{abstract}

\maketitle

\section{Introduction}\label{S:intro}

The purpose of this survey is to collect various 
classical (mainly 19th century) results on invariants of polynomials,
focussing on
explicit formulas for invariants of polynomials
of degrees 2, 3, 4.
Invariants of polynomials are equivalent to
invariants of binary forms, so we begin (\refS{Sbinary})
with a summary of definitions and some key result for these,
mainly based on \citet{Schur}; some other books on invariants
(which we only partly have consulted) are
\citet{Dickson}, \citet{Elliott}, \citet{Glenn}, \citet{Hilbert},
\citet{Olver}. See these books for further results and proofs.
Some formulas below have been calculated using \maple.

The theory is really simpler and more symmetric for binary forms, 
and the obvious correspondence between binary forms and polynomials (see
\refS{Spol}) makes it in principle trivial to transfer the definitions and
results to polynomials. 
Nevertheless, since polynomials are so common in other parts of mathematics,
we find it interesting to perform this translation explicitly and to give
detailed formulas for polynomials.

\begin{remark}
  The formulas are purely algebraic and are valid for any ground field of
  characteristic 0, for example $\bbQ$, $\bbR$ or $\bbC$.   

The formulas
  give invariants also for fields of finite characteristic, at least as long as
it does not divide any denominator (for degree $\le4$, only characteristic 2
or 3 may have such problems),
but there
  are also other invariants in finite characteristic. One example is the
  invariant \cite{Dickson09}
  \begin{equation}\label{F3}
a_0^{2}{a_2}
+{a_0}\,a_2^{2}
+{a_0}\,a_1^{2}
+a_1^{2}{a_2}
-a_0^{3}-a_2^{3}
  \end{equation}
of a quadratic polynomial $a_0x^2+a_1x+a_2$ in $\bbF_3$. 
(Cf.\  \refS{S2}, and note that \eqref{F3} does not vanish for $f(x)=x^2$,
unlike the discriminant $\gD$.)
See further
\cite{Dickson} and, for example, \cite{Williams}. 
\end{remark}

We ignore trivial complications with the invariant that is identically 0;
for example, we may say that there is no invariant of some type, really
meaning that there is no such invariant that is not identically zero.
Similarly, we for simplicity may say that an invariant $\Phi$ is the only
invariant of some type, really meaning this up to constant factors, \ie{}
that every such invariant is a multiple $c\Phi$ of $\Phi$ (in other words,
the space of such invariants is 1-dimensional).
Note also that constant factors in the definition of specific invariants 
usually are uninteresting, and different choices of such factors often are
made in different references.

When giving examples of different notations in other papers and books, we
use subscripts; for 
example, $A\scite{Schur}$ means $A$ in \cite{Schur}.

We denote falling factorials by
\begin{equation}
(n)_k\=n(n-1)\dotsm(n-k+1)=\frac{n!}{(n-k)!}
=\binom nk k!\,.  
\end{equation}

\section{Invariants of binary forms}\label{Sbinary}

We begin by collecting some definitions and general results. See \eg{}
\citet{Elliott}, \citet{Hilbert}, \citet{KR} and 
\citet{Schur} for further details. (In particular, see \cite{KR} for the
\emph{umbral calculus}, which is a useful method to describe and study
invariants and covariants, but which will not be used here.)

\begin{warning}
Note that the notation in these and many other references
is different, since the
forms there are written as $\sumin \binom {\nn}i a_ix^{\nn-i}y^i$ instead of
\eqref{f} below; \ie, 
$a\scitex{i}{Elliott}=a\scitex{i}{Hilbert}=a\scitex{i}{Schur}=\ha_i$, where
\begin{equation}\label{ha}
\ha_i\=
\frac{a_i}{\binom{\nn}i}.  
\end{equation}
The variables $\ha_i$ are often more convenient for theoretical purposes,
see \eg{} \refE{EHankel} below and \cite[Satz 2.18]{Schur} or \cite{KR}, and
they are generally used in standard treatments,
but for our purposes
we prefer our $a_i$, and will only rarely use $\ha_i$.
\end{warning}

\begin{remark}
  The definitions in this section extend to forms in any number $n\ge2$
  variables, but we will only consider the binary case. See 
\cite{Elliott}, \cite{Hilbert} and
\cite[I]{Schur}.
\end{remark}

A homogeneous binary form of degree (order) $\nn$ can be written as
\begin{equation}\label{f}
  f(\xx)=f(x,y)=\sumin a_i x^{\nn-i}y^i.
\end{equation}
We write $\xx\=(x,y)$ 
and $\ax\=(a_0,\dots,a_{\nn})$. (We regard these
as row vectors.) 
We sometimes use instead the notation
$(x_1,x_2)=(x,y)$. 
We further write $\pdx=\pd_1=\pd/\pd x=\pd/\pd x_1$ and 
$\pdy=\pd_2=\pd/\pd y=\pd/\pd x_2$, and note that, for $0\le i\le \nn$,
\begin{equation}\label{adiff}
  a_i=a_i(f)=\frac1{(\nn-i)!}\pdx^{\nn-i}f(0,1)
=\frac1{(\nn-i)!\,i!}\pdx^{\nn-i}\pdy^i f
,
\end{equation}
and thus 
\begin{equation}
  \label{hadiff}
\ha_i=\frac1{n!}\pdx^{\nn-i}\pdy^i f.
\end{equation}
We use occasionally subscripts $\dg \nn$ 
to denote the degree of the
considered forms or polynomials; for example $a\dgxx{i}{n}$.

A $2\times 2$ matrix $T=\smatrixx{\ga&\gb\\\gam&\gd}$ acts on the variables
(to the right) by $\xx'=\xx T$ and on forms (to the left) by
\begin{equation}\label{Tf}
  Tf(\xx)\=f(\xx T) = f(\ga x+\gam y,\gb x+\gd y).
\end{equation}
This gives an action of the general linear group $GL(2)$ on the set of all
binary forms of degree $\nn$.

\begin{definition}\label{Dinv}
 A (projective) \emph{invariant} (of binary forms of a given degree $\nn$) is a
 homogeneous polynomial $\Phi(f)$ in the coefficients $\ax$ such that 
  \begin{equation}\label{inv}
	\Phi(Tf)=|T|^w\Phi(f)
  \end{equation}
for some number $w$ and all $f$ and $T\in GL(2)$. The number $w$ is the
\emph{weight} (or \emph{index}) 
of $\Phi$. We denote the \emph{degree} of $\Phi$ by $\nu$.
(We generally use $\nu$ for the degree and $w$ for the weight, sometimes
without comment; 
similarly we later use $\mm$ for the order of covariants and seminvariants.
There are no standard notations; some examples of other notations are
$i\scite{Elliott}=i\scite{Glenn}=g\scite{Hilbert}=r\scite{Schur}=\nu$ 
for the degree
and $w\scite{Elliott}=k\scite{Glenn}=p\scite{Hilbert}=p\scite{Schur}=w$ 
for the weight.
Further $p\scite{Elliott}=m\scite{Glenn}=n\scite{Hilbert}=k\scite{Schur}=\nn$ 
for the degree of the form
and
$\go\scite{Glenn}=m\scite{Schur}=\mm$ 
for the order, see below.)
  \end{definition}

The weight $w$ is necessarily an integer.
Taking $T=\gl I$, which gives $Tf=\gl^\nn f$, we see that 
\begin{equation}
  \label{weight}
\nn\nu=2w.
\end{equation}
Hence $w\ge0$, and $w>0$ except in the trivial case $\nu=0$ when the
invariant is a constant.

\begin{remark}
If $\Phi$ satisfies the more general equation $\Phi(Tf)=c_T\Phi(f)$ for some
collection of numbers $c_T$, then necessarily $c_T=|T|^w$ for some $w$, so
$\Phi$ is an invariant as defined above. Similarly, in definitions 
below, we may equivalently allow arbitrary factors $c_T$
in \eqref{jinv}, \eqref{cinv}, \eqref{jcinv}, \eqref{sinv}, \eqref{jsinv};
these necessarily 
have to have the given form $|T|^w$ or $\ga^{\mm}|T|^w$ for some $w$ and $\mm$.
\end{remark}

\begin{remark}
  The identity \eqref{inv} 
is a polynomial identity in the entries of $T$, and thus it
extends to all $2\times 2$ matrices $T$, also
  singular.  Thus  $\Phi(Tf)=0$ whenever $T$ is singular, except in the
  trivial case $\nu=w=0$ when $\Phi$ is a constant.
The same applies to similar formulas below.
\end{remark}

\begin{definition}
Similarly, a \emph{joint invariant} of several forms
$f_1,\dots,f_\ell$, 
of degrees $\nn_1,\dots,\nn_\ell$, is a polynomial in the coefficients of 
$f_1,\dots,f_\ell$, homogeneous of degrees $\nu_1,\dots,\nu_\ell$, respectively,
such that
 \begin{equation}\label{jinv}
  \Phi(Tf_1,\dots,Tf_\ell)=|T|^w\Phi(f_1,\dots,f_\ell)
\end{equation}   
for some $w$, the \emph{weight} of $\Phi$,
and all $f_1,\dots,f_\ell$ and $T\in GL(2)$. 
\end{definition}

In this case we have
\begin{equation}  \label{jweight}
  \nn_1\nu_1+\dots+ \nn_\ell \nu_\ell=2w.
\end{equation}
Again $w$ is an integer with $w\ge0$, and $w>0$ except in the trivial case
of a constant invariant.

\begin{remark}
  The assumption that $\Phi$ is homogeneous separately in the coefficients
  of each $f_j$ is no real restriction, since any invariant polynomial $Q$
  can be decomposed into homogeneous components which are invariant.
\end{remark}

\begin{example}
  \label{Eapolar}
The  \emph{apolar invariant} 
of two binary forms $f(x,y)=\sumin a_i x^{\nn-i}y^i$ and
$g(x,y)=\sumin b_i x^{\nn-i}y^i$ of the same degree $\nn$ is 
\begin{equation}\label{apolar}
  \begin{split}
  A(f,g)&\=\sumin(-1)^ii!\,(\nn-i)!\,a_ib_{\nn-i}
={\nn!}\sumin(-1)^i\frac{a_ib_{\nn-i}}{\binom {\nn}i}
\\&\phantom:
={\nn!}\sumin(-1)^i{\binom {\nn}i}{\ha_i\hb_{\nn-i}}
\\&\phantom:
=f(\pdy,-\pdx)g(x,y)
=g(-\pdy,\pdx)f(x,y)
.	
  \end{split}
\end{equation}
This is a joint invariant of $f$ and $g$ of degrees $\nu_1=\nu_2=1$ and weight
$\nn$.
(Our definition differs from \cite{Schur} by a factor $\nn!$:
$A\scite{Schur}(f,g)=A(f,g)/\nn!$.)
The apolar invariant is also called \emph{transvectant}, see \refE{Etrans}
below. 
Using \eqref{adiff}, we also have
\begin{equation}
  \label{apolardiff}
A(f,g)=\sumin(-1)^i(\pdx^{\nn-i}f\cdot\pdx^i g)(0,1).
\end{equation}
Note that $A(g,f)=(-1)^\nn A(f,g)$; hence the apolar invariant is symmetric in $f$ and
$g$ if $\nn$ is even, and antisymmetric if $\nn$ is odd.

The apolar invariant is 
the only joint invariant 
with degrees $\nu_1=\nu_2=1$
of two binary forms of the same degree, and there are no such invariants of
binary forms of different degrees \cite[Satz 2.6]{Schur}.

See \refE{Eapolarcov} for a generalization.
\end{example}

\begin{example}\label{Eapolar1}
Taking $f=g$ in \refE{Eapolar} we obtain
  the \emph{apolar invariant} 
(or \emph{transvectant}, see \refE{Etrans})
of a single binary form
  \begin{equation}
A(f,f)\=\sumin(-1)^ii!\,(\nn-i)!\,a_ia_{\nn-i}; 	
  \end{equation}
this is an invariant of
  degree $\nu=2$ and weight $w=\nn$, for any even $\nn$.
(Note that $A(f,f)=0$
  when $\nn$ is odd.)

In fact this is 
the only invariant of degree 2;
if $\nn$ is odd there is thus no such invariant \cite[Satz 2.5]{Schur}.
\end{example}

\begin{example}
  \label{EHankel}
If $\nn=2q$ is even, then the \emph{Hankel determinant}
\begin{equation}\label{hankel}
  \Han(f)=\bigabs{\ha_{i+j}}_{i,j=0}^q,
\qquad \text{with}\quad \ha_l=\xfrac{a_l}{\textstyle\binom {\nn}l},
\end{equation}
is an 
invariant of degree $\nu=q+1=\nn/2+1$ and, by \eqref{weight},
weight $w=q(q+1)$.
The Hankel determinant is also called the \emph{catalecticant}. 
\end{example}

\subsection{Covariants}

\begin{definition}
More generally, a (projective) \emph{covariant} 
is a polynomial $\Psi(f;\xx)=\Psi(\ax;\xx)$ in 
$\xx$ and the coefficients $\ax$ of $f$
such that
\begin{romenumerate}
\item 
$\Psi$ is homogeneous in $\ax$ of some degree $\nu$, the \emph{degree} of $\Psi$;
\item 
$\Psi$ is homogeneous in $\xx$ of some degree $\mm$, the \emph{order} of $\Psi$;
\item 
  \begin{equation}\label{cinv}
\Psi(Tf;\xx T\qw)=|T|^w \Psi(f;\xx)	
  \end{equation}
for some integer $w$, the \emph{weight} of $\Psi$,
and all forms $f$ (of degree $\nn$) and all $T\in GL(2)$.
\end{romenumerate}  
\end{definition}

Hence, an invariant is a covariant of order 0. 

The relation \eqref{weight}
generalizes to 
\begin{equation}  \label{cweight}
  \nn\nu=m+2w.
\end{equation}

\begin{definition}
Similarly, 
a  \emph{joint covariant} of forms $f_1,\dots,f_\ell$ of degrees
$\nn_1,\dots,\nn_\ell$ 
is a polynomial $\Psi(f_1,\dots,f_\ell;\xx)=\Psi(\ax_1,\dots,\ax_\ell;\xx)$
in the coefficients $\ax_j$ of $f_j$, $j=1,\dots,\ell$,
that is homogeneous in each $\ax_j$ of degree $\nu_j$,
homogeneous in $\xx$ of degree $\mm$, the \emph{order} of $\Psi$,
and such that
  \begin{equation}\label{jcinv}
\Psi(Tf_1,\dots,Tf_\ell;\xx T\qw)=|T|^w \Psi(f_1,\dots,f_\ell;\xx)	
  \end{equation}
for some integer $w$, the \emph{weight} of $\Psi$, and all forms
$f_1,\dots,f_\ell$ and all $T\in GL(2)$.  
\end{definition}

We now have
\begin{equation}\label{jcweight}
  \nn_1\nu_1+\dots+\nn_\ell \nu_\ell=m+2w.
\end{equation}

\begin{example}\label{Ef}
  The form $f(\xx)$ itself is a covariant of degree $\nu=1$, order $\mm=\nn$ and
  weight $w=0$.
\end{example}

\begin{example}\label{EHessian}
  The \emph{Hessian covariant}
  \begin{equation}
H(f)=
H(f;\xx)\=\lrabs{\frac{\partial^2 f}{\partial x_i\partial x_j}}_{1\le i,j\le 2}
  \end{equation}
is a covariant of degree $\nu=2$, order $\mm=2(\nn-2)$ and weight $w=2$.
(Other notation:
$\cH\scite{Hilbert}=(n(n-1))\qww H$.)
\end{example}

\begin{example}\label{EJacobian}
  The \emph{Jacobian determinant}
  \begin{equation}
	J(f_1,f_2)
=	J(f_1,f_2;\xx)
\=\lrabs{\frac{\partial f_i}{\partial x_j}}_{1\le i,j\le 2}
  \end{equation}
is a joint covariant of degrees $\nu_1=\nu_2=1$, order $\nn_1+\nn_2-2$ and
weight $w=1$.
Note that $J$ is antisymmetric; $J(g,f)=-J(f,g)$, and $J(f,f)=0$.
\end{example}

\begin{example}
  \label{Etrans}
The \emph{$k$:th transvectant} 
$\tv{f,g}_k$ is a joint covariant of two forms $f$ and $g$ of
arbitrary degrees $\nn_1$ and $\nn_2$, defined by
\begin{equation}\label{trans}
  \begin{split}
  \tv{f,g}_k&=
\Bigpar{
\frac{\pd}{\pd x_1}\frac{\pd}{\pd y_2}-\frac{\pd}{\pd x_2}\frac{\pd}{\pd y_1}
}^kf(\xx)g(\yy)\Big|_{\yy=\xx}
\\&
=\sum_{i=0}^k(-1)^i\binom ki 
\frac{\pd^k f}{\pd x_1^{k-i}\pd x_2^i}
\frac{\pd^k g}{\pd x_1^{i}\pd x_2^{k-i}}
.	
  \end{split}
\end{equation}
Here $k\ge0$ is an arbitrary positive integer, but it is easy to see that
$\tv{f,g}_k=0$ unless $k\le\min(\nn_1,\nn_2)$. (Trivially
$\tv{f,g}_0=fg$.) 
It is easy to see that $\tv{f,g}_k$ is a joint covariant of degrees
$\nu_1=\nu_2=1$, order $\nn_1+\nn_2-2k$ and weight $w=k$.
(Other notations:  
$\tv{f,g}_k
=(n_1)_k(n_2)_k (f,g)\scitex{k}{highertrans}
=(f,g)^k\scite{Glenn}
=(n_1)_k(n_2)_k (f,g)\scitex{k}{Hilbert}
=(f,g)^{(k)}\scite{Olver}=
[f,g]^k\scite{Salden}$.)

Furthermore, $\tv{f,g}_k=(-1)^k\tv{g,f}_k$, so $\tv{f,g}_k$ is symmetric if
$k$ is even 
and anti-symmetric if $k$ is odd. In particular, $\tv{f,f}_k=0$ for odd $k$,
but for even $k\le\nn$, $\tv{f,f}_k$ is a non-trivial
covariant of degree $2$, order $2\nn-2k$ and weight $k$.
(Other notations:  
$f\scitex{k}{Hilbert}\=\frac12(f,f)\scitex{k}{Hilbert}
=\frac12(n)_k^{-2}\tv{f,f}_k$.)

The first transvectant is the Jacobian covariant in \refE{EJacobian}:
\begin{equation}
  \tv{f,g}_1=J(f,g).
\end{equation}

The second transvectant $\tv{f,f}_2$ is (twice) the Hessian covariant
in \refE{EHessian}: 
\begin{equation}
  \tv{f,f}_2=2H(f).
\end{equation}

Furthermore, in the case $n_1=n_2=n=k$, $\tv{f,g}_n$ is of order 0, \ie, an
invariant. In this case, by a binomial expansion in \eqref{trans}, 
\eqref{hadiff} and \eqref{apolar},
\begin{equation}\label{trans=a}
  \begin{split}
\tv{f,g}_n
&=	
\sumin \binom ni(-1)^i
\Bigpar{
\frac{\pd}{\pd x_1}\frac{\pd}{\pd y_2}}^{n-i}
\Bigpar{\frac{\pd}{\pd x_2}\frac{\pd}{\pd y_1}}^{i}
 f(\xx)g(\yy)
\\&=	
\sumin \binom ni(-1)^i
n!\,\ha_i\,n!\,\hb_{n-i}
\\&=
n!\, A(f,g).
  \end{split}
\end{equation}
Hence the apolar invariant equals (apart from a factor $1/n!$) the $n$th
transvectant 
$\tv{f,g}_n$. 

For relations between transvectants, and interpretations in terms of
representations of $SL_2$, see \citet{highertrans}.
\end{example}

\begin{example}
  \label{Eapolarcov}
As shown in \eqref{trans=a}, the transvectant $\tv{f,g}_n$ of two binary
forms of equal degree $n$ is (apart from a constant factor) their apolar
invariant.
More generally, if $f$ and $g$ are binary forms of degrees $n$ and $m$ with
$n\ge m\ge0$,  the \emph{apolar covariant} 
$\tv{f,g}$ is defined as the highest non-trivial transvectant (\ie, the
$m$th transvectant), which by \eqref{trans} and a short calculation
can be expressed as
\begin{equation}
\tv{f,g}
  \=
\tv{f,g}_m
=m!\,g(-\pd_2,\pd_1)f(x_1,x_2).
\end{equation}
(Hence, if $m=n$, $\tv{f,g}=n!\,A(f,g)$, so the apolar covariant then
reduces to the apolar invariant in \refE{Eapolar}, except for the trivial but
inconvenient factor $n!$.)

By \refE{Etrans}, $\tv{f,g}$ is a joint covariant of degrees
$\nu_1=\nu_2=1$, order $\mm=n-m$ and weight $w=m$.

Note the asymmetry in the definition; we assume $n\ge m$.
\end{example}

\begin{example}
  \label{EGund}
The \emph{$k$th Gundelfinger covariant} $G_k(f)$, for $k=0,1,\dots$, is the
$(k+1)\times(k+1)$ determinant
\begin{equation}\label{gund}
  G_k(f)
\=
\lrabs{\frac{\pd^{2k}f(x,y)}{\pd x^{2k-i-j}\pd y^{i+j}}}_{0\le i,j\le k};
\end{equation}
this is a covariant of degree $\nu=k+1$, order $\mm=(k+1)(n-2k)$ and weight
$w=k(k+1)$,
see \cite{Gund} and \cite{Kung}.
Note that $G_0(f)=f$ and $G_1(f)=H(f)$, the Hessian covariant; further
$G_k(f)=0$ if $k>n/2$. If $n$ is even and $k=n/2$, then, by \eqref{hadiff}
and \eqref{hankel}, 
\begin{equation}\label{gund=hankel}
  G_{n/2}(f)
=\bigabs{n!\,\ha_{i+j}}_{i,j=0}^{n/2}=n!^{n/2+1}\Han(f),
\end{equation}
a constant times the Hankel determinant (catalecticant) in \refE{EHankel}.
\end{example}

\smallskip

A  covariant $\Psi$ of order $\mm$ can be written
$\Psi(f;\xx)=\Phi(f)x_1^{\mm}+\dots$; we call the coefficient $\Phi(f)$ the
\emph{source} or \emph{leading coefficient} of $\Psi$. (And similarly for
joint covariants.) 
The source of $\Psi$ is thus given by 
\begin{equation}\label{co2semi}
  \Phi(f)\=\Psi(f;1,0);
\end{equation}
equivalently,
\begin{equation}\label{co2semidiff}
  \Phi(f)\=\tfrac1{\mm!}\pd_1^{\mm}\Psi(f;\xx).
\end{equation}
Conversely, by \eqref{cinv}, $\Psi$ can be recovered by
\begin{equation}\label{semi2co}
  \Psi(f;x,y)=
x^{-2w}\Phi\bigpar{T_{x,y}^{(1)}f}
=x^{-w}\Phi\bigpar{T_{x,y}^{(2)}f}
=x^{\mm}\Phi\bigpar{T_{x,y}^{(3)}f},
\end{equation}
where
\begin{align}
T_{x,y}^{(1)}&\= \matrixx{x&y\\0&x},&
T_{x,y}^{(2)}&\= \matrixx{x&y\\0&1},&
T_{x,y}^{(3)}&\= \matrixx{1&y/x\\0&1}.
\end{align}

\subsection{Seminvariants}

\begin{definition}
  A  \emph{seminvariant} (of binary forms of degree $\nn$) is a
  homogeneous polynomial $\Phi(f)$ in the coefficients $\ax$ such that 
  \begin{equation}\label{sinv}
	\Phi(Tf)=\ga^{\mm}|T|^w\Phi(f)
  \end{equation}
for some $\mm,w\ge 0$ and all $f$ and $T$ of the form
$\smatrixx{\ga&0\\\gam&\gd}$. The number $\mm$ is the \emph{order} and $w$ is the 
\emph{weight} of $\Phi$. We denote the \emph{degree} of $\Phi$ by $\nu$.
  \end{definition}

\begin{definition}
Similarly, a \emph{joint seminvariant} of several forms
$f_1,\dots,f_\ell$, 
of degrees $\nn_1,\dots,\nn_\ell$, is a polynomial $\Phi (f_1,\dots,f_\ell)$
in the coefficients of 
$f_1,\dots,f_\ell$, homogeneous of degrees $\nu_1,\dots,\nu_\ell$, respectively,
such that
 \begin{equation}\label{jsinv}
  \Phi(Tf_1,\dots,Tf_\ell)=\ga^{\mm}|T|^w\Phi(f_1,\dots,f_\ell).
\end{equation}   
for some $\mm,w\ge 0$, the \emph{order} and \emph{weight} of $\Phi$,
and all $f_1,\dots,f_\ell$ and 
$T=\smatrixx{\ga&0\\\gam&\gd}$. 
\end{definition}

In other words, a (joint) seminvariant is an invariant for the subgroup of
$GL(2)$ given by
$\lrcpar{\smatrixx{\ga&\gb\\\gam&\gd}\in GL(2):\gb= 0}=
\Bigcpar{\smatrixx{\ga&0\\\gam&\gd}:\ga\gd\neq 0}$.

We still have \eqref{cweight} and \eqref{jcweight}, respectively. In fact,
these are equivalent to invariance for all 
$T=\gl I=\smatrixx{\gl&0\\ 0&\gl}$. 
Consequently, if \eqref{cweight} or \eqref{jcweight} holds, it is
enough that \eqref{sinv} or \eqref{jsinv}
holds for $T$ of the form $\smatrixx{\ga&0\\\gam&1}$; these are
the transformations $(x,y)\mapsto(\ga x+\gam y,y)$, which form a group $A(1)$ 
obviously isomorphic to the group of affine maps $x\mapsto \ga x+\gam$ in one
dimension. 

Furthermore,we say that a coefficient $a_i$ has \emph{weight} $i$, and more
generally 
that a monomial $a_0^{k_0}a_1^{k_1}a_2^{k_2}\dotsm$ has weight
$k_1+2k_2+\cdots$. A polynomial in $\ax=(a_0,\dots,a_{\nn})$ is \emph{isobaric}
if all its terms has the same weight, and then this is said to be the weight
of the polynomial. It is easily seen that the invariance \eqref{sinv} or
\eqref{jsinv} holds for all $T$ of the form 
$\smatrixx{1&0\\0&\gd}$ if and only if $\Phi$ is isobaric of weight $w$.
(In this case, \eqref{sinv}=\eqref{inv} and \eqref{jsinv}=\eqref{jinv}.) 
Consequently, the invariance \eqref{sinv} or \eqref{jsinv} holds for all
diagonal matrices $T$ if and only if $\Phi$ is homogeneous and isobaric 
and \eqref{cweight} or \eqref{jcweight} holds.
This leads to the following characterization, see
\cite[\S II.2]{Schur}.

\begin{theorem}\label{Tsemi}
The following are equivalent for a polynomial $\Phi$ in the coefficients of
one or several binary forms.
\begin{romenumerate}
\item 
$\Phi$ is a (joint) seminvariant   
\item 
$\Phi$ is homogeneous and invariant for $A(1)$.
\item 
$\Phi$ is homogeneous and isobaric and invariant for all $T$ of the form
  $\smatrixx{1&0\\t&1}$, \ie, translations $(x,y)\mapsto(x+t,y)$.
(For such $T$, the invariance is simply $\Phi(Tf)=\Phi(f)$ or
  $\Phi(Tf_1,\dots,Tf_\ell) =\Phi(f_1,\dots,f_\ell) $.)
\item 
$\Phi$ is homogeneous and isobaric and satisfies the 
\emph{Cayley--Aronhold differential equation}
\begin{equation}\label{cayley}
\gO(\Phi)\=
  \sum_{i=1}^\nn (\nn-i+1) a_{i-1} \frac{\partial\Phi}{\partial a_i}
=0;
\end{equation}
for a joint seminvariant $\Phi(\ax_1,\dots,\ax_\ell)$, 
with $\ax_j=(a_{1,j},\dots,a_{\nn_j,j})$,
the equation takes
the form
\begin{equation}\label{jcayley}
 \gO(\Phi)\=
\sum_{j=1}^\ell  \sum_{i=1}^{\nn_j} (\nn_j-i+1) a_{i-1,j} 
\frac{\partial\Phi}{\partial a_{i,j}}
=0,
\end{equation}
where thus $\gO=\sum_{j=1}^\ell \gO_j$.
\end{romenumerate}
\end{theorem}

In (iii), it suffices to consider the special $T=\smatrixx{1&0\\1&1}$, \ie,
$(x,y)\mapsto(x+y,y)$. 

\begin{remark}
Using $\ha_i\=a_i/\binom {\nn}i$ as in \cite{Schur},
\eqref{cayley} becomes 
\begin{equation}
\gO(\Phi)=
  \sum_{i=1}^{\nn} i \ha_{i-1} \frac{\partial\Phi}{\partial \ha_i}
=0.  
  \end{equation}
\end{remark}

\begin{remark}
  There is also a dual differential operator
\begin{equation}\label{gox}
\gOx(\Phi)\=
  \sum_{i=0}^{\nn-1} (i+1) a_{i+1} \frac{\partial\Phi}{\partial a_i}
=  \sum_{i=0}^{\nn-1} (\nn-i) \ha_{i+1} \frac{\partial\Phi}{\partial \ha_i},
\end{equation}
and similarly for joint seminvariants with $\gOx\=\sum_j\gOx_j$.
The differential equation
\begin{equation}\label{cayleyx}
  \gOx(\Phi)=0 
\end{equation}
holds for \emph{invariants} $\Phi$, but not for other seminvariants.
In fact, \eqref{cayleyx} is a necessary and sufficient condition for a
seminvariant $\Phi$ to be an invariant
\cite[\Satze 2.1--2.2]{Schur}.
(Other notations: $\cD\scite{Schur}=\gO$, $\gD\scite{Schur}=\gOx$.)
\end{remark}

Obviously, an invariant is a seminvariant. 
Moreover, there is an important correspondence between covariants and
seminvariants.

\begin{theorem}\label{Tcov}
  For any $\nn$, there is a one-to-one correspondence between
  covariants $\Psi$ and seminvariants $\Phi$, such that $\Phi$ is the
  source 
of $\Psi$; see \eqref{co2semi}--\eqref{semi2co}.

More generally,
for any $\nn_1,\dots,\nn_\ell$, there is a one-to-one correspondence between
joint  covariants and joint seminvariants given by taking the source (leading
coefficient). 

The degrees, order and weight are preserved by this correspondence.
\end{theorem}

\begin{remark}
  Another way to recover the covariant $\Psi$
from its source $\Phi$ is by the formula \cite[pp. 56--58]{Schur}
  \begin{equation}
\Psi=
\sum_{j=0}^{\mm} \frac{(\gOx)^j(\Phi)}{j!} x^{\mm-j},
  \end{equation}
where $\mm$ is the order. Since further $(\gOx)^{\mm+1}\Phi=0$, the sum can
also be written $\sumj (\gOx)^j	(\Phi)\, x^{\mm-j}/j!$\,.
\end{remark}

\begin{remark}
  We have defined the weight of a covariant so that it equals the weight of
  its source. 
It is easy to see, arguing as for \refT{Tsemi}, that if we give $x$ weight 1
and $y$ weight 0, then a covariant is isobaric, with each term of weight
$w+\mm$.
(Some references, \eg{} \cite{Glenn}, call our $w$ the \emph{index} of the
covariant, and call $w+\mm$ the \emph{weight}, but we do not make this
definition.
Note that if we instead give $x$ weight 0 and $y$ weight $-1$, then the
covariant is isobaric with weight $w$.)
\end{remark}

\begin{example}
  \label{Efs}
The source of $f(\xx)$, \ie, 
the seminvariant corresponding to $f(\xx)$, see \refE{Ef}, is $a_0$. 
This has degree 1, order $\nn$, weight 0.
\end{example}

\begin{example}
  \label{EHessians}
The Hessian seminvariant $H_0$ is the source of the Hessian covariant $H$ in
\refE{EHessian}. It 
is, by a simple calculation,
\begin{equation}\label{hessians}
H_0(f)\=
  2\nn(\nn-1)a_0a_2-(\nn-1)^2a_1^2
=n^2(n-1)^2(\ha_0\,\ha_2-\ha_1^2).
\end{equation}
$H_0$ has degree 2, order $2\nn-4$ and weight 2.
\end{example}

\begin{example}
  \label{EJacobians}
The Jacobian joint seminvariant 
of two binary forms 
$f(x)=\sumini a_i x^{\nn_1-i}y^i$ and
$g(x)=\suminii b_i x^{\nn_2-i}y^i$, 
corresponding to the Jacobian joint covariant in
\refE{EJacobian}, 
is, by a simple calculation,
\begin{equation}
\nn_1a_0b_1-\nn_2b_0a_1
=\nn_1\nn_2(\ha_0\hb_1-\hb_0\ha_1).
\end{equation}
This has degrees $\nu_1=\nu_2=1$, order $\nn_1+\nn_2-2$ and weight 1.
\end{example}

\begin{example}\label{EGunds}
  The source $g_k$ of the $k$th Gundelfinger covariant in \refE{EGund}
is, by \eqref{gund},
  \eqref{co2semi} 
  and  \eqref{hadiff},
\begin{equation}\label{gunds}
  \begin{split}
  g_k(f)
&=G_k(f;1,0)
=\lrabs{\frac{\pd^{2k}f}{\pd x^{2k-i-j}\pd y^{i+j}}(1,0)}_{0\le i,j\le k}
\\&
=\lrabs{\frac{1}{(n-2k)!}\frac{\pd^{n}f}{\pd x^{n-i-j}\pd y^{i+j}}}_
 {0\le i,j\le k}
\\&
=\bigpar{(n)_{2k}}^{k+1}\bigabs{\ha_{i+j}}_{i,j=0}^k.
  \end{split}
\end{equation}
This is a seminvariant of degree $\nu=k+1$, order $\mm=(k+1)(n-2k)$ and weight
$w=k(k+1)$ by \refE{EGund}.

Cf.\ the special cases in \refE{EHessians} ($k=1$) and \refE{EHankel} ($n=2k$).
\end{example}

\begin{example}
  \label{Etranss}
The source $\tau_k(f,g)$ of the transvectant $\tv{f,g}_k$ is 
given by,
using \eqref{co2semidiff}, \eqref{trans}, \eqref{trans=a},
\eqref{apolardiff}, \eqref{adiff} and \eqref{ha},
\begin{equation}\label{etranss}
  \begin{split}
\tau_k&(f,g)= \frac{1}{(\nn_1+\nn_2-2k)!}\pd_1^{\nn_1+\nn_2-2k}
\Bigpar{
\Bigpar{
\frac{\pd}{\pd x_1}\frac{\pd}{\pd y_2}-\frac{\pd}{\pd x_2}\frac{\pd}{\pd y_1}}^k
 f(\xx)g(\yy)\Big|_{\yy=\xx}}
\\&
= \frac{1}{(\nn_1+\nn_2-2k)!}
\Bigpar{
\Bigpar{\frac{\pd}{\pd x_1}+\frac{\pd}{\pd y_1}}^{\nn_1+\nn_2-2k}
\Bigpar{
\frac{\pd}{\pd x_1}\frac{\pd}{\pd y_2}-\frac{\pd}{\pd x_2}\frac{\pd}{\pd y_1}}^k
f(\xx)g(\yy)\Big|_{\yy=\xx}}
\\&
= \frac{1}{(\nn_1+\nn_2-2k)!}
\Bigpar{
\frac{\pd}{\pd x_1}\frac{\pd}{\pd y_2}-\frac{\pd}{\pd x_2}\frac{\pd}{\pd y_1}}^k
\Bigpar{\frac{\pd}{\pd x_1}+\frac{\pd}{\pd y_1}}^{\nn_1+\nn_2-2k}
f(\xx)g(\yy)\Big|_{\yy=\xx}
\\&
= \frac{1}{(\nn_1-k)!\,(\nn_2-k)!}
\Bigpar{
\frac{\pd}{\pd x_1}\frac{\pd}{\pd y_2}-\frac{\pd}{\pd x_2}\frac{\pd}{\pd y_1}}^k
\pd_1^{\nn_1-k}f(\xx)\pd_1^{\nn_2-k}g(\yy)\Big|_{\yy=\xx}
\\&
= \frac{1}{(\nn_1-k)!\,(\nn_2-k)!}
\bigtv{
\pd_1^{\nn_1-k}f,\,\pd_1^{\nn_2-k}g}_k
\\&
= \frac{k!}{(\nn_1-k)!\,(\nn_2-k)!}
A\dgx{k}
\bigpar{
\pd_1^{\nn_1-k}f,\,\pd_1^{\nn_2-k}g}
\\&
= \frac{k!}{(\nn_1-k)!\,(\nn_2-k)!}
\sum_{i=0}^k(-1)^i
\bigpar{\pd_1^{\nn_1-i}f\cdot\pd_1^{\nn_2-k+i}g}(0,1)
\\&
= \frac{k!}{(\nn_1-k)!\,(\nn_2-k)!}
\sum_{i=0}^k(-1)^i
(\nn_1-i)!\,a_i\,(\nn_2-k+i)!\,b_{k-i}
\\&
= \frac{n_1!\,n_2!}{(\nn_1-k)!\,(\nn_2-k)!}
\sum_{i=0}^k(-1)^i\binom{k}{i}
\ha_i\,\hb_{k-i}.
  \end{split}
\raisetag{ 1.5\baselineskip}
\end{equation}

Note from \refE{Etrans} 
that $\tau_1(f,g)$ is the Jacobian seminvariant in \refE{EJacobians},
and $\tau_2(f,f)=2\,H_0(f)$, the Hessian seminvariant in \refE{EHessians}, 
while if $n_1=n_2=n$,  then
$\tau_n(f,g)=\tv{f,g}_n$ is $n!\,A(f,g)$, the apolar invariant in
\refE{Eapolar}. Further, the special case $n_2=k$ yields the source of the
apolar covariant in \refE{Eapolarcov}.
\end{example}

The group $GL(2)$ is generated by the subgroup $A(1)$ and the reflection
$\rho=\smatrixx{0&1\\1&0}$ which interchanges $x$ and $y$. Hence, $\Phi$ is
invariant if and only it is invariant under both $A(1)$ and $\rho$, \ie, 
if and only if it is a seminvariant that is invariant under $\rho$.
We have
$\rho(x,y)=(y,x)$ and thus, by \eqref{Tf}, $\rho f(x,y)=f(y,x)$. 
We denote $\rho f$ by $f\refl$.
It follows
from \eqref{f} that if $f$ has coefficients $\ax=(a_0,\dots,a_{\nn})$ as in
\eqref{f}, then $f\refl$ has coefficients 
\begin{equation}
  \label{arefl}
 \ax\refl\=(a_{\nn},\dots,a_0).
\end{equation}
This leads to the following companion to \refT{Tsemi}, which can be used
together with \refT{Tsemi} to find convenient criteria for invariants.

\begin{theorem}\label{Tinv}
The following are equivalent for a polynomial $\Phi$ in the coefficients of
one or several binary forms.
\begin{romenumerate}
\item 
$\Phi$ is a (joint) invariant.   
\item 
$\Phi$ is a (joint) seminvariant and $\Phi(\ax\refl)=(-1)^w\Phi(\ax)$
or
$\Phi(\ax_1\refl,\dots,\ax_\ell\refl)\allowbreak=(-1)^w\Phi(\ax_1,\dots,\ax_\ell)$.
\item 
$\Phi$ is a (joint) seminvariant of order $\mm=0$.
\item 
$\Phi$ is a (joint) seminvariant and $\nn\nu=2w$ or $\nn_1\nu_1+\dots+\nn_\ell \nu_\ell=2w$
for the degree(s) and the weight (\ie, \eqref{weight} or \eqref{jweight}
holds). 
\end{romenumerate}
\end{theorem}
\begin{proof}
  (i)$\iff$(ii) by the discussion above. 

 (i)$\iff$(iii) by the correspondence in \refT{Tcov} and the fact that 
an invariant is a   covariant of order 0 and conversely.

(iii)$\iff$(iv) by \eqref{cweight} and \eqref{jcweight}.
\end{proof}

It is obvious that we can take linear combinations of
(joint) invariants, covariants or seminvariants with the same degrees,
weights and orders. Furthermore, a product of 
(joint) invariants, covariants or seminvariants 
is always another invariant, covariant or seminvariant, with
degrees, weights and orders in the factors added. 
Consequently, an isobaric polynomial in 
invariants is another invariant; the same is
true for covariants and seminvariants provided the result also is
homogeneous in the coefficients $a_0,\dots,a_{\nn}$.

We say that a set $\cB$ of invariants (\etc{}) is a \emph{basis} if every
invariant 
(\etc{}), of forms of the given degree(s), is a (necessarily isobaric)
polynomial in elements of
$\cB$. (Less formally, one also says that the invariants in $\cB$ are all
invariants, thus really meaning that every invariant is a polynomial of
invariants in $\cB$.)
It is proved by Gordan (and more generally by Hilbert), that for any $n$,
there exists a finite basis of the invariants (covariants or 
seminvariants).

We have also the following.

\begin{theorem}\label{Tcov2}
  A covariant of a sequence of covariants $\Psi_1(\ax_1,\dots,\ax_l;\xx)$,
$\Psi_2(\ax_1,\dots,\ax_l;\xx)$, \dots{} is itself a covariant.
\end{theorem}

\begin{example}\label{EAf2}
  As said in \refE{Ef}, the form $f(\xx)$ itself is a covariant of degree
  $1$, order $\nn$ and weight $0$. Thus $f^2$ is a covariant of degree 2,
  order $2\nn$ and weight 0.
Hence, see \refE{Eapolar1}, the apolar invariant $A(f^2,f^2)$ is an invariant (note
that $2\nn$ is even); it is easily seen that this invariant has
degree 4 and weight $2\nn$, \cf{} \eqref{weight}.
(It is shown in \cite[p.~42]{Schur} that $A(f^2,f^2)$ does not vanish
identically 
for any $\nn\ge2$.)
\end{example}

\begin{example}
  \label{EAHH}
The Hessian covariant $H(f;\xx)$ in \refE{EHessian} has degree $2$ and order
$2\nn-4$; hence the apolar invariant $A(H(f;\xx),H(f;\xx))$ is an invariant of degree 4
and, by \eqref{weight},  weight $2\nn$.
(It is shown in \cite[p.~43]{Schur} that $A(H(f),H(f))$ does not vanish
identically for any $\nn\ge2$.)
\end{example}

\subsection{Rational invariants}\label{SSrational}
By definition, invariants \etc{} are required to be polynomials in the
coefficients. A few times we will consider a minor extension.

\begin{definition}\label{Drational}
  A \emph{rational invariant} (joint invariant, covariant, \etc) is 
a rational function
  of the coefficients that has the invariance property \eqref{inv} (\etc).
\end{definition}

It is easily seen that a rational invariant is the same as a 
  quotient of two invariants (\etc) \cite[Satz 1.4]{Schur}.
Note that a rational invariant may be infinite or undefined for certain
values of the coefficients.

We will in the sequel sometimes use rational seminvariants of the form
$a_0^{-k}\Phi$, where $\Phi$ is a seminvariant.
Another interesting case is the following.
\begin{definition}
An \emph{absolute invariant} is a rational invariant with weight $w=0$;
it is thus a rational function of the coefficients that satisfies
  \begin{equation}\label{absinv}
	\Phi(Tf)=\Phi(f)
  \end{equation}
for all $T\in GL(2)$. 
\end{definition}
  
By \eqref{inv}, there are no non-trivial invariants that are absolute
invariants; we have to consider rational invariants here. Any absolute
invariant is the quotient $\Phi_1/\Phi_2$ of two invariants of the 
same weight, and thus the also the same degree; conversely, any such
quotient is an absolute invariant. One example is given in \refE{E4abs}.

\subsection{Dimensions}

The set of all covariants of degree $\nu$ and weight $w$ of binary forms of
a given degree $n$ is a linear space. We let $N(n,\nu,w)$ be its dimension,
\ie, the number of linearly independent covariants of this degree and
weight. By \refT{Tcov}, $N(n,\nu,w)$ is also the dimension of the linear
space of seminvariants of degree $\nu$ and weight $w$. (\refT{Tcov} yields
an isomorphism between the two linear spaces.)
Note that we get the invariants of degree $\nu$ by taking $w=n\nu/2$
(provided this is an integer), see \eqref{weight} and \eqref{cweight}.

The number $N(n,\nu,w)$ can be computed as follows by a formula by Cayley (the
first complete proof was given  by Sylvester), see 
\cite[\Satze 2.21--2.22]{Schur}. 

Let $\GQ nk$ be the \emph{Gaussian polynomial} defined by
\begin{equation}
  \GQ nk \=
\frac{\prod_{i=1}^n(1-q^i)}{\prod_{i=1}^k(1-q^i) \prod_{j=1}^{n-k}(1-q^j)}
=
\frac{\prod_{i=n-k+1}^n(1-q^i)}{\prod_{i=1}^k(1-q^i)}.
\end{equation}
(See further \eg{} \citet{Andrews}.)
We let $[q^w] P(q)$ denote the coefficient of $q^w$ in a polynomial $P(q)$.

\begin{theorem}
  \label{TGauss}
If\/ $2w\le n\nu$, then
  \begin{equation}
	\begin{split}
N(n,\nu,w)
&=[q^w]\lrpar{(1-q)\GQ{n+\nu}{n}}	
=[q^w]\lrpar{(1-q)\GQ{n+\nu}{\nu}}	
\\&	  
=[q^w]
\frac{\prod_{i=\nu+1}^{\nu+n}(1-q^i)}{\prod_{i=2}^n(1-q^i)}
=[q^w]
\frac{\prod_{i=n+1}^{n+\nu}(1-q^i)}{\prod_{i=2}^\nu(1-q^i)}
\\&
=[q^w]\GQ{n+\nu}{n} - [q^{w-1}]\GQ{n+\nu}{n}.
	\end{split}
  \end{equation}
If $2w> n\nu$, then $N(n,\nu,w)=0$.
\end{theorem}

It follows that if we fix $n$ and $w$, $N(n,\nu,w)$ is the same for all
$\nu\ge w$, and is given by a simple generating function; see also
\refR{Rredall} below.
\begin{corollary}
  \label{CGauss}
If $n\ge 2$ and $\nu\ge w$, then 
\begin{equation}
  N(n,\nu,w)
=[q^w]
{\prod_{i=2}^n(1-q^i)\qw}.
\end{equation}
\end{corollary}
\begin{proof}
  The factors $1-q^i$ with $i\ge \nu+1>w$ do not affect $[q^w]$.
\end{proof}

\section{Invariants of polynomials}\label{Spol}

We may identify the binary form $\tf(x,y)=\sumin a_ix^{\nn-i}y^i$ and the
polynomial $f(x)=\sumin a_ix^{\nn-i}$; this gives a one-to-one
correspondence between binary forms of degree $\nn$ and polynomials of degree
(at most) $\nn$ described by $f(x)=\tf(x,1)$ and, conversely,
$\tf(x,y)=y^\nn f(x/y)$. We let $\cP_\nn$ denote the set of all such
polynomials  $\sumin a_ix^{\nn-i}$.

\begin{remark}\label{Rdegree}
When $\nn$ is given, we will say ``polynomial of degree $\nn$''
for any polynomial $\sumin a_ix^{\nn-i}$, even if $a_0=0$. (As just said, 
this gives a correspondence with binary forms of degree $\nn$.)
This thus includes
polynomials of lower degrees. See further \refSS{SSrestr}.
\end{remark}

A transform $T=\smatrixx{\ga&\gb\\\gam&\gd}$ acts on binary forms by
\eqref{Tf}; this transfer to the action
\begin{equation}\label{actpol}
Tf(x)=(\gb x+\gd)^\nn f\Bigparfrac{\ga x+\gam}{\gb x+\gd}   
\end{equation}
on polynomials.

We define a (projective) invariant or seminvariant of a polynomial $f$ (of some
given degree) as an invariant or seminvariant of the
corresponding binary form $\tf$; similarly, a (projective) covariant is a
polynomial $\Psi(\ax;x)$ of some degree $\mm$ (or less) in $x$ such that the
corresponding 
binary form $\widetilde\Psi(\ax;x,y)$ of degree $\mm$ is a covariant of the form
$\tf$; these definitions extend to joint invariants \etc{} in the obvious way.

Thus, a polynomial $\Phi(f)$ in the coefficients of a polynomial
$f(x)=\sumin a_ix^{\nn-i}$ is an 
invariant if
\begin{equation}
\Phi\Bigpar{ (\gb x+\gd)^\nn f\bigparfrac{\ga x+\gam}{\gb x+\gd}}
=(\ga\gd-\gb\gam)^w\Phi\bigpar{f(x,y)}   
\end{equation}
for all $f$ and $\smatrixx{\ga&\gb\\\gam&\gd}$.

Similarly, by \refT{Tsemi}(ii), a polynomial $\Phi(f)$ in the coefficients
of a polynomial $f(x)=\sumin a_ix^{\nn-i}$ is a (projective) seminvariant if
it is homogeneous and invariant for $A(1)$, \ie{}
\begin{equation}
\Phi\bigpar{ f(\ga x+\gam)}
=\ga^{\mm+w}\Phi\bigpar{f(x,y)}   
\end{equation}
for all $f$ and $(\ga,\gam)$. 
In other words, a seminvariant is the same as
an \emph{affine invariant} for polynomials. 
(However, we continue to use the traditional term seminvariant).

The same applies with obvious modifications to joint invariants and
seminvariants.

\begin{example}
  \label{EHessianp}
If $f$ is a polynomial of degree $\nn$ and $\tf$ the corresponding binary
form, then 
$x\frac{\partial \tf}{\partial x}+y\frac{\partial \tf}{\partial y}=\nn\tf$.
It follows after some calculations that, for $y=1$,
\begin{equation*}
  \begin{vmatrix}
\frac{\partial^2 \tf}{\partial x^2}&\frac{\partial^2 \tf}{\partial x\partial y}
\\
\frac{\partial^2 \tf}{\partial x\partial y}&\frac{\partial^2 \tf}{\partial y^2} 
  \end{vmatrix}
=
  \begin{vmatrix}
\frac{\partial^2 \tf}{\partial x^2} &(\nn-1)\frac{\partial \tf}{\partial x} \\
(\nn-1)\frac{\partial \tf}{\partial x} & \nn(\nn-1)\tf
  \end{vmatrix}
=
\nn(\nn-1)\tf\frac{\partial^2 \tf}{\partial x^2}
-(\nn-1)^2\Bigpar{\frac{\partial \tf}{\partial x}}^2.
\end{equation*}
Hence the Hessian covariant, see \refE{EHessian}, of a polynomial $f$ of
degree $\nn$ is the polynomial
\begin{equation}\label{hessianp}
  H(f;x)\=\nn(\nn-1)f(x)f''(x)-(\nn-1)^2(f'(x))^2.
\end{equation}
The source of $H(f)$ is 
a seminvariant $H_0(f)$ of degree 2, order $2\nn-4$ and weight 2;
by \refE{EHessians}, it is given by
\begin{equation}\label{hessian0p}
H_0(f)=
  2\nn(\nn-1)a_0a_2-(\nn-1)^2a_1^2.
\end{equation}
\end{example}

\begin{example}
  \label{EJacobianp}
Similar calculation show that the Jacobian joint covariant of two
polynomials $f$ and $g$ of degrees $n_1$ and $n_2$ is given by
\begin{equation}\label{jacobianp}
J(f,g)=  n_2f'g-n_1fg'.
\end{equation}

In particular, it follows from  \eqref{hessianp} and \eqref{jacobianp} that
\begin{equation}\label{hjf'}
  H(f)=(n-1)J\dgx{n-1,n}(f',f).
\end{equation}
\end{example}

\begin{example}\label{EGundp}
The calculations  in \refE{EHessianp} generalize to the Gundelfinger
covariants in \refE{EGund} and show that the $k$th Gundelfinger covariant of a 
polynomial $f$ of degree $n$ is the determinant
\begin{equation}
  \begin{split}
  G_k(f;x)
&=
\lrabs{\frac{(n-2k+i+j)!}{(n-2k)!}f^{(2k-i-j)}(x)}_{0\le i,j\le k}
\\&
=
\lrabs{\frac{(n-i-j)!}{(n-2k)!}f^{(i+j)}(x)}_{0\le i,j\le k}.
  \end{split}
\end{equation}
\end{example}

Theorems \refand{Tsemi}{Tinv} translate to criteria for 
(joint) invariants and seminvariants of polynomials. For example, we have
the following. 

\begin{theorem}\label{Tsemipol}
A polynomial $\Phi$ in the coefficients of
one or several polynomials is a (joint) seminvariant   
if and only if
$\Phi$ is homogeneous and isobaric and invariant for all 
translations $x\mapsto x+x_0$, \ie, 
$\Phi(f(x+x_0))=\Phi(f(x))$ or
  $\Phi(f_1(x+x_0),\dots,f_\ell(x+x_0)) =\Phi(f_1,\dots,f_\ell) $.
\end{theorem}

\begin{theorem}\label{Tinvpol}
The following are equivalent for a polynomial $\Phi$ in the coefficients of
one or several binary forms.
\begin{romenumerate}
\item 
$\Phi$ is a (joint)  invariant.
\item 
$\Phi$ is a (joint) seminvariant and $\Phi(\ax\refl)=\Phi(\ax)$
or
$\Phi(\ax_1\refl,\dots,\ax_\ell\refl)=\Phi(\ax_1,\dots,\ax_\ell)$.
\item 
$\Phi$ is a (joint) seminvariant of order $\mm=0$.
\item 
$\Phi$ is a (joint) seminvariant and $\nn\nu=2w$ or 
$\nn_1\nu_1+\dots+\nn_\ell \nu_\ell=2w$ 
for the degree(s) and the weight (\ie, \eqref{weight} or \eqref{jweight}
holds). 
\end{romenumerate}
\end{theorem}

Here $\ax\refl$ is given by \eqref{arefl}; if $\ax$ are the coefficients of
$f$, then these are the coefficients of the reflected polynomial
$f\refl(x)\=x^\nn f(1/x)$. 

\subsection{Derivatives}\label{SSderivative}
The derivative $f'(x)$ is \emph{not} a (projective) covariant. (If it were,
it would be 
of order $\mm=\nn-1$ and its source would be $\nn a_0$; however, $\nn a_0$ is a
seminvariant of order $\nn$, not $\nn-1$, so \eqref{cweight} would not hold.)
Nevertheless, it is, as well as higher derivaties $f^{(j)}$, 
obviously invariant under translations (and affine maps),
and \refT{Tsemipol} implies the following, together with the obvious
generalization to joint seminvariants.
\begin{theorem}\label{Tdiff}
  If $\Phi\dgx m$ is a seminvariant of polynomials of degree $m\le \nn$, then 
$\Phi\dgx m(f^{(\nn-m)})$ is a seminvariant  of polynomials of degree $\nn$.
If $\Phi\dgx m$ has degree $\nu$, weight $w$ and order $\mm$, 
then $\Phi(f^{(\nn-m)})$ has the same degree $\nu$ and weight $w$,
while the order is increased to $\mm+(\nn-m)\nu$.
\end{theorem}
(This theorem is equivalent to \cite[Satz 2.18]{Schur}; note that the form
given there requires using the variables $\ha_i$.)

\begin{proof}
  The formula for the order follows from  \eqref{cweight}.
\end{proof}

\begin{remark}\label{Rdiff}
In particular, even if $\Phi$ is an invariant ($\mm=0$), 
$\Phi(f^{(\nn-m)})$ is not; it is only a seminvariant since its order is
$(\nn-m)\nu>0$. 
It follows  that the covariant corresponding to $\Phi(f^{(\nn-m)})$
can not be obtained immediately from the covariant corresponding to $\Phi$.
\end{remark}

Recall that we use subscripts $\dg \nn$ 
(on coefficients or seminvariants)
to denote the degree of the
considered polynomials. We have
\begin{align}\label{af'}
  a\dgxx{i}{\nn-1}(f')=(\nn-i)a_i(f),&&&
\ha\dgxx{i}{\nn-1}(f')=\nn\ha_i(f),
\end{align}
and, more generally, 
\begin{align}
  a\dgxx{i}{m}(f^{(\nn-m)})=(\nn-i)_{n-m}\, a_i(f),&&&
  \ha\dgxx{i}{m}(f^{(\nn-m)})=(\nn)_{n-m}\, \ha_i(f).
\end{align}

\begin{example}\label{EdiffHess}
  Applying the Hessian seminvariant for degree $\nn-1$ to $f'$, we obtain by
  \eqref{hessians}
  \begin{equation}\label{diffhess}
	\begin{split}
H\dgxx{0}{n-1}(f')&=
2(\nn-1)(\nn-2)\nn a_0(\nn-2)a_2-(\nn-2)^2((\nn-1)a_1)^2	
\\&
=(\nn-2)^2\,H_0(f),	  
	\end{split}
  \end{equation}
so, apart from a constant factor, we obtain the Hessian covariant for degree
$\nn$.   
\end{example}

\begin{example}\label{EdiffJac}
Similarly, for the Jacobian joint seminvariant in \refE{EJacobians},
\begin{equation}
J\dgx{n_1-1,n_2}(f',g)=(\nn_1-1)J(f,g).  
\end{equation}
\end{example}

\begin{example}
  For the $k$th Gundelfinger seminvariant we obtain by \eqref{gunds} and
  \eqref{af'}, 
generalizing \eqref{diffhess},
\begin{equation}
  \begin{split}
  g\dgxx{k}{n-1}(f')
&
=\bigpar{(n-1)_{2k}}^{k+1}\bigabs{\ha\dgxx{i+j}{n-1}(f')}_{i,j=0}^k
\\&
=\bigpar{(n-1)_{2k}}^{k+1}\bigabs{n\ha_{i+j}(f)}_{i,j=0}^k
=\bigpar{n(n-1)_{2k}}^{k+1}\bigabs{\ha_{i+j}(f)}_{i,j=0}^k
\\&
=(n-2k)^{k+1}g_k(f).
\raisetag\baselineskip
  \end{split}
\end{equation}

\end{example}

\begin{example}
  The $k$th transvectant seminvariant $\tau_k(f,g)$ is by \eqref{etranss} 
obtained by applying the apolar invariant to suitable derivatives:
  \begin{equation}
	\tau\dgxx{k}{n_1,n_2}(f,g)
= \frac{k!}{(\nn_1-k)!\,(\nn_2-k)!}
A\dgx{k}
\bigpar{f^{(\nn_1-k)},\,g^{(\nn_2-k)}}.
  \end{equation}
As a consequence,
  \begin{equation}
	\tau\dgxx{k}{n_1-1,n_2}(f',g)
=(n_1-k)
	\tau\dgxx{k}{n_1,n_2}(f,g)
  \end{equation}
and, more generally,
  \begin{equation}
	\tau\dgxx{k}{n_1-\ell_1,n_2-\ell_2}(f^{(\ell_1)},g^{(\ell_2)})
=(n_1-k)_{\ell_1}(n_2-k)_{\ell_2}
	\tau\dgxx{k}{n_1,n_2}(f,g),
  \end{equation}
so, apart from a constant factor, we obtain the $k$th transvectant
seminvariant of the original fumctions.
Note the special cases in \refE{EdiffHess} ($k=2$, $\ell_1=\ell_2=1$)
and \refE{EdiffJac} ($k=1$, $\ell_1=1$, $\ell_2=0$).
\end{example}

\subsection{Restriction to lower degree}\label{SSrestr}
Let $\Phi$ be a seminvariant of polynomials of degree (at most) $\nn$.
If $m<\nn$, then 
$\cP_m\subset\cP_n$, so
every polynomial of degree $m$ can
be regarded as a polynomial $\sumin a_i x^{\nn-i}$ with
$a_0=\dots=a_{\nn-m-1}=0$; 
thus, $\Phi(f)$ is defined for every such polynomial. (See \refR{Rdegree}.)

Note that we write a polynomial of
degree $m<\nn$ as
$\sum_{j=0}^ma\dgxx{j}{m}x^{m-j}=\sumin a\dgxx{i}{\nn} x^{\nn-i}$, and thus
  \begin{equation}
	\label{akl}
a\dgxx{i}{\nn}=
\begin{cases}
  a\dgxx{i-(\nn-m)}{m}, & \text{if } i\ge \nn-m\\
0. & \text{if } i<\nn-m.
\end{cases}
  \end{equation}

We denote the restriction of a seminvariant $\Phi$ to polynomials of degree
$m$ by $\Phi\restr{m}$.

\begin{theorem}
  \label{Trestr}
A seminvariant $\Phi$ of polynomials of degree $\nn$ is also a seminvariant of
polynomials of any given lower degree $\nn-j$. If $\Phi$ has degree $\nu$,
weight $w$ and order $\mm$, then its restriction $\Phi\restr {\nn-j}$ has degree
$\nu$, weight $w-jn$ and order $\mm+jn$. 
\end{theorem}
\begin{proof}
  It is an immediate consequence of \refT{Tsemipol} that $\Phi\restr{\nn-j}$ is a
  seminvariant. The degree is the same, but the weight of each $a_i$ is
  decreased by $j$ by \eqref{akl}, and thus the new weight is $w-j\nu$. The
  new order is by \eqref{cweight} given by
  \begin{equation*}
(\nn-j)\nu-2(w-j\nu)=	\nn-2w+j\nu=m+j\nu.
\qedhere
  \end{equation*}
\end{proof}
In particular, $\Phi\restr{\nn-j}$ has order $\mm+j\nu\ge\nu>0$ for any
seminvariant $\Phi$ and any $j>0$; hence, a non-trivial restriction is
never an invariant, even if $\Phi$ is an invariant.

\begin{example}
  The restriction $H_0\restr{\nn-1}$ of the Hessian seminvariant in
  \refE{EHessianp} is, recalling $a_0=0$,
  \begin{equation}
-(\nn-1)^2a\dgxx{1}{\nn}^2	
=
-(\nn-1)^2a\dgxx{0}{\nn-1}^2.	
  \end{equation}
This has degree 2, weight 0 and order $2\nn-2$, in agreement with \refT{Trestr}.
\end{example}

\begin{example}
Combining \refT{Trestr} and \refT{Tdiff}, we see that 
  if $\Phi$ is a seminvariant of 
polynomials of degree $\nn$, 
then so is $\Phi(f')=\Phi\restr{\nn-1}(f')$,
and more generally
$\Phi(f^{(j)})$ for every $j\ge1$. 
If $\Phi$ has
degree $\nu$, weight $w$ and order   $\mm=\nn\nu-2w$, 
then $\Phi(f')$ is a seminvariant
with degree $\nu$, weight $w-\nu$ and order $\mm+2\nu$.
\end{example}

\subsection{Reduced form}
The \emph{reduced form} of a polynomial $f(x)=\sumin a_ix^{\nn-i}$ of degree $\nn$
is the polynomial 
\begin{equation}\label{red}
\red f(x)=
\sumin \ar_i x^{\nn-i}
\=
f\Bigpar{x-\frac{a_1}{\nn a_0}};  
\end{equation}
note that $\ar_0\=a_0$ and $\ar_1=0$.
The reduced form is thus the unique translation $f(x-x_0)$ of $f$ with
vanishing coefficient for the second highest degree $x^{\nn-1}$.
Explicitly, by \eqref{red} and binomial expansions,
\begin{equation}
  \ar_i=\sum_{j=0}^ia_j\binom{\nn-j}{\nn-i}\Bigpar{-\frac{a_1}{\nn\,a_0}}^{i-j}.
\end{equation}
The coefficient $\ar_i$ is $a_0$ times a polynomial of degree $i$ in
$(a_j/a_0)_{j=1}^\nn$, and thus $a_0^{i-1}\ar_i$ is a polynomial in
$a_0,\dots,a_{\nn}$. This polynomial is homogeneous of degree $i$ and, as is
easily checked, isobaric with weight $i$. Furthermore, the reduced form is
the same for all translations $f(x-x_0)$, so its coefficients are
translation invariant.

\begin{theorem}\label{Tred}
  The coefficients $\ar_i$ of the reduced form of $f$ are 
rational seminvariants; more precisely $\ar_i$ is a
seminvariant
divided by  $a_0^{i-1}$.
The seminvariant
$a_0^{i-1}\ar_i$ has degree and weight $\nu=w=i$ and thus order
$\mm=(\nn-2)i$. 
\end{theorem}

\begin{example}\label{Ered0}
The constant term  
\begin{equation}
  \ar_{\nn}=\red f(0)=f\Bigpar{-\frac{a_1}{\nn\,a_0}}
=\sum_{j=0}^\nn a_j \Bigpar{-\frac{a_1}{\nn\,a_0}}^{\nn-j}
\end{equation}
is a real seminvariant  and $a_0^{\nn-1}\ar_{\nn}=a_0^{\nn-1}f(-a_1/\nn a_0)$ is a
seminvariant of degree and weight $\nn$ and order $\nn(\nn-2)$.

Note that every coefficient $\ar_i$ can be obtained as the constant term of
a derivative ${\red f}^{(\nn-i)}=\red {f^{(\nn-i)}}$, \cf{} \refT{Tdiff}.
\end{example}

\begin{example}
  \label{Ered1}
The first non-trivial reduced coefficient is
\begin{equation}
  \ar_2=a_0\binom \nn2\parfrac{a_1}{\nn\,a_0}^2-a_1(\nn-1)\frac{a_1}{\nn\,a_0}+a_2
=a_2-\frac{(\nn-1)\,a_1^2}{2\nn\,a_0}
=\frac{H_0}{2\nn(\nn-1)\,a_0},
\end{equation}
see \eqref{hessians}. The seminvariant $a_0\ar_2$ is thus a constant times
the Hessian seminvariant $H_0$.
\end{example}

Every homogeneous and isobaric polynomial in $\ar_2,\dots,\ar_{\nn}$ times a
power $a_0^s$
is a rational seminvariant, and a seminvariant if the exponent $s$ is large
enough. Conversely, every seminvariant $\Phi$ is translation invariant, and
thus $\Phi(f)=\Phi(\fr)$; hence every seminvariant is a polynomial in $a_0$ and
$\ar_2,\dots,\ar_{\nn}$. Up  to powers of $a_0$, every
seminvariant is thus a polynomial in the seminvariants $a_0^{i-1}\ar_i$.
However, these seminvariants do not form a basis (when $\nn\ge3$), since we 
may need need negative powers of $a_0$ in the representation.
For example, for $\nn=3$, by \eqref{p3}--\eqref{gdpq},
\begin{equation}
  \gD=-4a_0\,\ar_2^3-27 a_0^2\,\ar_3^2
=\frac{-4(a_0\ar_2^2)^3-27(a_0^2\ar_3)^2}{a_0^2}.
\end{equation}
In general, we have the following theorem.
\begin{theorem}\label{Tredall}
  If $\Phi$ is a seminvariant with degree $\nu$ and weight $w$, then 
  \begin{equation}
	\label{tredall}
\Phi=a_0^{\nu-w}G\bigpar{(a_0^{i-1}\ar_i)_{i=2}^\nn} 
  \end{equation}
for some isobaric polynomial\/ $G$ of weight $w$.
Consequently, $\Phi$ is a polynomial in $a_0$ and
$a_0\ar_2,\dots,a_0^{\nn-1}\ar_{\nn}$ if and only if $\nu\ge w$.

If\/ $\nu\ge w$, then \eqref{tredall} gives a one-to-one correspondence
between seminvariants with degree $\nu$ and weight $w$, and isobaric
polynomials  
$G\bigpar{(a_0^{i-1}\ar_i)_{i=2}^\nn} $
of weight $w$.
\end{theorem}
\begin{proof}
Each term  $a_0^{i-1}\ar_i$ has the same degree and weight, and thus so has
every (isobaric)
polynomial $G\bigpar{(a_0^{i-1}\ar_i)_{i=2}^\nn}$ in them, while $a_0$ has
degree 1 and weight $0$. 
Hence an isobaric term 
$a_0^s G\bigpar{(a_0^{i-1}\ar_i)_{i=2}^\nn}$ has weight $v$ and degree $s+v$
for some $w$, and thus we must have $v=w$ and $s=\nu-v=\nu-w$.
\end{proof}

By \refE{DD} below, the discriminant $\gD$ has $\nu=2(\nn-1)$ and
$w=\nn(\nn-1)$, so 
$\nu-w=-(\nn-2)(\nn-1)<0$ for any $\nn\ge3$, and then $\gD$ is not a
polynomial in  $a_0$ and $a_0^{i-1}\ar_i$.

\begin{remark}
\label{Rredall}  
In the case $\nu\ge w$, we see again
that the dimension $N(n,\nu,w)$ is independent of $\nu$ as long as $\nu\ge w$.
Moreover, $N(n,\nu,w)$ then equals the number of isobaric monomials 
of weight $w$ in $(\ar_i)_{i=2}^n$; this number has the generating function
$\prod_{i=2}^n(1-q^i)\qw$,  which yields another proof of \refC{CGauss}.
\end{remark}

\section{Invariants  and roots}\label{Sroots}

Let the polynomial $f$ of degree $\nn$ have roots $\xin$
(possibly in some extension of the ground field). Then, as is well-known,
\begin{equation}\label{sofie}
  a_i=(-1)^i a_0 e_i(\xin),
\end{equation}
where $e_i$ is the $i$:th symmetric polynomial; note that $e_i$ has degree
$i$.
If $\Phi(f)$ is a seminvariant, we can thus write $\Phi(f)$ as a
polynomial $\Phix(\xin;a_0)$.
\begin{theorem}\label{Tinvroots}
  A polynomial $\Phix(\xin;a_0)$ is a seminvariant of degree $\nu$ and
  weight $w$ if and only if
$\Phix(\xin;a_0)=a_0^{\nu}\gf(\xin)$ where
  \begin{romenumerate}
  \item 
$\gf$ is symmetric in $\xin$;
  \item 
$\gf$ is homogeneous of degree $w$ in $\xin$;
  \item
$\gf$ is translation invariant, \ie,
	$\gf(\xi_1-x_0,\dots,\xi_\nn-x_0)=\gf(\xin)$. Equivalently, $\gf(\xin)$
	is a polynomial in the differences $\xi_j-\xi_{\nn}$.
  \item 
$\nu\ge \deg_{\xi_1}\bigpar{\gf(\xin)}$, the degree of $\xi_1$ in $\gf(\xin)$.
  \end{romenumerate}

Furthermore, $\Phix$ is an 
invariant if and only 
the above holds and  $\nn\nu=2w$; in this case
\begin{equation}\label{xisym}
  (\xi_1\dotsm\xi_{\nn})^{\nu}\gf(\xi_1\qw,\dots,\xi_{\nn}\qw)=(-1)^w\gf(\xin).
\end{equation}
\end{theorem}
\begin{proof}
Recall that every symmetric polynomial is a polynomial in $e_1,\dots,e_{\nn}$.
Then use \refT{Tsemipol} and \eqref{sofie} and note that each $\xi_j$ has
  weight 1  by the fact that $a_i$ has weight $i$. 
This might yield terms containing negative powers of $a_0$, and
(iv) is necessary and sufficient for $\Phi$ to be a polynomial in
$a_0,\dots,a_{\nn}$. 
The symmetry $\Phi(\ax\refl)=(-1)^w\Phi(\ax)$ translates to \eqref{xisym}.
We omit the details.
\end{proof}


\begin{example}
  \label{DD}
The \emph{discriminant} of $f$ is
\begin{equation}\label{dd}
\gD(f)\=a_0^{2\nn-2} \gD_0(f)
=a_0^{2\nn-2}\prod_{1\le i< j\le \nn} (\xi_i-\xi_j)^2,
\end{equation}

This is symmetric and has degree $w=\nn(\nn-1)$ in $\xin$. 
It follows from \refT{Tinvroots}, 
since $\nu=2(\nn-1)$, that the discriminant $\gD$ is an 
invariant of
degree $\nu=2(\nn-1)$ and weight $w=\nn(\nn-1)$.
(Other notation: $\gD=D\scite{Schur}$.)
\end{example}

\begin{example}
  \label{EHessxi}
The sum $\sum_{1\le i<j\le \nn}(\xi_i-\xi_j)^2 $ satisfies (i)--(iii) in
\refT{Tinvroots},  and has degree 2 in $\xi_1$, so 
$a_0^2\sum_{1\le i<j\le \nn}(\xi_i-\xi_j)^2 $ is a seminvariant.
We have, using \eqref{sofie},
\begin{multline*}
\sum_{1\le i<j\le \nn}(\xi_i-\xi_j)^2  
=(\nn-1)\sumiin \xi_i^2-2\sum_{1\le i<j\le \nn}\xi_i\xi_j
\\
=(\nn-1)\Bigpar{\sumiin \xi_i}^2-2\nn\sum_{1\le i<j\le \nn}\xi_i\xi_j
=(\nn-1)\Bigparfrac{-a_1}{a_0}^2-2\nn\frac{a_2}{a_0},  
\end{multline*}
so
\begin{equation}\label{Hroots}
a_0^2\sum_{1\le i<j\le \nn}(\xi_i-\xi_j)^2  
=(\nn-1)a_1^2-2\nn a_0a_2
=-\frac{1}{\nn-1} H_0(f),  
\end{equation}
where $H_0$ is the Hessian seminvariant in Examples \ref{EHessians} and
\ref{EHessianp}. 
\end{example}

\begin{example}\label{Ereda}
  Let $\bxi\=\frac1n\sumiin\xi_i=-a_1/na_0$.
Then the roots of the reduced polynomial $\red f$ are
$\xi_1-\bxi,\dots,\xi_n-\bxi$. 
Any symmetric homogeneous polynomial in $\xibxi$ satifies
\refT{Tinvroots}(i)--(iii), and multiplied by a suitable power of $a_0$, it
is thus a seminvariant. Since any such polynomial can be written as an isobaric
polynomial in $\ar_i/a_0$, this also follows by \refT{Tred} or
\refT{Tredall}. 

In particular, the elementary symmetric polynomials $e_k$ yield the rational
seminvariants
\begin{equation}\label{ek}
 e_k(\xibxi)=(-1)^k\red a_k/a_0.
\end{equation}
\end{example}

\begin{example}\label{Eredb}
  As another example of the construction in \refE{Ereda}, consider the power
  sum
  \begin{equation}
S_k\=\sumiin(\xi_i-\bxi)^k	
  \end{equation}
and the seminvariant $a_0^kS_k$.
Note that $S_0=n$ is a constant and $S_1=0$. Further,
\begin{equation}
  S_2=\sumiin\xi_i^2-n\bxi^2
=\frac1{2n}\sum_{i,j=1}^n(\xi_i-\xi_j)^2,
\end{equation}
so by \eqref{Hroots},
\begin{equation}
  a_0^2S_2=-\frac{1}{n(n-1)}H_0(f).
\end{equation}

Furthermore, $S_k$ can be expressed in $e_1,\dots,e_k$ by the 
standard generating function identity
\begin{equation}\label{e=S}
  \log\lrpar{\sum_{k=0}^\infty e_k(-t)^k}
=
\sumiin\log(1-t\xi_i)
=-\sum_{k=1}^\infty S_k\frac{t^k}k.
\end{equation}
which leads to the
classical \emph{Newton's identities}
(with $e_0=1$ and $e_k=0$ for $k>n$),
\begin{equation}
 ke_k=\sum_{i=1}^k(-1)^{i-1}e_{k-i}S_i,\qquad k\ge1.
\end{equation}
In our situation, the arguments are $\xibxi$; thus
$S_1=e_1=0$ and we have, for example, 
\begin{align}
  S_2&=e_1^2-2e_2=-2e_2, \\
S_3&=3e_3,\\
S_4&=-4e_4+2e_2^2.
\end{align}
Thus, by \eqref{ek} and \refE{Ered1}, we obtain the seminvariants
\begin{align}
  a_0^2S_2&=-2a_0^2e_2=-2a_0\ar_2=-\frac{1}{n(n-1)}H_0, \label{s2}\\
  a_0^3S_3&=3a_0^3e_3=-3a_0^2\ar_3,\label{s3}\\
  a_0^4S_4&=-4a_0^4e_4+2(a_0^2e_2)^2=-4a_0^3\ar_4+2(a_0\ar_2)^2.\label{s4}
\end{align}
Note that $a_0^kS_k$ has degree and weight $\nu=w=k$ (see \refT{Tinvroots}).
\end{example}

\begin{example}\label{Eredc}
  Consider the random variable $X=\xi_Y$, where $Y\in\set{1,\dots,n}$ is a
  random index (with uniform distribution).
Then $X$ has mean $\E X=\bxi$ and centred moments
\begin{equation}
  \E(X-\E X)^k
=\E(X-\bxi)^k=\tfrac1n S_k.
\end{equation}
Thus, $a_0^k  \E(X-\E X)^k$ equals $n\qw$ times the seminvariant in
\refE{Eredb}. 

\citet[\S7.6]{KR} suggested studying the \emph{cumulants} $\chi_k$ of
$X$, $k\ge2$. These are defined by the generating function
\begin{equation}
  \exp\lrpar{\sum_{k=1}^\infty \chi_k \frac{t^k}{k!}}
=\E e^{tX} =\sumk \E X^k\frac{t^k}{k!},
\end{equation}
and thus, since $\chi_1= EX=\bxi$, 
\begin{equation}\label{chi=s}
  \exp\lrpar{\sum_{k=2}^\infty \chi_k \frac{t^k}{k!}}
=\E e^{t(X-\bxi)} =\sumk \frac{1}{n}S_k\frac{t^k}{k!}
.
\end{equation}
By expanding, we obtain the standard formulas for $\chi_k$ as a polynomial
in $S_1,\dots,S_k$, or, using \eqref{e=S}, in $e_1,\dots,e_k$.
For example,
\begin{align}
  \chi_2&=\E(X-\E X)^2=\frac1n S_2=-\frac2n e_2, \label{chi2}\\
  \chi_3&=\E(X-\E X)^3=\frac1n S_3=\frac3n e_3, \label{chi3}\\
  \chi_4&=\E(X-\E X)^4-3\lrpar{\E(X-\E X)^2}^2
=\frac1n S_4 -\frac3{n^2}S_2^2=-\frac4n e_4+\frac{2n-12}{n^2}e_2^2.
\label{chi4}
\end{align}
It follows from \eqref{chi=s} that $a_0^k\chi_k$ is an isobaric polynomial in
$a_0^jS_j$, $j=2,\dots,k$, and thus a seminvariant, with degree and
weight $\nu=w=k$.
For example, by \eqref{chi2}--\eqref{chi4} and \eqref{s2}--\eqref{s4},
\begin{align}
 a_0^2 \chi_2&=-\frac2n a_0^2e_2 =-\frac2n a_0 \ar_2
=-\frac1{n^2(n-1)}H_0,
\label{achi2}\\
a_0^3 \chi_3&=\frac3n a_0^3e_3=-\frac3n a_0^2\ar_3,
\label{achi3}\\
a_0^4  \chi_4&
=-\frac4n a_0^4e_4+\frac{2n-12}{n^2}\bigpar{a_0^2e_2}^2
=-\frac4n a_0^3\ar_4+\frac{2n-12}{n^2}\bigpar{a_0\ar_2}^2.
\label{achi4}
\end{align}
\end{example}

\subsection{The case $a_0=0$}
We have implicitly assumed $a_0\neq0$ above. If $a_0=0$, then $f$ has degree
at most $\nn-1$, and thus at most $\nn-1$ roots. We then adopt the projective
view and regard $\infty$ as a
root of multiplicity $\nn-\deg(f)$, so that $f$ still has $\nn$ roots (counted
with multiplicity); these correspond (just as in the case $a_0\neq0$)
to the zeros of the corresponding
binary form $\sumin a_i x^{\nn-i}y^i$ of degree $\nn$.

We can apply a limit argument to find 
the expression for a seminvariants in the roots of $f$ in this case too.

\begin{theorem}\label{Tinvroots-1}
Let $\Phi$ be a seminvariant of polynomials of degree $\nn$, and that
$\Phi(f)=a_0^{\nu}\gf(\xin)$  
for some polynomial $\gf$ in the roots $\xi_1,\dots,\xi_{\nn}$ of
$f$.
Then the restriction to polynomials of degree $\nn-1$ is given by
$\Phi\restr{\nn-1}(f)=
a\dgxx{0}{\nn-1}^{\nu}\gf\restr{\nn-1}(\xini)$,
where $\gf\restr{\nn-1}$ is obtained  
from $\gf$ by first replacing each 
monomial $\xi_1^{j_1}\dotsm\xi_{\nn}^{j_{\nn}}$
by $\xi_1^{j_1}\dotsm\xi_{\nn-1}^{j_{\nn-1}}$ if $j_{\nn}=\nu$ and by $0$
otherwise, 
and then multiplying by $(-1)^{\nu}$.
\end{theorem}
Note that $j_{\nn}\le\nu$ for every term by \refT{Tinvroots}(iv) (and symmetry).

In other words, we delete all terms in $\gf$ not containing
a factor $\xi_{\nn}^{\nu}$, and 
replace each factor $\xi_{\nn}^{\nu}$ by $(-1)^{\nu}$.

\begin{proof}
  Fix $\xini$ and let
$\xi_{\nn}=b/a_0$ for some $b$; now let $a_0\to0$.
(This limit can be done in a purely formal way, for any field, and does not
  really assume any kind of continuity. All quantities below are polynomials
  in $a_0$, and we just substitute $a_0=0$ in them.)
Then, for $1\le i\le \nn$,
\begin{equation*}
  \begin{split}
  a_0e_i(\xi_1,\dots,\xi_{\nn})
&
= a_0e_i(\xi_1,\dots,\xi_{\nn-1})
+  a_0\xi_{\nn}e_{i-1}(\xi_1,\dots,\xi_{\nn-1})
\\&
\to b e_{i-1}(\xi_1,\dots,\xi_{\nn-1}).	
  \end{split}
\end{equation*}
Hence, comparing with \eqref{sofie}, we see that the coefficients of the 
polynomial $f$ with roots $\xin$ and leading coefficient $a_0$ tend to the
coefficients of the polynomial $f_1$ of degree $\nn-1$
with roots $\xini$ and leading term $-bx^{\nn-1}$, \ie, leading coefficient
$a\dgxx{0}{\nn-1}(f_1)=-b$. 
Consequently, $\Phi(f)\to\Phi(f_1)$. The result follows by noting that, as
$a_0\to0$, $a_0^{\nu}\xi_{\nn}^{j_{\nn}}\to0$ if $j_{\nn}<\nu$, while
$a_0^{\nu}\xi_{\nn}^{\nu}=b^{\nu}=(-a\dgxx{0}{\nn-1}(f_1))^{\nu}$.
\end{proof}

\begin{example}
  \label{Ediskr-1}
Applying \refT{Tinvroots-1} to the discriminant in \eqref{dd} we find
for a polynomial $f$ of degree $\nn-1$
\begin{equation}\label{dd-1}
\gD\dgx{\nn}(f)
=a_0^{2\nn-2}\prod_{1\le i< j\le \nn-1} (\xi_i-\xi_j)^2
=a_0^2\,\gD\dgx{\nn-1}(f).
\end{equation}
If we repeat, we find that for any $f$ of degree $\nn-2$ (or smaller),
$\gD\dgx{\nn}(f)=0$, in accordance with our view that then $f$ has a double
root $\infty$.
\end{example}

\subsection{Joint invariants}
\refT{Tinvroots} extends to the case of several polynomials
$f_1,\dots,f_\ell$.
Let the polynomial $f_j$ have degree $\nn_j$ and roots $\xijn$
(possibly in some extension of the ground field). Then, by \eqref{sofie},
\begin{equation}\label{sofiej}
  a_i(f_j)=(-1)^i a_0(f_j) e_i(\xijn),
\end{equation}
and if $\Phi(f_1,\dots,f_\ell)$ is a joint seminvariant, we can write
it as a 
polynomial $\Phix(\xijn;a_0(f_1),\dots,a_0(f_\ell))$ in all roots and
leading coefficients.

\begin{theorem}\label{Tinvrootsj}
  A polynomial $\Phix(\xiln;a_0(f_1),\dots,a_0(f_\ell))$ is a joint
  seminvariant of  
$f_1,\dots,f_\ell$ with
degrees $\nu_1,\dots,\nu_\ell$ and
  weight $w$ if and only if
$\Phix(\xiln;a_0)=\prod_{j=1}^\ell a_0(f_j)^{\nu_j}\cdot
\gf(\xiln)$, where
  \begin{romenumerate}
  \item 
$\gf$ is symmetric in each set of roots $\xijn$, $j=1,\dots,\ell$;
  \item 
$\gf$ is homogeneous of degree $w$ in $\xiln$;
  \item
$\gf$ is translation invariant, \ie,
	$\gf(\xi^{(1)}_1-x_0,\dots,\xi^{(\ell)}_\nn-x_0)=
	\gf(\xi^{(1)}_1,\dots,\xi^{(\ell)}_\nn)$.
  \item 
$\nu_j\ge \deg_{\xij_1}\bigpar{\gf(\xiln)}$, the degree of $\xij_1$ in 
$\gf(\xiln)$.
  \end{romenumerate}

Furthermore, $\Phix$ is an
invariant if and only 
the above holds and  $\nn_1\nu_1+\dots+\nn_\ell\nu_\ell=2w$; in this case
\begin{equation}\label{xisymj}
 \prod_{j=1}^{\ell} (\xij_1\dotsm\xij_{\nn})^{\nu_j}\cdot
\gf\bigpar{(\xi^{(1)}_1)\qw,\dots,(\xi^{(\ell)}_{\nn})\qw}=(-1)^w\gf(\xiln).
\end{equation}
\end{theorem}
\begin{proof}
  As for \refT{Tinvroots}, with obvious modifications.
\end{proof}

\begin{example}
  \label{Eresultant}
The \emph{resultant} of two polynomials $f=\sum_{i=0}^n a_{n-i}x^i$ and 
$g=\sum_{j=0}^m b_{m-j}x^j$ of degrees $n$ and $m$ 
and with roots $\xin$ and $\eta_1,\dots,\eta_m$ 
is
\begin{equation}\label{resultant}
R(f,g)\=a_0^{m}b_0^n\prod_{i=1}^n\prod_{j=1}^m (\xi_i-\eta_j).
\end{equation}
This is symmetric in $\xijn$ and in $\eta_1,\dots,\eta_m$ 
and has degree $w=nm$ in $\xin,\eta_1,\dots,\eta_m$. 
\refT{Tinvrootsj} applies, with $\nn_1=n$, $\nn_2=m$, $\nu_1=m$, $\nu_2=n$,
and thus 
$n_1\nu_1+n_2\nu_2=2mn=2w$; hence the resultant $R$ is a joint 
invariant of
degrees $(m,n)$ and weight $nm$.

Note that $R(g,f)=(-1)^{mn}R(f,g)$ and the formulas, see \eg{} \cite{SJN5},
\begin{equation}
  R(f,g)=a_0^m\prod_{i=1}^ng(\xi_i)
=(-1)^{mn}b_0^n\prod_{j=1}^mf(\eta_j).
\end{equation}
\end{example}

\begin{example}\label{EResHess}
  Consider the resultant of $f$ and $H(f)$, where $f$ is a polynomial of
  degree $n$. By \eqref{hessianp},
$H(f)(\xi_i)=-(n-1)^2(f'(\xi_i))^2$, and thus
  \begin{equation}
	\begin{split}
R\bigpar{f,H(f)}
&=
a_0^{2n-4}(-1)^n(n-1)^{2n}\prod_{i=1}^n f'(\xi_i)^2
=	  a_0^{-2}(-1)^n(n-1)^{2n}R(f,f')^2
\\&
=(-1)^n(n-1)^{2n}\gD^2,
	\end{split}
\raisetag\baselineskip
  \end{equation}
since $\gD=(-1)^{n(n-1)/2}a_0\qw R(f,f')$, see \cite{SJN5}.
\end{example}

\begin{example}
  \label{EJacxi}
Consider again two polynomials $f=\sum_{i=0}^n a_{n-i}x^i$ and 
$g=\sum_{j=0}^m b_{m-j}x^j$ of degrees $n$ and $m$ 
and with roots $\xin$ and $\eta_1,\dots,\eta_m$.
The difference $m\sum_{i=1}^{n}\xi_i-n\sum_{j=1}^{m}\eta_j$
satisfies (i)--(iii) in
\refT{Tinvrootsj},  and has degree 1 in $\xi_1$ and $\eta_1$, so 
$a_0b_0\bigpar{m\sum_{i=1}^{n}\xi_i-n\sum_{j=1}^{m}\eta_j}$
is a joint seminvariant.
We have, using \eqref{sofie},
\begin{multline*}
a_0b_0\Bigpar{m\sum_{i=1}^{n}\xi_i-n\sum_{j=1}^{m}\eta_j}
=
a_0b_0\Bigpar{m\frac{-a_1}{a_0}-n\frac{-b_1}{b_0}}
=-ma_1b_0+na_0b_1,
\end{multline*}
so
this equals the Jacobian seminvariant in Examples
\refand{EJacobians}{EJacobianp}. 
\end{example}

\subsection{Covariants}

Similarly, a covariant can be written as a polynomial in $x$ and the roots
$\xin$ of $f$. The following theorem yields an explicit formula.

\begin{theorem}\label{Tcovroots}
  Let $\Psi$ be a covariant of polynomials of degree $n$, and let $\Phi$ be
  its source.
Suppose that $\Phi(f)=a_0^{\nu}\gf(\xin)$ as in \refT{Tinvroots}.
Then
\begin{equation*}
  \Psi(f;x)=a_0^\nu\prodin(x-\xi_i)^\nu\cdot
 \gf\Bigpar{\frac1{x-\xi_1},\dots,\frac1{x-\xi_n}}.
\end{equation*}
\end{theorem}

\begin{proof}
 Let $\tf$ be the binary form corresponding to $f$, and let $\tPsi(\tf)$ be
 the covariant corresponding to $\Psi(f)$; thus $\Psi(f;x)=\tPsi(\tf;x,1)$;
further, let $\tPhi$ be the source of $\tPsi$.

The reflection
$\rho(x,y)\=(y,x)$ has matrix $\smatrixx{0&1\\1&0}$ and determinant
$|\rho|=-1$; thus \eqref{cinv} yields
\begin{equation}\label{kk2}
  \Psi(f;0)=\tPsi(\tf;0,1)=(-1)^w\tPsi(\rho \tf;1,0)
=(-1)^w\tPhi(\rho \tf)
=(-1)^w\Phi(f\refl),
\end{equation}
where $f\refl$
is the polynomial corresponding to
$\rho\tf(x,y)=\tf(y,x)=\sumin a_iy^{n-i}x^i$.
We have
\begin{equation}
  \begin{split}
f\refl(x)
&=\sumin a_ix^i=x^nf(1/x)
=a_0x^n\prodin\lrpar{x\qw-\xi_i}
=a_0\prodin(1-x\xi_i)
\\&
=a_0\prodin(-\xi_i)\,\prodin(x-\xi_i\qw),
  \end{split}
\end{equation}
with roots $\xi_1\qw,\dots,\xi_n\qw$.
Consequently, 
\begin{equation}\label{kk3}
\Phi(f\refl)
=\lrpar{a_0\prodin(-\xi_i)}^\nu\gf(\xi_1\qw,\dots,\xi_n\qw)
	\end{equation}
and thus by \eqref{kk2},
since $\gf$ is homogeneous of degree $w$,	  
	\begin{equation}\label{kk4}
  \Psi(f;0)
=(-1)^w
\Phi(f\refl)
=
a_0^\nu\prodin(-\xi_i)^\nu\gf\Bigpar{\frac1{-\xi_1},\dots,\frac1{-\xi_n}}.	
\end{equation}
We have $\Psi(f;x)=\Psi^*(\xin;x)$ for some polynomial $\Psi^*$, and 
\begin{equation}
  \label{kk1}
\Psi^*(\xin;x)=\Psi^*(\xi_1-x,\dots,\xi_n-x;0)
\end{equation}
by translation invariance.
The result follows by \eqref{kk1} and  \eqref{kk4}.
\end{proof}

\begin{example}
  As a trivial example, the covariant $f$ 
has  source $a_0$, and \refT{Tcovroots} yields, with $\gf=1$,
$f=a_0\prodin(x-\xi_i)$.
\end{example}

\begin{example}
  \label{EHessroots}
The source of the Hessian covariant is, by \eqref{Hroots}, 
\begin{equation}
  H_0(f)=
-(n-1)a_0^2\sum_{1\le i<j\le n}(\xi_i-\xi_j)^2.
\end{equation}
Hence \refT{Tcovroots} shows that the Hessian covariant is given  by
\begin{equation}
  \begin{split}
  H(f;x)&=
-(n-1)a_0^2\prodkn(x-\xi_k)^2
\sum_{1\le i<j\le n}\Bigpar{\frac1{x-\xi_i}-\frac1{x-\xi_j}}^2
\\&
=
-(n-1)a_0^2
\sum_{1\le i<j\le n}\xpar{\xi_i-\xi_j}^2
\prod_{k\neq i,j}(x-\xi_k)^2.
  \end{split}
\end{equation}
Note that by extracting the leading coefficients (the coefficients of
$x^{2n-4}$), we recover \eqref{Hroots}.
\end{example}

The extension to joint covariants is straightforward; we leave the
formulation to the reader and give only a simple example.

\begin{example}
  The source of the Jacobian covariant $J(f,g)$ is by \refE{EJacxi}
  \begin{equation}
a_0b_0\Bigpar{m\sum_{i=1}^{n}\xi_i-n\sum_{j=1}^{m}\eta_j}
  \end{equation}
and thus
  \begin{equation}
J(f,g)=
a_0b_0\prodin(x-\xi_i)\prodjm(x-\eta_j)
\lrpar{m\sum_{i=1}^{n}(x-\xi_i)\qw-n\sum_{j=1}^{m}(x-\eta_j)\qw}.
  \end{equation}
\end{example}

\section{Some characterizations of vanishing invariants}\label{S=0}

In some cases, there are simple characterizations of vanishing invariants or
covariants. For example, the following basic result is an immediate
consequence of \eqref{dd}.
For simplicity, we assume in this section that $a_0\neq0$, \ie, that the
actual degree is $n$; the results immediately extend to the case $a_0=0$
by projective invariance (considering roots at infinity), see for example
\refE{Ediskr-1}.  (The results all have invariant formulations for binary
forms.) 

\begin{theorem}\label{TD=0}
  The discriminant $\gD(f)=0$ if and only if $f$ has a double root (in some
  extension field). 
\end{theorem}

Equivalently, a binary form $f(x,y)$ has discriminant 0  if and only if 
it has a square factor $(ax+by)^2$.

\begin{theorem}\label{THess=0}
  The  Hessian covariant $H(f)=0$ if and only if $f$
has a single root, \ie,  $\xi_1=\dots=\xi_n$; equivalently,
$f(x)=a_0(x-\xi)^n$ for some $a_0$ and $\xi$. 
\end{theorem}

Equivalently, a binary form $f(x,y)$ has Hessian covariant $H(f)=0$ 
if and only if it equals $c(ax+by)^n$ for some $a,b,c$.

\begin{theorem}\label{TJac=0}
  The  Jacobian joint covariant $J(f,g)=0$ if and only if $f$ and $g$
have the same roots, and their multiplicities always are in the same
proportion $n_1/n_2$;
equivalently, $f(x)=a_0h(x)^{\mud_1}$ and $g(x)=b_0h(x)^{\mud_2}$ for some
polynomial $h$ and some integers $\mud_1,\mud_2\ge1$.
In particular, if $f$ and $g$ have the same degree, then $J(f,g)=0$ if and
only if $f$ and $g$ are proportional.
\end{theorem}

\begin{proof}
  Suppose that $\xi$ is a root of $f$ or $g$, and let the multiplicities of
  the root be $k_1\ge0$ and $k_2\ge0$ (with $k_1+k_2>0$). By a projective
  transformation we may assume that $\xi=0$. Then $f(x)=ax^{k_1}+\dots$ and
  $g(x)=bx^{k_2}+\dots$ (showing the lowest degree terms only), with
  $a,b\neq0$,
and \refE{EJacobianp} shows that
\begin{equation}
  J(f,g)=(n_2k_1-n_1k_2)abx^{k_1+k_2-1}+\dots
\end{equation}
Hence $J(f,g)=0$ implies $n_2k_1-n_1k_2=0$, and thus both $k_1$ and $k_2$ are
non-zero  and $k_1/k_2=n_1/n_2$.
This shows that $f^{n_2}$ and $g^{n_1}$ have the same roots, with the same
multiplicities, and thus $f^{n_2}=cg^{n_1}$ for some $c$. The result follows
by the unique factorization of polynomials into irreducible ones.

Conversely, if $f=a_0h^{\mud_1}$ and $g=b_0h^{\mud_2}$, then
\begin{equation*}
  J(f,g)=a_0b_0\mud_1h^{\mud_1-1}\mud_2h^{\mud_2-1}J(h,h)=0.
\qedhere
\end{equation*}
\end{proof}

\begin{proof}[Proof of \refT{THess=0}]
  If $H(f)=0$, then \eqref{hjf'} yields $J(f',f)=0$, and thus \refT{TJac=0}
yields
$f'=ah^{\mud_1}$, $f=bh^{\mud_2}$ for some polynomial $h$ and constants $a$,
  $b$, $\mud_1$ and $\mud_2$. Then $n-1=\mud_1\deg(h)$ and $n=\mud_2\deg(h)$,
and consequently $1=(\mud_2-\mud_1)\deg(h)$, which implies $\deg(h)=1$ and
$\mud_2=n$.

The converse follows directly from \eqref{hjf'} and \refT{TJac=0}.
\end{proof}

\begin{theorem}[{\cite[Satz 2.11]{Schur}}]\label{Tinv0}
All non-constant invariants of a polynomial of degree $\nn$ vanish if and only
if the polynomial has a root of multiplicity $>\nn/2$. 
(This includes the case when the actual degree is $<\nn$;
there is a root at $\infty$ of multiplicity more than $\nn/2$ when
the degree is $<\nn/2$.)
\end{theorem}

Note that seminvariants still may be non-zero.

\begin{example}
  The seminvariant $a_0(f)=0$ if and only if $\deg(f)\le \nn-1$. 
If, for example,  $f(x)=x^\nn$, then $a_0\neq0$ while all invariants vanish by
\refT{Tinv0}. 
\end{example}

\begin{example}
  If $f(x)=x^{\nn-1}(x-b)$, with $a\neq0$, then $f$ has a root
  $\xi_1=\dots=\xi_{\nn-1}=0$ of   multiplicity $(\nn-1)$ and a simple root
  $\xi_{\nn}=b$. If $\nn\ge3$, then all invariants of $f$ vanish by \refT{Tinv0}.
However, the Hessian seminvariant is $-b^2\neq0$ by \eqref{Hroots} (or by
\eqref{hessians} and $a_1=-b$, $a_2=0$).

Similarly, again by \eqref{Hroots}, the Hessian seminvariant is non-zero for
any polynomial with all $\nn$ roots real, unless all roots coincide.
\end{example}

\refT{THess=0} characterizes the polynomials that can be written as an $n$th
power $c(x-\xi)^n$. There is a generalization (due to \citet{Gund}, see also
\citet{Kung}) 
to sums $\sum_{i=1}^m c_i(x-\xi_i)^n$ of a given number $m$ of such powers;
however, we also have to 
include limit cases corresponding to several coinciding $\xi_i$, and the
precise statement is as follows.

\begin{theorem}\label{TGund=0}
  The following are equivalent, for any polynomial $f$ of degree $n$ and
  $1\le m\le n$.
  \begin{romenumerate}
  \item 
$G_m(f)=0$.
  \item 
$f$ belongs to the closure $\bcP_{n,m}$ of the set of polynomials
$\cP_{n,m}\=\set{\sum_{i=1}^m c_i(x-\xi_i)^n:c_i,\xi_i\in F}$.
\item \label{tgundsum}
$f=\sum_{i=1}^l \sum_{j=0}^{m_i-1} c_{ij}(x-\xi_i)^{n-j}$
for some $l\le m$, $m_i\ge1$ with $\sum_{i=1}^l m_i=m$, $c_{ij}\in F$ and
$\xi_i\in \Fx$,
for $i=1,\dots,l$ and $j=0,\dots,m_i-1$.
\item \label{tgundapolar}
There exists a polynomial $g$ of degree (at most) $m$ such that the
  apolar invariant $\tv{f,g}_m=0$.
  \end{romenumerate}
\end{theorem}

\begin{remark}
In \ref{tgundsum}, we allow the possibility $\xi_i=\infty$; in this case we use the
interpretation $(x-\infty)^{n-j}\=x^j$ (which is natural from a projective
perspective). 

By ``closure'' in (ii), we mean in the ordinary topological sense
(identifying a polynomial with its vector $\ax$ of coefficients) if, for
example, we consider the field of rational, real or complex numbers.
In general, the closure can be interpreted algebraically, as the set of all
$f=f_0$ for some family $f_\eps$ of polynomials, 
with coefficients that are
polynomials in a parameter 
$\eps\in F$ (or $\eps\in\bbQ$),
such that $f_\eps\in\cP_{n,m}$ 
for all $\eps\neq0$.  
\end{remark}

In particular, if $m>n/2$, then $\bcP_{n,m}=\cP_n$, \ie, every polynomial is
in $\bcP_{n,m}$, since then $G_m$ vanishes identically on $\cP_n$, see
\refE{EGund}. If $n$ is even, then $G_{n/2}$ is a multiple of the
catalecticant $\Han(f)$, see \refE{EGund}, and thus 
we have the corollary:
\begin{corollary}\label{CGund=0}
  If $n$ is even then $\Han(f)=0$ if and only if
$f\in\bcP_{n,n/2}$, \ie, if and only if
$f$ is as in \refT{TGund=0}\ref{tgundsum} with $m=n/2$.
\end{corollary}

The relation between \ref{tgundsum} and \ref{tgundapolar} in \refT{TGund=0}
can be made more precise as follows, see \cite{KR}.

\begin{theorem}\label{TGund=0x}
  The following are equivalent, for a polynomial $f$ of degree $n$ and
given $\xi_i\in \Fx$ and $m_i\ge1$, $i=1,\dots,l$,
with $m\=\sum_{i=1}^l m_i\le n$, 
  \begin{romenumerate}
\item 
$f=\sum_{i=1}^l \sum_{j=0}^{m_i-1} c_{ij}(x-\xi_i)^{n-j}$
for some 
$c_{ij}\in F$, $i=1,\dots,l$ and $j=0,\dots,m_i-1$.
\item 
If $g=\prod_{i=1}^l(x-\xi_i)^{m_i}$, then the
 apolar invariant $\tv{f,g}_m=0$. (Note that $g$ is a poynomial of degree $m$.)
  \end{romenumerate}
\end{theorem}
If $\xi_i=\infty$, we interpret $(x-\infty)^{n-j}$ as $x^j$ in (i), as
above, and $(x-\infty)^{m_j}$ as $1$ in (ii).

\section{Invariants of polynomials of degree 1}\label{S1}

All seminvariants of a linear polynomial $a_0x+a_1$ are of the form
$c\,a_0^w$. In other words, \set{a_0} is a basis for the seminvariants. 

There are no invariants (except constants).
(The discriminant is trivially~$1$.)

\section{Invariants of polynomials of degree 2}\label{S2}

 We consider invariants \etc{} of a polynomial $f(x)=a_0x^2+a_1x+a_2$ of
 degree~2
(a \emph{quadratic} polynomial).

\subsection{Invariants}
The discriminant is, as is well-known and easily verified,
\begin{equation}\label{discr2}
  \gD(f)= a_1^2 -4a_0a_2.
\end{equation}
The discriminant is an invariant of 
degree $\nu=2$ and weight $w=2$.

The reduced form of $f$ is
\begin{equation}
\red f(x)\=
  f(x-a_1/2a_0)=a_0x^2-\frac{a_1^2 -4a_0a_2}{4a_0}
=a_0x^2-\frac{\gD}{4a_0};
\end{equation}
hence the only non-trivial coefficient of the reduced form is 
$\red a_2\=-\gD/4a_0$, so $a_0\red a_2$ is the invariant $-\gD/4$ of degree
$\nu=2$ and weight $w=2$, \cf{} \refT{Tred}.

The apolar invariant of $f$, see \refE{Eapolar1}, is
\begin{equation}\label{A2}
  A(f,f)=4a_0a_2-a_1^2=-\gD.
\end{equation}
This is another invariant of degree and weight $\nu=w=2$.

The Hankel determinant (catalecticant) of $f$, see \refE{EHankel}, is 
\begin{equation}\label{han2}
\Han(f)=
\begin{vmatrix}
  \ha_0&\ha_1\\\ha_1&\ha_2
\end{vmatrix}
=
\begin{vmatrix}
  a_0&\frac12a_1\\ \frac12 a_1&a_2
\end{vmatrix}
=a_0a_2-\frac14{a_1^2}=-\frac14\gD.
\end{equation}
Again, this is an 
invariant of degree $\nu=2$ and weight $w=2$.

The Hessian covariant is by \refE{EHessian} a covariant of order $2(\nn-2)=0$
for $\nn=2$, \ie, an  invariant. Thus, $H_0=H$.
We have, using \refE{EHessianp},
\begin{equation}\label{HD2}
  H(f)=H(f;x)
=4a_0(a_0x^2+a_1x+a_2)-(2a_0x+a_1)^2
=4a_0a_2-a_1^2
=-\gD.
\end{equation}
Once again, this is an invariant of degree $\nu=2$ and weight $w=2$.

Of course, these invariants are multiples of each other. In fact, as said
in \refE{Eapolar1} for any $\nn$, there is no other  invariants of
degree 2.
Moreover, $\gD$ is a basis for the invariants, \ie, every invariant is
$c\gD^\ell$ for some $c$ and $\ell$ 
\cite[\Satze 1.9 and 2.8]{Schur}.

\subsection{Seminvariants and covariants}
The leading coefficient $a_0$ is a seminvariant of degree 1 and weight 0.
It is the source of the covariant $f(x)$ of degree 1, order 2 and weight 0,
see Examples \refand{Ef}{Efs}.

\begin{theorem}[{\cite[Satz 2.16]{Schur}}]
  The covariants 
$f$ and $\gD$ form a basis of the covariants; thus \set{a_0,\gD} is a basis
of the seminvariants 
\cite[Satz 2.16]{Schur}.
\end{theorem}

Hence, the only seminvariant of degree $\nu$ and weight $w$ (up to constant
factors) is $a_0^{\nu-w}\gD^{w/2}$ provided $w$ is even and $\nu\ge w$; there
are no seminvariants for other $\nu$ and $w$.

As a further example,
the only non-trivial Gundelfinger covariant, see \refE{EGund}, is the invariant
$G_1(f)=H(f)=-\gD$; by \eqref{gund=hankel} we also have $G_1(f)=4\Han(f)$ in
accordance with \eqref{han2}.

\subsection{Seminvariants of $f'$}\label{SS2f'}

Since the only seminvariant of a linear function is the leading coefficient
$a_0$, the only seminvariant of $f'(x)=2a_0x+a_1$ is $2a_0$.

\subsection{The case $a_0=0$}\label{SS20}
When $a_0=0$, \ie, considering the restriction to polynomials of degree 1,
the essentially only non-trivial formula is $\gD(a_1x+a_2)=a_1^2$, or,
equivalently, 
\begin{equation}
  \gD\dgx2(a\dgxx01 x+a\dgxx11)=a\dgxx01^2,
\end{equation}
\cf{} \refE{Ediskr-1}.
In particular, for $f\in\cP_2$,
\begin{equation}
  \gD\dgx2(f')=4a_0^2.
\end{equation}

\subsection{Seminvariants and roots}\label{SS2roots}
By \refE{DD}, 
\begin{equation}
  \gD=a_0^2(\xi_1-\xi_2)^2.
\end{equation}
This agrees with \eqref{Hroots}, since $H_0=H=-\gD$ by \eqref{HD2}.

The general seminvariant of degree $\nu$ and weight $w$ is thus
\begin{equation}
  a_0^{\nu-w}\gD^{w/2}=a_0^{\nu}(\xi_1-\xi_2)^w,
\end{equation}
for $\nu\ge w$ and $w$ even (otherwise there is no such invariant).

\subsection{Further examples}\label{SS2further}

As examples of invariants of higher degree, we compute the basic invariants
(covariants, seminvariants) for $\nn=4$ (see \refS{S4} below) of $f^2$; these
are clearly invariants (\etc) of 
$f$ by \refT{Tcov2}:
\begin{align}
    A(f^2,f^2)&=4\gD^2,
\\
    I(f^2)&=\gD^2,
\\
    J(f^2)&=-2\gD^3,
\\
    \gD(f^2)&=0,
\\
  H_0(f^2)&=-12\,a_0^2\,\gD.
\\
  P(f^2)&=-4\,a_0^2\,\gD. \label{Pf22}
\\
  Q(f^2)&=0,
\\
  H(f^2)&=-12\,\gD\,f^2,
\\
  \JH(f^2)&=0.
\end{align}

\section{Invariants of polynomials of degree 3}\label{S3}

We consider invariants \etc{} of a polynomial $f(x)=a_0x^3+a_1x^2+a_2x+a_3$ of
degree 3 (a \emph{cubic} polynomial).

We give a table of covariants of low degree in \refT{T3c}, and the
corresponding seminvariants in \refT{T3s}, using notation introduced below.
(The tables give bases; further examples may be constructed by taking linear
combinations of the covariants (seminvariants) in each entry.)
It is easily checked that the dimensions agree with \refT{TGauss} (using for
example \cite[Table 14.3]{Andrews}).
The invariants have $w=3\nu/2$; these are all powers of $\gD$, and the only
example in the tables is $\gD$.

\begin{table}
\newcommand\x{\\[1pt]\hline}
  \begin{tabular}{l|l|l|l|l|l|l|l|l|l|l|}
	&0&1&2&3&4&5&6&7&8&9 \\
\hline
1&$a_0$&&&&&&&&& \x
2&$a_0^2$&&$P$&&& &&&&\x
3& $a_0^3$ && $a_0P$ & $Q$ &&&&&&\x
4& $a_0^4$ && $a_0^2P$ & $a_0Q$ & $P^2$ && $\gD$ &&&\x
5& $a_0^5$ && $a_0^3P$ & $a_0^2Q$ & $a_0P^2$ & $PQ$ & $a_0\gD$ &&&\x
6& $a_0^6$& & $a_0^4P$ & $a_0^3Q$& $a_0^2P^2$& $a_0PQ$ & $a_0^2\gD,P^3;Q^2$ &
& $\gD P$ & \x
7& $a_0^7$& & $a_0^5P$ & $a_0^4Q$& $a_0^3P^2$& $a_0^2PQ$ & 
$a_0^3\gD,a_0P^3;a_0Q^2$ & $P^2Q$ & $a_0\gD P$ & $\gD Q$ \x
  \end{tabular}
\caption{Invariants and seminvariants of low degree
of cubic polynomials. 
Each entry gives either a basis for the linear space of seminvariants of
given degree (row) and weight (column), or a basis separated by a semicolon
from further examples of such seminvariants.}
\label{T3c}
\end{table}

\begin{table}
\newcommand\x{\\[1pt]\hline}
  \begin{tabular}{l|l|l|l|l|l|l|l|l|l|l|}
	&0&1&2&3&4&5&6&7&8&9 \\
\hline
1&$f$&&&&&&&&& \x
2&$f^2$&&$H$&&& &&&&\x
3& $f^3$ && $fH$ & $G$ &&&&&&\x
4& $f^4$ && $f^2H$ & $fG$ & $H^2$ && $\gD$ &&&\x
5& $f^5$ && $f^3H$ & $f^2G$ & $fH^2$ & $HG$ & $f\gD$ &&&\x
6& $f^6$& & $f^4H$ & $f^3G$& $f^2H^2$& $fHG$ & $f^2\gD,H^3;G^2$ &
& $\gD H$ & \x
7& $f^7$& & $f^5H$ & $f^4G$& $f^3H^2$& $f^2HG$ & 
$f^3\gD,fH^3;fG^2$ & $H^2G$ & $f\gD H$ & $\gD G$\x
  \end{tabular}
\caption{Invariants and covariants of low degree
of cubic polynomials. 
Each entry gives either a basis for the linear space of covariants of
given degree (row) and weight (column), or a basis separated by a semicolon
from further examples of such covariants.}
\label{T3s}
\end{table}

\subsection{Invariants}
The discriminant is, see \eg{} \cite{SJN5},
\begin{equation}\label{discr3}
\begin{split}
  \gD(f)
&= a_1^2a_2^2-4 a_1^3a_3
-4 a_0 a_2^3 +18a_0a_1a_2a_3
-27 a_0^2a_3^2.
 \end{split}
\end{equation}
This is an invariant of degree $4$ and weight 6.

Different normalizations are sometimes used. We have
$\gD=D\scite{Schur}=-27 d\scite{Schur}=\tfrac{27}2R\scite{Glenn}$.

As for $\nn=2$,
$\gD$ is a basis for the  invariants, \ie, every
 invariant is $c\gD^\ell$ for some $c$ and $\ell$
\cite[Satz 2.8]{Schur}.

The apolar invariant $A(f,f)$ vanishes since $\nn=3$ is odd.

\subsection{Reduced form}
The reduced form of  $f$ is
\begin{equation}\label{red3}
\red f(x)
=a_0x^3+px+q
\=f\Bigpar{x-\frac{a_1}{3a_0}},
\end{equation}
which yields
\begin{align}\label{p3}
  p&\=\frac{3a_0a_2-a_1^2}{3a_0},
\\
q&\=\frac{2 a_1^3  + 27 a_0^2 a_3 - 9 a_0 a_1 a_2}{27a_0^2}.
\label{q3}
\end{align}

In terms of the coefficients of the reduced polynomial $\red
f(x)=a_0x^3+px+q$,
the discriminant is given by
\begin{equation}\label{gdpq}
\gD(f)=\gD(\red f)
= -4 a_0p^3-27a_0^2q^2
\end{equation}

\subsection{Seminvariants}
The coefficients $p$ and $q$ in \eqref{red3} are rational seminvariants by
\refT{Tred}. 
We conventionally denote the numerators in \eqref{p3} and \eqref{q3} by $P$
and $Q$ and have thus the seminvariants
\begin{align}\label{P3}
  P&\=3a_0a_2-a_1^2,
\\
Q&\=2 a_1^3  + 27 a_0^2 a_3 - 9 a_0 a_1 a_2.
\label{Q3}
\end{align}
$P$ has degree $2$ and weight $2$; $Q$ has degree 3 and weight 3.
Conversely, we have
\begin{align}\label{p3P}
  p&=\frac{P}{3\,a_0},
\\
q&=\frac{Q}{27\,a_0^2}.
\label{q3Q}
\end{align}
(Other notations: 
$P=-P\scite{Cred}$, with opposite sign;
$Q=U\scite{Cred}$.)

By \eqref{gdpq} and \eqref{p3P}--\eqref{q3Q},
the discriminant is given by
\begin{equation}
\gD
=-\frac{4 P^3}{27\,a_0^2}-\frac{Q^2}{27\,a_0^2}.
\end{equation}
Hence, the relation 
(\emph{syzygy}) 
\begin{equation}\label{gDPQ} 
  27\,a_0^2\,\gD = -4P^3-Q^2.
\end{equation}

\subsection{Covariants}
The form $f$ itself is a covariant of degree 1, weight 0 and order 3, see
\refE{Ef}. 

The Hessian covariant is the polynomial of degree $2(\nn-2)=2$ given by
\eqref{hessianp}, which yields
\begin{equation}\label{H3}
H(f;x)=
\left( 12\,{a_0}\,{a_2}-4\,a_1^{2} \right) {x}^{2}
+ \left( 36\,{a_0}\,{a_3}-4\,{a_1}\,{a_2} \right) x
+12\,{a_1}\,{a_3}-4\,a_2^{2}.
\end{equation}
This is a covariant of degree 2, weight 2 and order 2.

The Hessian source $H_0$ (\refE{EHessians})
is thus the seminvariant of degree 2 and weight~2
\begin{equation}\label{HP3}
H_0=
12\,{a_0}\,{a_2}-4\,a_1^{2}
=12 a_0\, p
= 4P.
\end{equation}

Conversely, $P$ is the source of the covariant
  \begin{equation}\label{tH3}
	\begin{split}
\tH(x)&=\tH(f;x)
\=
  \frac14 H(f;x)
\\&
=
 \left( 3\,{a_0}\,{a_2}-a_1^{2} \right) {x}^{2}
+ \left( 9\,{a_0}\,{a_3}-{a_1}\,{a_2} \right) x
+3\,{a_1}\,{a_3}-a_2^{2}.
	\end{split}
  \end{equation}
(Other notations: $H(X)\scite{Cred}=-\tH(X)$, so
$H(X)\scite{Cred}$ has source $-P=P\scite{Cred}$.
Further, 
$H=18\gD\scite{Glenn}=36h\scite{Schur}$;
$\tH=6\gD\scite{Glenn}=9h\scite{Schur}$.)

The only non-trivial Gundelfinger covariant, see \refE{EGund}, is
$G_1(f)=H(f)$.

The Jacobian (see \refE{EJacobian}) of $f(x)$ and $H(f;x)$ is
\begin{equation}
  \begin{split}
J(f,H(f))
&=
 \left( 108\,a_0^{2}\,{a_3}-36\,{a_0}\,{a_1}\,{a_2}
       +8\,a_1^{3} \right) {x}^{3}
\\&\qquad
+ \left( 108\,{a_0}\,{a_1}\,{a_3}-72\,{a_0}\,a_2^{2}
        +12\,a_1^{2}\,{a_2} \right) {x}^{2}
\\&\qquad
+ \left( -108\,{a_0}\,{a_2}\,{a_3}+72\,a_1^{2}\,{a_3}
        -12\,{a_1}\,a_2^{2} \right) x
\\&\qquad
+ 36\,{a_1}\,{a_2}\,{a_3}-8\,a_2^{3}
        -108\,{a_0}\,a_3^{2}.
  \end{split}
\end{equation}
This is, by \refT{Tcov2}, a covariant, which has degree 3, order 3 and
weight 3. Its source is
\begin{equation}
108\,a_0^{2}\,{a_3}-36\,{a_0}\,{a_1}\,{a_2}+8\,a_1^{3} 
=4Q.  
\end{equation}
Conversely, the covariant corresponding to the seminvariant $Q$ is
\begin{equation}\label{G3}
  \begin{split}
G(x)\=
\frac14J(f,H(f))
&=
\left( 27\,a_0^{2}\,{a_3}-9\,{a_0}\,{a_1}\,{a_2}
  +2\,a_1^{3} \right) {x}^{3}
\\&\qquad
+ \left( 27\,{a_0}\,{a_1}\,{a_3}-18\,{a_0}\,a_2^{2}
  +3\,a_1^{2}\,{a_2} \right) {x}^{2}
\\&\qquad
+ \left( -27\,{a_0}\,{a_2}\,{a_3}+18\,a_1^{2}\,{a_3}
  -3\,{a_1}\,a_2^{2} \right) x
\\&\qquad
+9\,{a_1}\,{a_2}\,{a_3}-2\,a_2^{3}-27\,{a_0}\,a_3^{2}
  \end{split}
\end{equation}
(Other notations: $G(x)=G(x)\scite{Cred}=27\,Q\scite{Glenn}=27\,
j\scite{Schur}$; $T\scite{KR}=J(H(f),f)=-4G$.) 

The relation \eqref{gDPQ} corresponds to the similar relation (syzygy)
between the corresponding covariants
\begin{equation}
    27f(x)^2\gD = -4(H(f;x)/4)^3-G(x)^2
\end{equation}
or
\begin{equation}\label{syz3}
    432\,\gD\,f(x)^2 +H(f;x)^3+16\,G(x)^2
=
  432\,\gD\,f(x)^2 +H(f;x)^3+J(f,H(f))^2
=0
\end{equation}

\begin{theorem}[{\cite[Satz 2.24]{Schur}}]
  The covariants \set{f,H,G,\gD} form a basis of all covariants for cubic
  polynomials. Equivalently,
\set{a_0,P,Q,\gD} is a basis of all seminvariants.
\end{theorem}
The basis is not algebraically independent since we have the syzygy
\eqref{syz3}, \ie{} 
$432\gD f^2+H^3+16G^2=0$,
or, equivalently, \eqref{gDPQ}.

\subsection{Seminvariants of $f'$}\label{SS3f'}

The discriminant of the quadratic polynomial $f'$ is a seminvariant by
\refT{Tdiff}; it is given by
\begin{equation}\label{gdf'3}
  \gD\dgx2(f')=\gD\dgx2(3a_0x^2+2a_1x+a_2)=4\,a_1^2-12\,a_0\,a_2=-4P.
\end{equation}
This has, \cf{} \refR{Rdiff},
degree 2 and weight 2 as the discriminant for $\nn=2$, see \refS{S2}; 
its order is 2.

Alternatively,
by \refE{EdiffHess} and \eqref{HP3}, we have 
\begin{equation}
  H_0(f')=H_0(f)=4P.
\end{equation}
Since $\gD=-H=-H_0$ for a quadratic polynomial, see \eqref{HD2}, we obtain
\eqref{gdf'3}.

\subsection{The case $a_0=0$}\label{SS30}
When $a_0=0$, \ie, considering the restriction to polynomials of degree 2,
we have, \cf{} \refE{Ediskr-1},
\begin{align}
  \gD\dgx3(a\dgxx02 x^2+a\dgxx12 x+a\dgxx22)&=a\dgxx02^2\gD\dgx2,
\\
  P\dgx3(a\dgxx02 x^2+a\dgxx12 x+a\dgxx22)&=-a\dgxx02^2,
\\
  Q\dgx3(a\dgxx02 x^2+a\dgxx12 x+a\dgxx22)&=2\,a\dgxx02^3.
\end{align}
In particular, for $f\in\cP_3$, using \eqref{gdf'3},
\begin{align}
  \gD\dgx3(f')&=9\,a_0^2\,\gD\dgx2(f')=-36\,a_0^2\,P,
\\
  P\dgx3(f')&=-9\,a_0^2,
\\
  Q\dgx3(f')&=54\,a_0^3.
\end{align}

\subsection{Seminvariants and roots}\label{SS3roots}
By \refE{DD}, 
\begin{equation}
  \gD=a_0^4(\xi_1-\xi_2)^2(\xi_1-\xi_3)^2(\xi_2-\xi_3)^2.
\end{equation}

By \eqref{Hroots}, 
\begin{equation}\label{h0r3}
  \begin{split}
  H_0&=-2\,a_0^2\bigpar{(\xi_1-\xi_2)^2+(\xi_1-\xi_3)^2+(\xi_2-\xi_3)^2}\\
	&=
-4\,a_0^{2} 
\left( \xi_1^{2}+\xi_2^{2}+\xi_3^{2}
-{\xi_1}\,{\xi_2}-{\xi_1}\,{\xi_3}-{\xi_2}\,{\xi_3} \right), 
  \end{split}
\end{equation}
and thus, by \eqref{HP3},
\begin{equation}
  \begin{split}
  P&=-\frac{a_0^2}2\bigpar{(\xi_1-\xi_2)^2+(\xi_1-\xi_3)^2+(\xi_2-\xi_3)^2}\\
	&=
-a_0^{2} 
\left( \xi_1^{2}+\xi_2^{2}+\xi_3^{2}
-{\xi_1}\,{\xi_2}-{\xi_1}\,{\xi_3}-{\xi_2}\,{\xi_3} \right).
  \end{split}
\end{equation}

Further, by a calculation or from \eqref{q3Q} and \eqref{red3},
noting that $\fr$ has roots $\xi_i-(\xi_1+\xi_2+\xi_3)/3$,
\begin{equation}\label{qr3}
  \begin{split}
	Q&=
-a_0^{3} 
\left( 2\,{\xi_1}-{\xi_2}-{\xi_3} \right)  
\left( 2\,{\xi_2}-{\xi_1}-{\xi_3} \right)  
\left( 2\,{\xi_3}-{\xi_1}-{\xi_2} \right)  
\\&=
-a_0^{3} \bigl(
2\,\xi_1^{3}+2\,\xi_2^{3}+2\,\xi_3^{3}
-3\,{\xi_1}\xi_2^{2}-3\,{\xi_1}\xi_3^{2}
-3\,\xi_1^{2}{\xi_2}-3\,\xi_1^{2}{\xi_3}
\\&\qquad\qquad
-3\,{\xi_2}\xi_3^{2}-3\,\xi_2^{2}{\xi_3}
+12{\xi_1}{\xi_2}{\xi_3}
\bigr).
  \end{split}
\end{equation}

\subsection{Covariants and roots}
For the corresponding covariants we have first
by \refE{EHessroots}, \cf, \eqref{h0r3},
\begin{equation}
  H(f;x)=-2a_0^2\bigpar{(\xi_1-\xi_2)^2(x-\xi_3)^2 +(\xi_1-\xi_3)^2(x-\xi_2)^2
+(\xi_2-\xi_3)^2(x-\xi_1)^2}.
\end{equation}

For $G$ we use \eqref{qr3} and \refT{Tcovroots}. We have 
\begin{equation*}
  \xi_1\xi_2\xi_3(2\xi_1\qw-\xi_2\qw-\xi_3\qw)=
2\xi_2\xi_3-\xi_1\xi_3-\xi_1\xi_2
=\xi_2(\xi_3-\xi_1)+\xi_3(\xi_2-\xi_1)
\end{equation*}
which after the substitution $\xi_i\mapsto x-\xi_i$ and permutation of the
indices leads to 
{\multlinegap=0pt
\begin{multline}
G(f;x)=-a_0^3	
\Bigpar{(x-\xi_2)(\xi_1-\xi_3)+(x-\xi_3)(\xi_1-\xi_2)}\
\\
\cdot
\Bigpar{(x-\xi_1)(\xi_2-\xi_3)+(x-\xi_3)(\xi_2-\xi_1)}
\Bigpar{(x-\xi_1)(\xi_3-\xi_2)+(x-\xi_2)(\xi_3-\xi_1)}
.
\end{multline}
}

\subsection{Further examples}\label{SS3further}
The apolar invariant of the Hessian covariant is an invariant 
given by, see \eqref{H3} and \eqref{A2},
\begin{equation}\label{AHH3}
  \begin{split}
  A(H(f),H(f))&
=
4\, \left( 12\,{a_0}\,{a_2}-4\,a_1^{2} \right)  
\left( 12\,{a_1}\,{a_3}-4\,a_2^{2} \right) 
- \left( 36\,{a_0}\,{a_3}-4\,{a_1}\,{a_2} \right) ^{2}
\\
&=
-1296\,a_0^{2}\,a_3^{2}
+864\,{a_0}\,{a_1}\,{a_2}\,{a_3}
-192\,{a_0}\,a_2^{3}
-192\,a_1^{3}\,{a_3}
+48\,a_1^{2}\,a_2^{2}
\\&= 48\, \gD.
  \end{split}
\raisetag{\baselineskip}
\end{equation}
This has degree 4 and weight 6.

The apolar invariant of the 6th degree polynomial $f^2$ is
\begin{equation}
  \begin{split}
  A(f^2,f^2)
&=
1296\,a_0^{2}\,a_3^{2}
-864\,{a_0}\,{a_1}\,{a_2}\,{a_3}
+192\,{a_0}\,a_2^{3}
+192\,a_1^{3}\,{a_3}
-48\,a_1^{2}\,a_2^{2}
\\&
=-48\,\gD.
  \end{split}
\raisetag\baselineskip
\end{equation}

Similarly, 
the apolar invariant of the 12th degree polynomial $f^4$ is
\begin{equation}
  \begin{split}
  A(f^4,f^4)
&=1244160\,\gD^2
=
2^{10}\cdot3^5\cdot5\cdot
\gD^2.
  \end{split}
\end{equation}

Recall that every invariant is a constant times a power of $\gD$, so these
formulas are no surprises.

The discriminant of the quadratic covariant $H(x)$ is
\begin{equation}
  \begin{split}
\gD\dgx2(H(x))
&=
1296\,a_0^{2}\,a_3^{2}
-864\,{a_0}\,{a_1}\,{a_2}\,{a_3}
+192\,{a_0}\,a_2^{3}
+192\,a_1^{3}\,{a_3}
-48\,a_1^{2}\,a_2^{2}
\\&=
-48\gD,
  \end{split}
\raisetag\baselineskip
\end{equation}
\cf{} \eqref{A2} and \eqref{AHH3}.
Thus the covariant $\tH(x)$ in \eqref{tH3} corresponding to $P$ 
has discriminant $-3\gD$.

The discriminant and covariants $H$ and $G$ of the cubic covariant $G(x)$ are
\begin{align}
\gD(G(x))&=729\,\gD^3, \label{DG3}
\\
H(G(x))&=27\,\gD\,H(x),
\\
G(G(x))&=-729\,\gD^2\,f(x).
\end{align}

We calculate also the resultants of $f$, $H(f)$ and $G(f)$:
\begin{align}
  R\xpar{f,H}&=-64\gD^2, \\
  R\xpar{f,G}&=8\gD^3, \\
  R\xpar{H,G}&=-1728\gD^3, 
\end{align}
where the first also follows by \refE{EResHess}.

For the seminvariants in Examples \ref{Eredb}--\ref{Eredc}, we have,
recalling $\ar_2=p=P/3a_0$, $\ar_3=q=Q/27a_0^2$ and $\ar_4=0$, see
\eqref{red3} and \eqref{p3P}--\eqref{q3Q},
\begin{align}
  a_0^2S_2&=-\frac23 P 
, \label{s32}\\
  a_0^3S_3&=-\frac19Q,\label{s33}\\
  a_0^4S_4&=\frac29 P^2,\label{s34}
\intertext{and}
 a_0^2 \chi_2&=-\frac29 P,
\label{achi32}\\
a_0^3 \chi_3&=-\frac1{27}Q,
\label{achi33}\\
a_0^4  \chi_4&
=-\frac{2}{27} P^2.
\label{achi34}
\end{align}

\subsection{Vanishing invariants and covariants}\label{SS3=0}
\begin{theorem} Let $f$ be a polynomial of degree $3$.
  \begin{romenumerate}
  \item $\gD(f)=0$ if and only if $f$ has a double (or triple) root; \ie, if
	and only if it has a square factor.
  \item $H(f)=0$ if and only if $f$ has a triple root, \ie, if and only if
	$f(x)=c(x-x_0)^3$. 
  \item $G(f)=0$ if and only if $f$ has a triple root, \ie, if and only if
	$f(x)=c(x-x_0)^3$. 
  \end{romenumerate}
\end{theorem}
\begin{proof}
  Parts (i) and (ii) are Theorems \ref{TD=0} and \ref{THess=0}.
For (iii), suppose that $G=0$. By \eqref{DG3}, then $\gD=0$, so $f$ has a
double root $\xi$. By projective invariance, we may assume that $\xi=0$, so
$f(x)=a_0x^3+a_1x^2$. Then, by \eqref{G3}, $G(x)=2a_1^3\,x^3$, and thus
$a_1=0$ too, and $\xi=0$ is a triple root. 
(Alternatively, $G=0$ and $\gD=0$ imply $H=0$ by \eqref{syz3}, and we may
use (ii).)

The converse follows similarly from
\eqref{G3} and projective invariance, or from \eqref{syz3} and (i)+(ii).
\end{proof}

\subsection{Geometry of real cubics}\label{SS3geo}

Let $f$ be a real cubic, with $a_0\neq0$. Then $f$ has an inflection point
$(x_0,y_0)$ given by $0=f''(x_0)=6a_0 x_0+2a_1$, so $x_0=-a_1/3a_0$ and, using
\eqref{red3},
\begin{equation}\label{g3y0}
  y_0=f(x_0)=\red f(0)=q.
\end{equation}
Thus, by \eqref{q3Q},
\begin{equation}\label{g3inf}
  (x_0,y_0)
=\Bigpar{-\frac{a_1}{3a_0},\,q}
=\Bigpar{-\frac{a_1}{3a_0},\,\frac{Q}{27\,a_0^2}}.
\end{equation}
Note that $f$ is symmetric about $(x_0,y_0)$, \cf{} \eqref{red3}.

The extreme points $x_\pm$ are given by, using \eqref{red3} again,
\begin{equation}\label{g3ext}
  0=f'(x)={\red{f}}'(x-x_0)=3a_0(x-x_0)^2+p;
\end{equation}
hence, using also \eqref{p3P}, 
\begin{equation}
  \label{g3+-}
x_\pm=x_0\pm\sqrt{\frac{-p}{3a_0}}
=x_0\pm \frac{\sqrt{-P}}{3a_0}
= \frac{-a_1\pm\sqrt{-P}}{3a_0}.
\end{equation}
Consequently, $f$ has real (local) maximum and minimum points if $P<0$, but not
if $P\ge0$; in the latter case, $f$ is monotonously increasing (if $a_0>0$)
or decreasing (if $a_0<0$) on $(-\infty,\infty)$ .
(This includes the case $P=0$, when 
$f'(x_0)=f''(x_0)=0$.)

Moreover, the extreme values $y_\pm=f(x_\pm)$ are given by, using
\eqref{red3}, \eqref{g3+-} and \eqref{p3P}--\eqref{q3Q},
\begin{equation}\label{gf3+-}
  \begin{split}
y_\pm&\= f(x_\pm)=\red f(x_\pm-x_0)=a_0(x_\pm-x_0)^3+p(x_\pm-x_0)+q
\\&\phantom:
=a_0\lrpar{\frac{\pm\sqrt{-P}}{3a_0}}^3+p\,\frac{\pm\sqrt{-P}}{3a_0}+q
=\frac{\pm 2P\sqrt{-P}}{27a_0^2}+\frac{Q}{27a_0^2}	
\\&\phantom:
=\frac{Q\pm 2P\sqrt{-P}}{27a_0^2}.
  \end{split}
\end{equation}

In particular, we see that
$f$ has three distinct real roots\\\hbox{\qquad} 
$\iff$ $x_\pm$ are real and 
$y_-<0<y_+$ or $y_+<0<y_-$\\\hbox{\qquad} 
$\iff$ $P< 0$ and $|2P\sqrt{-P}| >|Q|$\\\hbox{\qquad} 
$\iff$ $-4P^3>Q^2$\\\hbox{\qquad} 
$\iff$ $\gD=-(4P^3+Q^2)/27a_0^2>0$.

Similarly, there is a real double root if $P<0$ and $\gD=0$, and a triple
root if $P=0=Q$. We thus have found the following
classical result, which also follows
directly from \eqref{dd}, see  \cite{SJN5,SJN7}:
\begin{theorem}\label{Tf3roots}
Let $f$ be a real cubic.
  \begin{romenumerate}
	\item If $\gD>0$, then $f$ has $3$ distinct real roots.
\item If $\gD=0$, then $f$ has either one double and one simple root, both
  real $(P<0)$, or a real triple root $(P=Q=0)$.
\item If $\gD<0$, then $f$ has one real root and a pair of two   (non-real) 
conjugate complex roots.
  \end{romenumerate}
\end{theorem}

\begin{remark}\label{RD}
More generally, it follows from \eqref{dd} that if $f$ is a real polynomial
of degree $n$ with only simple roots, having $n-2m$ real roots and $m$ 
pairs of conjugate complex (non-real) roots, then $\sign(\gD(f))=(-1)^m$.
\end{remark}

We also have, by \eqref{gf3+-} and \eqref{gDPQ}, the quantitative relation
\begin{equation}\label{g3y+-}
  y_+y_-=\frac{Q^2+4P^3}{729\,a_0^4}=-\frac{\gD}{27\,a_0^2}.
\end{equation}
In fact, since $\gD=-a_0\qw R(f,f')$, where $R$ is the resultant,
this follows immediately from a standard property of the resultant; 
more generally,
for a polynomial of arbitrary degree $n$, with stationary points (roots of
$f'$) $\eta_1,\dots,\eta_{n-1}$,
\begin{equation}\label{tdd}
\gD(f)
=(-1)^{n(n-1)/2}n^na_0^{n-1}\prod_{j=1}^{n-1}f(\eta_j),
\end{equation}
see \cite{SJN5}.

Note also the corresponding formula, by \eqref{g3+-} and \eqref{P3} or
directly from $f'(x)=3a_0x^2+2a_1x+a_2$,
\begin{equation}\label{g3x+-}
  x_+x_-=\frac{a_2}{3a_0}.
\end{equation}

We can further study the location of the roots. We have for example the
following criteria for positive roots.

\begin{theorem}\label{T3>0}
  Let $f$ be a real cubic with $a_0>0$.
  \begin{romenumerate}
  \item $f$ has three distinct positive roots in $(0,\infty)$\\ 
$\iff$ $\gD>0$ (which implies $P<0$), $a_3<0$ and $-a_1>\sqrt{-P}$ \\
$\iff$ $\gD>0$, $a_1<0$, $a_2>0$, $a_3<0$.
  \item $f$ has three  roots (not necessarily distinct) 
in $[0,\infty)$ \\
$\iff$ $\gD\ge0$ (which implies $P\le0$), $a_3\le0$ and $-a_1\ge\sqrt{-P}$ \\
$\iff$ $\gD\ge0$, $a_1\le0$, $a_2\ge0$, $a_3\le0$.
  \end{romenumerate}
\end{theorem}
\begin{proof}
  Consider for example (i). We may suppose that $f$ has three real roots, so
  $\gD>0$, and then $P<0$ by \eqref{gDPQ}. 
A geometric consideration shows that
the roots are all positive $\iff$ 
$x_\pm>0$ and $a_3=f(0)<0$, and the result follows by \eqref{g3+-} and
  \eqref{P3}. Case (ii) is similar, considering also cases with a double or
  triple root.
\end{proof}

Note that \eqref{sofie} immediately implies that if $\xi_1,\xi_2,\xi_3\ge0$,
and $a_0>0$, then $a_1\le0$, $a_2\ge0$, $a_3\le0$, but the converse is less
obvious. 

\section{Invariants of polynomials of degree 4}\label{S4}
We consider invariants \etc{}
of a polynomial $f(x)=a_0x^4+a_1x^3+a_2x^2+a_3x+a_4$ of
degree 4 (a \emph{quartic} polynomial).

We give a table of covariants of low degree in \refT{T4c}, and 
corresponding seminvariants in \refT{T4s}, using notation introduced below.
Again, it is easily checked that the dimensions agree with \refT{TGauss}
(using for example \cite[Table 14.3]{Andrews}).
(The examples given in the table is a rather arbitrary selection when the
dimension is $>1$. For example, note that when $\nu=w$, there is by
\refT{Tredall} always a
basis for the seminvariants
consisting of monomials in $P,Q,R$; for example, for $\nu=w=5$,
\set{P^3,Q^2,PR}.) 
The invariants have $w=2\nu$; for each such $\nu$ and $w$, there is a basis
consisting of monomials in $I$ and $J$ (but for $\nu=6$, $w=12$, \set{\gD,I^3}
is another example).

\begin{table}
\newcommand\x{\\[1pt]\hline}
  \begin{tabular}{l|l|l|l|l|l|l|l|l|l|}
	&0&1&2&3&4&5&6&7&8 \\
\hline
1&$f$&&&&&&&& \x
2&$f^2$&&$H$&&$I$& &&&\x
3& $f^3$ && $fH$ & $\JH$ & $fI$ & & $J$ &  &\x
4& $f^4$ && $f^2H$ & $f\JH$ & $f^2I, H^2$ & & $fJ,IH$ &  & $I^2$\x
5& $f^5$ && $f^3H$ & $f^2\JH$ & $f^3I, fH^2$ & $H\JH$ & $f^2J,fIH$
& $I\JH$ &  $fI^2,JH$ \x
6& $f^6$& & $f^4H$ & $f^3\JH$& $f^4I, f^2H^2$& $fH\JH$ &
$f^3J,f^2IH,H^3;\JH^2$ & $fI\JH$&$f^2I^2,fJH,IH^2$  \x
  \end{tabular}
\vskip\baselineskip
  \begin{tabular}{l|l|l|l|l|}
	&9&10&11&12 \\
\hline
5 && $IJ$ &&\x
6& $J\JH$& $fIJ,I^2H$ && $I^3,J^2;\gD$ \x
  \end{tabular}
\caption{Invariants and covariants of low degree
of quartic polynomials. 
Each entry gives either a basis for the linear space of covariants of
given degree (row) and weight (column), or a basis separated by a semicolon
from further examples of such covariants.}
\label{T4c}
\end{table}

\begin{table}
\newcommand\x{\\[1pt]\hline}
  \begin{tabular}{l|l|l|l|l|l|l|l|l|l|}
	&0&1&2&3&4&5&6&7&8 \\
\hline
1&$a_0$&&&&&&&& \x
2&$a_0^2$&&$P$&&$I$& &&&\x
3& $a_0^3$ && $a_0P$ & $Q$ & $a_0I$ & & $J$ &  &\x
4& $a_0^4$ && $a_0^2P$ & $a_0Q$ & $a_0^2I, P^2; R$ & & $a_0J,IP$ &  & $I^2$\x
5& $a_0^5$ && $a_0^3P$ & $a_0^2Q$ & $a_0^3I, a_0P^2$ & $PQ$ & $a_0^2J,a_0IP$
& $IQ$ &  $a_0I^2,JP$ \x
6& $a_0^6$& & $a_0^4P$ & $a_0^3Q$& $a_0^4I, a_0^2P^2$& $a_0PQ$ &
$a_0^3J,a_0^2IP,P^3;Q^2$ & $a_0IQ$&$a_0^2I^2,a_0JP,IP^2$  \x
  \end{tabular}
\vskip\baselineskip
  \begin{tabular}{l|l|l|l|l|}
	&9&10&11&12 \\
\hline
5 && $IJ$ &&\x
6& $JQ$& $a_0IJ,I^2P$ && $I^3,J^2;\gD$ \x
  \end{tabular}
\caption{Invariants and seminvariants of low degree
of quartic polynomials. 
Each entry gives either a basis for the linear space of seminvariants of
given degree (row) and weight (column), or a basis separated by a semicolon
from further examples of such seminvariants.}
\label{T4s}
\end{table}

\subsection{Invariants}
The discriminant is, see \cite{SJN5},
\begin{equation}\label{discr4}
\begin{split}
  \gD(f)
&= 
256\,a_0^{3}\,a_4^{3}
-192\,a_0^{2}\,{a_1}\,{a_3}\,a_4^{2}
-128\,a_0^{2}\,a_2^{2}\,a_4^{2}
\\&\qquad
+144\,a_0^{2}\,{a_2}\,a_3^{2}\,{a_4}
-27\,a_0^{2}\,a_3^{4}
+144\,{a_0}\,a_1^{2}\,{a_2}\,a_4^{2}
\\&\qquad
-6\,{a_0}\,a_1^{2}\,a_3^{2}\,{a_4}
-80\,{a_0}\,{a_1}\,a_2^{2}\,{a_3}\,{a_4}
+18\,{a_0}\,{a_1}\,{a_2}\,a_3^{3}
\\&\qquad
+16\,{a_0}\,a_2^{4}\,{a_4}
-4\,{a_0}\,a_2^{3}\,a_3^{2}
-27\,a_1^{4}\,a_4^{2}
\\&\qquad
+18\,a_1^{3}\,{a_2}\,{a_3}\,{a_4}
-4\,a_1^{3}\,a_3^{3}
-4\,a_1^{2}\,a_2^{3}\,{a_4}
+a_1^{2}\,a_2^{2}\,a_3^{2}
 .
 \end{split}
\raisetag{13pt}
\end{equation}
(See also \eqref{D4red} below.)
This is an invariant of degree 6 and weight 12.
(Other notations: 
$\gD\scite{Cred}=\gD\scite{Cclass}=27\gD$; 
$\gD_0{}\scite{Cred}=\gD$; 
$D\scite{Schur}=\gD$.)

There are simpler invariants, however.
The apolar invariant, see \refE{Eapolar1}, is 
\begin{equation}\label{A4}
A(f,f)=  
48\,{a_0}\,{a_4}-12\,{a_1}\,{a_3}+4\,a_2^{2}
=4I,
\end{equation}
where $I$ is the conveniently normalized invariant
\begin{equation}\label{I4}
I=12\,{a_0}\,{a_4}-3\,{a_1}\,{a_3}+a_2^{2}.  
\end{equation}
The apolar invariant and $I$ are invariants of degree 2 and weight 4.
(Other notations: $A = 4!\,A\scite{Schur} = 24\, A\scite{Schur}$;
$I = 6\,i\scite{Glenn}=12\, P\scite{Schur}$.)

The Hankel determinant (catalecticant), 
see \refE{EHankel}, is an invariant
of degree 3 and weight 6. It is, by a calculation, in our normalization,
\begin{equation}
  \begin{split}
\Han(f)
&=
\begin{vmatrix}
  \ha_0&\ha_1&\ha_2\\ \ha_1&\ha_2&\ha_3\\ \ha_2&\ha_3&\ha_4
\end{vmatrix}
=
\begin{vmatrix}
  a_0&\frac14a_1&\frac16a_2\\ 
  \frac14a_1&\frac16a_2&\frac14a_3\\ 
    \frac16a_2&\frac14a_3&a_4\\ 
\end{vmatrix}
\\&	=
\frac{
72\,{a_0}\,{a_2}\,{a_4}-27\,{a_0}\,a_3^{2}-27\,
a_1^{2}\,{a_4}+9\,{a_1}\,{a_2}\,{a_3}-2\,a_2^{3}
}
{432} 
\\&=
\frac{J}{432},
  \end{split}
\end{equation}
where we thus define
\begin{equation}\label{J4}
J\=
72\,{a_0}\,{a_2}\,{a_4}
-27\,{a_0}\,a_3^{2}
-27\,a_1^{2}\,{a_4}
+9\,{a_1}\,{a_2}\,{a_3}
-2\,a_2^{3}.
\end{equation}
$J$ is thus an invariant of degree 3 and weight 6.
(Other notation: $Q\scite{Schur}=\Han(f)=J\scite{Elliott}=J/432$;
$J\scite{Glenn}=J/72$.)

By \eqref{gund=hankel}, 
the second Gundelfinger covariant in \refE{EGund} is the invariant
\begin{equation}\label{gund24}
  G_2(f)=24^3\Han(f)=32\,J.
\end{equation}

Another way to construct $J$ is by taking
the joint apolar invariant $A(H(f),f)$; this invariant of degree 3 and
weight 6 equals 
$24 J$, see \eqref{AHf4}.

\begin{theorem}[{\cite[Satz 2.9]{Schur}}]\label{T4bas}
 $I$ and $J$ form a basis for the  invariants of quartic polynomials.
Furthermore, $I$ and $J$ are algebraically independent.
\end{theorem}

Thus, informally speaking, $I$ and $J$ are the only invariants. More
precisely, every invariant is an isobaric polynomial in $I$ and $J$.
For example, the discriminant is such a polynomial; a calculation reveals that
\begin{equation}\label{DIJ4}
  \gD=\frac{4}{27}I^3-\frac1{27}J^2.
\end{equation}

See \refSS{SS4further} for further examples.


\begin{example}
  \label{E4abs}
Since $I^3$ and $J^2$ both are invariants of degree 6 and weight 12, the
quotient $I^3/J^2$ is an absolute invariant. 
Similarly, $J^2/I^3$,
$I^3/\gD$, $J^2/\gD$, etc.\ are absolute invariants; these are all simple
rational functions of each other. In fact, since $I$ and $J$ form a basis
for the invariants, it is easy to see that
every absolute invariant is a rational function of
$I^3/J^2$, or of any other of the absolute invariants just given.
\end{example}

\subsection{Covariants}
The form $f$ itself is a covariant of degree 1, weight 0 and order 4, see
\refE{Ef}. 

The Hessian covariant is the polynomial of degree $2(\nn-2)=4$ given by
\eqref{hessianp}, which yields
  \begin{multline}\label{H4}
H(f;x)=
\left( 24\,{a_0}\,{a_2}-9\,a_1^{2} \right) {x}^{4}
+ \left( 72\,{a_0}\,{a_3}-12\,{a_1}\,{a_2} \right) {x}^{3}
\\+ \left( 144\,{a_0}\,{a_4}+18\,{a_1}\,{a_3}-12\,a_2^{2}\right) {x}^{2}
+ \left( 72\,{a_1}\,{a_4}-12\,{a_2}\,{a_3} \right) x
\\+(24\,{a_2}\,{a_4}-9\,a_3^{2})
.	
  \end{multline}
This is a covariant of degree 2, weight 2 and order 4.
We also define 
\begin{multline}\label{tH4}
  \tH(f;x)\=\frac13H(f;x)
=
 \left( 8\,{a_0}\,{a_2}-3\,a_1^{2} \right) {x}^{4}
+ \left( 24\,{a_0}\,{a_3}-4\,{a_1}\,{a_2} \right) {x}^{3}
\\
+\left(48\,{a_0}\,{a_4}+6\,{a_1}\,{a_3}-4\,a_2^{2}\right)
  {x}^{2}
+ \left( 24\,{a_1}\,{a_4}-4\,{a_2}\,{a_3} \right) x
+8\,{a_2}\,{a_4}-3\,a_3^{2}.
\end{multline}
$\tH$ too has degree 2, weight 2 and order~4.
(Other notations:
$g_4{}\scite{Cred}=g_4{}\scite{Cclass}=-\tH$;
$h\scite{Schur}=\tH/48=H/144$; $H\scite{Schur}=H$.)

The Jacobian determinant, see \refE{EJacobian}, of $f$ and $H(f)$ is a
covariant of order $4+4-2=6$ given by
\begin{equation}\label{JH4}
  \begin{split}
\JH(f)&=	
 \left( 288\,a_0^{2}\,{a_3}-144\,{a_0}\,{a_1}\,{a_2}
  +36\,a_1^{3} \right) {x}^{6}
\\&\qquad{}
+ \left( 1152\,a_0^{2}\,{a_4}+144\,{a_0}\,{a_1}\,{a_3}
  -288\,{a_0}\,a_2^{2}+72\,a_1^{2}\,{a_2} \right){x}^{5}
\\&\qquad{}
+ \left( 1440\,{a_0}\,{a_1}\,{a_4}-720\,{a_0}\,{a_2}\,{a_3}
  +180\,a_1^{2}\,{a_3} \right) {x}^{4}
\\&\qquad{}
 +\left(-720\,{a_0}\,a_3^{2}+720\,a_1^{2}\,{a_4}\right){x}^{3}
\\&\qquad{}
+ \left( -1440\,{a_0}\,{a_3}\,{a_4}+720\,{a_1}\,{a_2}\,{a_4}
  -180\,{a_1}\,a_3^{2} \right) {x}^{2}
\\&\qquad{}
+ \left( -1152\,{a_0}\,a_4^{2}-144\,{a_1}\,{a_3}\,{a_4}
  +288\,a_2^{2}\,{a_4}-72\,{a_2}\,a_3^{2} \right) x
\\&\qquad{}
+144\,{a_2}\,{a_3}\,{a_4}-36\,a_3^{3}
  -288\,{a_1}\,a_4^{2}.
  \end{split}
\end{equation}
We normalize this to $\tJH(f;x)\=\JH(f;x)/36$, where thus
\begin{equation}\label{tJH4}
  \begin{split}
\tJH(f)&=	
 \left( 8\,a_0^{2}\,{a_3}-4\,{a_0}\,{a_1}\,{a_2}
  +a_1^{3} \right) {x}^{6}
\\&\qquad
+ \left( 32\,a_0^{2}\,{a_4}+4\,{a_0}\,{a_1}\,{a_3}
  -8\,{a_0}\,a_2^{2}+2\,a_1^{2}\,{a_2} \right) {x}^{5}
\\&\qquad
+ \left( 40\,{a_0}\,{a_1}\,{a_4}-20\,{a_0}\,{a_2}\,{a_3}
  +5\,a_1^{2}\,{a_3} \right) {x}^{4}
\\&\qquad
+ \left(-20\,{a_0}\,a_3^{2}+20\,a_1^{2}\,{a_4} \right){x}^{3}
\\&\qquad
+ \left( -40\,{a_0}\,{a_3}\,{a_4}+20\,{a_1}\,{a_2}\,{a_4}
   -5\,{a_1}\,a_3^{2} \right) {x}^{2}
\\&\qquad
+ \left( -32\,{a_0}\,a_4^{2}-4\,{a_1}\,{a_3}\,{a_4}
  +8\,a_2^{2}\,{a_4}-2\,{a_2}\,a_3^{2} \right) x
\\&\qquad
+4\,{a_2}\,{a_3}\,{a_4}-a_3^{3}-8\,{a_1}\,a_4^{2}
  \end{split}
\end{equation}
(Other notations:
$g_6{}\scite{Cred}=g_6{}\scite{Cclass}=\tJH$;
$j\scite{Schur}=\tJH/32=\JH/1152$.)
$\JH$ and $\tJH$ have degree 3, weight 3 and order 6.

\begin{theorem}[{\cite[Satz 2.25]{Schur}}]\label{Tcov4}
  The invariants $I$ and $J$ and the covariants $f$, $H$ and $\JH$ form a
  basis for the covariants of quartic polynomials.
\end{theorem}

The basic covariants satisfy the relation ({syzygy})
\begin{equation}\label{syz41}
  \tH^3-48\, If^2\tH+64\,Jf^3+27\,\tJH^2=0.
\end{equation}
or 
\begin{equation}\label{syz42}
  2^4H^3-2^8 3^3 If^2\,H+2^{10}3^3Jf^3+3^2\JH^2=0.
\end{equation}

The only non-trivial Gundelfinger covariants are $G_1(f)=H(f;x)$ and
$G_2(f)=24^3\Han(f)=32\,J$, see
 \eqref{gund24}.

\subsection{Reduced form}
The reduced form of  $f$ is
\begin{equation}\label{red4}
\red f(x)
=a_0 x^4+px^2+qx+r
\=f\Bigpar{x-\frac{a_1}{4a_0}};
\end{equation}
thus
$p\=\ar_2$, $q\=\ar_3$, $r\=\ar_4$.
These rational seminvariants are given by
\begin{align}\label{p4}
  p&={\frac {8\,{a_0}\,{a_2}-3\,a_1^{2}}
{8\,{a_0}}}
\\\label{q4}
q&={\frac 
{8\,a_0^{2}\,{a_3}-4\,{a_0}\,{a_1}\,{a_2} +a_1^{3}}
{8\,a_0^{2}}}
\\\label{r4}
r&={\frac 
 {256\,a_0^{3}\,{a_4}
 -64\,a_0^{2}\,{a_1}\,{a_3}
 +16\,{a_0}\,a_1^{2}\,{a_2}
 -3\,a_1^{4}}
{256a_0^{3}}}
\end{align}

In terms of the coefficients of the reduced polynomial $\red f$, the
discriminant is given by
\begin{equation}\label{D4red}
\begin{split}
\gD(f)
=
\gD(\red f)
&=
-4\,a_0{p}^{3}{q}^{2}
+16\,a_0{p}^{4}r
-27\,a_0^2{q}^{4}
\\&\qquad
+144a_0^2\,p{q}^{2}r
-128\,a_0^2{p}^{2}{r}^{2}
+256\,a_0^3{r}^{3}
.
 \end{split}
\end{equation}

\subsection{Seminvariants}
We denote the numerators of \eqref{p4}--\eqref{r4} by 
$P$, $Q$, $R$, respectively,
and have thus
\begin{align}
  p&={\frac P {8\,{a_0}}},\label{p4P}
\\
q&={\frac Q {8\,a_0^{2}}},\label{q4Q}
\\
r&={\frac R{256a_0^{3}}}, \label{r4R}
\intertext{with}
  P&\= {8\,{a_0}\,{a_2}-3\,a_1^{2}},\label{P4}
\\
Q&\= {8\,a_0^{2}\,{a_3}-4\,{a_0}\,{a_1}\,{a_2}
+a_1^{3}},
\label{Q4}
\\
R&\= {256\,a_0^{3}\,{a_4}-64\,a_0^{2}\,{a_1}\,{a_3}
+16\,{a_0}\,a_1^{2}\,{a_2}-3\,a_1^{4}}.
\label{R4}
\end{align}
($R$ should not be confused with the resultant in \refS{Sroots}.
Other notations: 
$P=H\scite{Cred}=-p\scite{Cclass}$;
$Q=R\scite{Cred}=r\scite{Cclass}$.)
These are seminvariants of degree and weight $(2,2)$, $(3,3)$ and $(4,4)$.
We have
\begin{equation}\label{RIP4}
  R=\frac13\bigpar{64\,a_0^2\,I-P^2}.
\end{equation}

The Hessian source $H_0$ is by \eqref{H4} the seminvariant of degree 2
and weight~4 
\begin{equation}\label{HP4}
H_0= 24\,{a_0}\,{a_2}-9\,a_1^{2}
=3\,P.
\end{equation}
Thus the source of the covariant $\tH$ is  $P$.

By \eqref{JH4}--\eqref{tJH4} and \eqref{Q4},
$Q$ is the source of $\tJH$, while the source of $\JH$
is $36\,Q$.

\refT{Tcov4} and the syzygy \eqref{syz41} translate to the following.

\begin{theorem}\label{Tsem4}
  The invariants $I$ and $J$ and the seminvariants $a_0$, $P$ and $Q$ form a
  basis for the seminvariants of quartic polynomials. These satisfy
the syzygy
\begin{equation}\label{syz41s}
  P^3-48\, I\,a_0^2\,P+64\,J\,a_0^3+27\,Q^2=0.
\end{equation}
\end{theorem}

\subsection{Cubic resolvent}
Let $\tp\=p/a_0$, $\tq\=q/a_0$, $\tr\=r/a_0$, the coefficients of the
reduced monic polynomial $\fr/a_0$.
The  \emph{cubic resolvent} of $f$ is the cubic polynomial
\begin{equation}\label{Res4}
  \begin{split}
  \Res(f;x)&\=x^3+2\tp x^2+(\tp^2-4\tr)x-\tq^2\\
&\phantom:
=x^3+\frac{P}{4\,a_0^2}x^2+\frac{P^2-R}{64\,a_0^4}x-\frac{Q^2}{64\,a_0^6},
\end{split}
\end{equation}
see \eg{} \cite{SJN7}.
The numerator $P^2-R$ is a seminvariant of degree and weight 4, and we have
by \eqref{RIP4}
\begin{equation}\label{P2-R}
  \frac{P^2-R}4
=
\frac{P^2-16\,a_0^2\,I}{3}
=
-64\,a_0^{3}\,{a_4}+16\,a_0^{2}\,{a_1}\,{a_3}
+16\,a_0^{2}\,a_2^{2}-16\,{a_0}\,a_1^{2}\,{a_2}
+3\,a_1^{4}.
\end{equation}
This seminvariant is used in
\cite{Cred,Cclass} with the notations 
\begin{equation}\label{P2-Rb}
Q\scite{Cred}=q\scite{Cclass}\=
\frac13\bigpar{P^2-16\,a_0^2\,I}
=  \frac{P^2-R}4.
\end{equation}

The reduced form of the cubic resolvent is, after some calculations,
\begin{equation}\label{tRes4}
  \begin{split}
\tRes(f;x)\=
\Res\Bigpar{f;x-\frac{P}{12\,a_0^2}}	
=x^3-\frac{I}{3\,a_0^2}\,x
+\frac{J}{27\,a_0^3}.
  \end{split}
\end{equation}
Changing the variable to clear the denominators, we find
\begin{equation}
  \begin{split}\label{hRes4}
(3\,a_0)^3\,\tRes(f;x/3a_0)
=
27\,a_0^3\,\Res\Bigpar{f;\frac{4a_0\,x-P}{12\,a_0^2}}	
=x^3-3I\,x+J.
  \end{split}
\end{equation}
Thus the cubic polynomial $\Resx(f;x)\=x^3-3I\,x+J$ is also a form of the
resolvent. 

\begin{remark}\label{Rgam}
The roots 
of the cubic resolvent $\Res(f)$ are $\gam_1^2$, $\gam_2^2$, $\gam_3^2$, where 
\begin{align}
  \gam_1&\=\tfrac12(\xi_1+\xi_2-\xi_3-\xi_4),
\label{gam1}\\
  \gam_2&\=\tfrac12(\xi_1-\xi_2+\xi_3-\xi_4),
\label{gam2}\\
  \gam_3&\=\tfrac12(\xi_1-\xi_2-\xi_3+\xi_4),
\label{gam3}
\end{align}
The quartic equation $f(x)=0$ can thus be solved by finding the roots of
$\Res(f)$, 
taking the square roots to find $\gam_1,\gam_2,\gam_3$, 
with the signs satisfying $\gam_1\gam_2\gam_3=-\tq$, and finally inverting
\eqref{gam1}--\eqref{gam3} together with $\xi_1+\xi_2+\xi_3+\xi_4=-\tp$, see
\cite{SJN7}.  Alternatively,
one can first
find the roots of $\tRes(f)$ or
$\Resx(f)$; for example, if the roots of $\Resx(f)$ are $z_1,z_2,z_3$, we
take $\gam_i=\pm\frac12a_0\qw\sqrt{(4a_0 z_i-P)/3}$.
Equivalently, the roots of $\Resx(f)$ are 
\begin{equation}
z_i\=  3a_0\gam_i^2+\frac{P}{4a_0}
=  3a_0\gam_i^2+2p,
\qquad i=1,2,3,
\end{equation}
while the roots of  $\tRes(f)$ are 
\begin{equation}
\frac{z_i}{3a_0}= \gam_i^2+\frac{P}{12\,a_0^2}
=  \gam_i^2+\frac{2p}{3a_0},
\qquad i=1,2,3.
\end{equation}
\end{remark}

\begin{remark}
  Another common version of the cubic resolvent is (see \cite{SJN7})
  \begin{equation}
	\begin{split}
\Resy(f; x)&\=\Res\Bigpar{f;x-\tp-\frac{a_1^2}{8a_0^2}}
=\Res\Bigpar{f;x-\frac{4a_0a_2-a_1^2}{4a_0^2}}
\\&\phantom:
=x^3-\frac{a_2}{a_0}x^2+\frac{a_1a_3-4a_0a_2}{a_0^2}x
+\frac{4a_0a_2a_4-a_0a_3^2-a_1^2a_4}{a_0^3}	  
.
	\end{split}
  \end{equation}
This has the roots $\xi_1\xi_2+\xi_3\xi_4$, $\xi_1\xi_3+\xi_2\xi_4$ and
$\xi_1\xi_4+\xi_2\xi_3$. However, these roots are \emph{not} translation
invariant, so the coefficients of $\Resy$ are \emph{not} seminvariants.
\end{remark}

We have, by \eqref{tRes4} and \eqref{gdpq}, \eqref{DIJ4} and \eqref{dd},
  \begin{equation}\label{gDres4}
	\gD\dgx3(\Res(f))=
	\gD\dgx3(\tRes(f))=
\frac{4I^3}{27a_0^6}
-\frac{J^2}{27a_0^6}
=a_0^{-6}\gD=\gD_0.
  \end{equation}
Further, by \eqref{tRes4} and \eqref{P3} or \eqref{p3P},
  \begin{align}
	P\dgx3(\Res(f))&=
	P\dgx3(\tRes(f))
=-a_0^{-2}I, \label{Pres4}
\\
Q\dgx3(\Res(f))&=
Q\dgx3(\tRes(f))
=a_0^{-3}J. \label{Qres4}
  \end{align}
For the version $\Resx(f)=x^3-3I\,x+J$ we have, directly from
\eqref{gdpq}--\eqref{q3Q}, the corresponding
  \begin{align}
\gD\dgx3(\Resx(f))&=4\cdot27\,I^3-27\,J^2=3^6\gD, \label{gDhres4}
\\
P\dgx3(\Resx(f))&=-9\,I, \label{Phres4}
\\
Q\dgx3(\Resx(f))&=27\,J. \label{Qhres4}
  \end{align}

\subsection{Seminvariants of $f'$}\label{SS4f'}
We calculate the basic seminvariants of the cubic polynomial $f'$:
\begin{align}
  \gD\dgx3(f')&=
-432\,a_0^{2}\,a_3^{2}
+432\,{a_0}\,{a_1}\,{a_2}\,{a_3}
-128\,{a_0}\,a_2^{3}
-108\,a_1^{3}\,{a_3}
+36\,a_1^{2}\,a_2^{2}
\notag
\\&
=16\,a_0\,J-12\,I P,
\label{gdf'4}
\\
P\dgx3(f')&=
24\,{a_0}\,{a_2}-9\,a_1^{2}
=3P,
\label{pf'4}
\\
Q\dgx3(f')&=
432\,a_0^{2}\,{a_3}-216\,{a_0}\,{a_1}\,{a_2}
+54\,a_1^{3}
=54 \,Q.
\label{qf'4}
\end{align}

\subsection{The case $a_0=0$}\label{SS40}

When $a_0=0$, \ie, considering the restriction to polynomials of degree 3,
we have, \cf{} \refE{Ediskr-1}, for any polynomial $f\in\cP_3$,
\begin{align}
  \gD\dgx4
&=a\dgxx03^2\gD\dgx3,
\\
  I\dgx4
&=-P\dgx3,
\\
  J\dgx4
&=-Q\dgx3,
\\
  P\dgx4
&=-3\,a\dgxx03^2,
\\
  Q\dgx4
&=a\dgxx03^3,
\intertext{In particular, for $f\in\cP_4$, using \eqref{gdf'4}--\eqref{qf'4},}
  \gD\dgx4(f')&=16\,a_0^2\,\gD\dgx3(f')
=256\,a_0^3\,J-192\,a_0^2\,I\,P,
\\
  I\dgx4(f')
&=-3\,P,
\\
  J\dgx4(f')
&=-54\,Q,
\\
  P\dgx4(f')
&=-48\,a_0^2,
\\
  Q\dgx4(f')
&=64\,a_0^3.
\end{align}

\subsection{Seminvariants and roots}\label{SS4roots}
By \refE{DD}, 
\begin{equation}
  \gD=a_0^6(\xi_1-\xi_2)^2(\xi_1-\xi_3)^2(\xi_1-\xi_4)^2
      (\xi_2-\xi_3)^2(\xi_2-\xi_4)^2(\xi_3-\xi_4)^2.
\end{equation}

For $I$ and $J$ we obtain by calculations, using $\sum^*$ to denote a sum
over different indices, where moreover identical terms are counted only once
(thus, for example, $\sum^*_{i,j}\xi_i\xi_j=\sum_{i<j}\xi_i\xi_j$),
\begin{equation}\label{I4xi}
I=
a_0^2\biggpar{
\sumx_{i,j}\xi_i^2\xi_j^2
-\sumx_{i,j,k}\xi_i^2\xi_j\xi_k
+6\xi_1\xi_2\xi_3\xi_4}
\end{equation}
where the first sum has 6 terms and the second 12,
and
\begin{multline}\label{J4xi}
J=
a_0^3\biggpar{
-2\sumx_{i,j}\xi_i^3\xi_j^3
+3\sumx_{i,j,k}\xi_i^3\xi_j^2\xi_k
-12\sumx_{i,j,k,l}\xi_i^3\xi_j\xi_k\xi_l
\\
-12\sumx_{i,j,k}\xi_i^2\xi_j^2\xi_k^2
+6\sumx_{i,j,k,l}\xi_i^2\xi_j^2\xi_k\xi_l
},  
\end{multline}
where the sums have 6, 24, 4, 4 and 6 terms. 

For the seminvariants we have first,
by \eqref{Hroots}, 
\begin{equation}
  H_0=-3\,a_0^2\sum_{1\le i<j\le 4}(\xi_i-\xi_j)^2
	=-3\,a_0^2\biggpar{3\sum_{i=1}^4\xi_i^2-2\sum_{1\le i<j\le 4}\xi_i\xi_j}
\end{equation}
and thus, by \eqref{HP4},
\begin{equation}\label{P4xi}
  P=-\,a_0^2\sum_{1\le i<j\le 4}(\xi_i-\xi_j)^2
=
-a_0^2\biggpar{
3\sum_i \xi_i^2
-2\sumx_{i,j}\xi_i\xi_j}
\end{equation}
where the sums have  4 and 6 terms. Further, by calculation,
\begin{equation}\label{Q4xi}
 Q=
-a_0^3\biggpar{\sum_i\xi_i^3
-\sumx_{i,j}\xi_i^2\xi_j
+2\sumx_{i,j,k}\xi_i\xi_j\xi_k}
\end{equation}
where the sums have  4, 12  and 4 terms.

The formulas \eqref{I4xi} and \eqref{P4xi} for $I$ and $P$
cannot be factorized further, but for $J$ and $Q$ we have
\begin{equation}\label{J4xifact}
  \begin{split}
J&=
-a_0^3
\bigpar{\left( {\xi_1}-{\xi_3} \right)  \left( {\xi_2}-{\xi_4} \right) 
 + \left( {\xi_1}-{\xi_4} \right)  \left( {\xi_2}-{\xi_3} \right)}  
\\&\qquad\qquad\cdot
\bigpar{ \left( {\xi_1}-{\xi_2} \right)  \left({\xi_3}-{\xi_4} \right) 
 + \left( {\xi_1}-{\xi_4} \right)  \left( {\xi_3}-{\xi_2} \right)  }
\\&\qquad\qquad\cdot
\bigpar{\left( {\xi_1}-{\xi_2} \right)  \left( {\xi_4}-{\xi_3} \right) 
 + \left( {\xi_1}-{\xi_3} \right)  \left( {\xi_4}-{\xi_2} \right)}
  \end{split}
\end{equation}
and
\begin{equation}\label{Q4xifact}
  \begin{split}
Q&=
-a_0^{3} \left( {\xi_1}+{\xi_2}-{\xi_3}-{\xi_4} \right) 
 \left( {\xi_1}-{\xi_2}+{\xi_3}-{\xi_4} \right)  
\left( {\xi_1}-{\xi_2}-{\xi_3}+{\xi_4} \right) 	.
  \end{split}
\end{equation}

Explicit formulas for the covariants $H$ and $\JH$ in terms of the roots
can be obtained from \refE{EHessroots} and \refT{Tcovroots} together with
\eqref{Q4xi} or \eqref{Q4xifact}. We leave these to the reader.

\subsection{Cross ratio}\label{SS4cr}
Let $F$ be a field and $\Fx\=F\cup\set\infty$. The \emph{cross ratio}
$[x_1,x_2;x_3,x_4]$ is defined for $x_1,x_2,x_3,x_4\in\Fx$ by
\begin{equation}
  \label{cr}
\crx{x}\= 
\frac{\left( {x_1}-{x_3} \right)  \left( {x_2}-{x_4} \right) }
{\left( {x_1}-{x_4} \right)  \left( {x_2}-{x_3} \right)}
\in\Fx.
\end{equation}
More precisely, the cross ratio is well-defined by \eqref{cr} if
$x_1,x_2,x_3,x_4\in F$ are distinct, and more generally if
$x_1,x_2,x_3,x_4\in \Fx$ are distinct with the natural interpretations
\begin{equation}
\begin{aligned}
{}  [\infty,x_2;x_3,x_4]&=\frac{x_2-x_4}{x_2-x_3},
&\qquad
  [x_1,\infty;x_3,x_4]&=\frac{x_1-x_3}{x_1-x_4},
\\
  [x_1,x_2;\infty,x_4]&=\frac{x_2-x_4}{x_1-x_4},
&\qquad
  [x_1,x_2;x_3,\infty]&=\frac{x_1-x_3}{x_2-x_3}.
\end{aligned}  
\end{equation}
Furthermore, the cross ratio is also  defined when two of
$x_1,\dots,x_4\in\Fx$ coincide, and even when two different pairs of them
coincide. (In these cases, the cross ratio is always 0, 1 or $\infty$; it is
$0$ if $x_1=x_3$ or $x_2=x_4$, 1 if $x_1=x_2$ or $x_3=x_4$, and $\infty$ if
$x_1=x_4$ or $x_2=x_3$.)
In the remaining cases, when three or four of $x_1,\dots,x_4$ coincide, the
cross ratio is undefined.

If $x_2,x_3,x_4\in\Fx$ are distinct, then $x\mapsto[x,x_2;x_3,x_4]$ is the
unique projective (= fractional linear) map $\Fx\to\Fx$ that maps
$x_2\mapsto1$, $x_3\mapsto0$, $x_4\mapsto\infty$.

The cross ratio depends on the order of $x_1,\dots,x_4$, and the 24
different permutations give, in general, 6 different values. These values
determine each other;
if $\crx{x}=\gl$, then, 
\begin{align*}
\crx x&=
  [x_2,x_1;x_4,x_3]=[x_3,x_4;x_1,x_2]=[x_4,x_3;x_2,x_1]=\gl,
\\
\crxx1243&=\crxx2134=\crxx3421=\crxx4312=\frac1\gl,
\\
\crxx1324&=\crxx2413=\crxx3142=\crxx4231=1-\gl,
\\
\crxx1342&=\crxx2431=\crxx3124=\crxx4213=\frac{1}{1-\gl},
\\
\crxx1423&=\crxx2314=\crxx3241=\crxx4132=\frac{\gl-1}{\gl},
\\
\crxx1432&=\crxx2341=\crxx3214=\crxx4123=\frac{\gl}{\gl-1}.
\end{align*}
The symmetric group $S_4$ thus acts on the space $\Fx$. The functions of
$\gl$ above are all projective maps, and thus we have a homomorphism of
$S_4$ into the group $PGL(1,F)$ of projective maps; the kernel is the
four-group and the image is a subgroup of $PGL(1,F)$ of order 6, isomorphic
to $S_3$ (for example, by their permutations of \set{0,1,\infty}).
The orbits have in general 6 elements, but orbits including a fixpoint of
one of the non-trivial maps above are smaller; there are two or three such
exceptional orbits, \viz{} \set{0,1,\infty}, \set{-1,\frac12,2}, and,
provided $\sqrt{-3}\in F$,
\set{\frac12\pm\frac{\sqrt{-3}}2}.

We have $\crx{x}\in\set{0,1,\infty}$ if and only if two of $x_1,\dots,x_4$
coincide.

Quadruples $x_1,\dots,x_4$ with $\crx{x}\in\set{-1,\frac12,2}$ are called 
\emph{harmonic quadruples}. (For example, one point at infinity and three
points in an aritmetic sequence, such as $-1,0,1,\infty$. Another example is
four points equally spaced on a circle, such as $1,i,-1,-i$.)

Quadruples $x_1,\dots,x_4$ with $\crx{x}\in\set{\frac12\pm\frac{\sqrt{-3}}2}$
are called 
\emph{self-apolar} or \emph{equianharmonic}.
(For example, three points evenly spaced on a circle, together with either
the centre or infinity, such as $0,1,e^{2\pi\ii/3},e^{4\pi\ii/3}$.)

If $f\in\cP_4$, let $\xi_1,\dots,\xi_4$ be its roots, and
$\gl\=\crx{\xi}$. Then $\gl$ depends on the ordering of the roots, as
explained above, but the polynomial
\begin{equation}
\gL(z)\=(z-\gl)\Bigpar{z-\frac1\gl}\Bigpar{z-(1-\gl)}\Bigpar{z-\frac{1}{1-\gl}}
\Bigpar{z-\frac\gl{\gl-1}}\Bigpar{z-\frac{\gl-1}\gl}
\end{equation}
does not depend on the order, so it depends on $f$ only.
The coefficients of $\gL(z)$ are symmetric rational functions of
$\xi_1,\dots,\xi_4$, and are thus rational functions of the coefficients
$a_0,\dots,a_4$ of $f$. Moreover, $\gL(z)$ is invariant under projective
transformations, and is thus an absolute invariant of $f$.
A calculation yields, using \eqref{DIJ4},
\begin{equation}
  \begin{split}
\gL(z)&=z^6-3\,z^5-\frac{I^3+2\,J^2}{9\,\gD}z^4  
+\frac{26\,I^3+7\,J^2}{27\,\gD}z^3
-\frac{I^3+2\,J^2}{9\,\gD}z^2
-3\,z+1 	
\\
&=z^6-3\,z^5-\frac{3\,I^3+6\,J^2}{4\,I^3-J^2}z^4  
+\frac{26\,I^3+7\,J^2}{4\,I^3-J^2}z^3
-\frac{3\,I^3+6\,J^2}{4\,I^3-J^2}z^2 
-3\,z+1 ,	
  \end{split}
\end{equation}
where we recognize (slightly disguised) the absolute invariant $I^3/J^2$,
see \refE{E4abs}. 
We have $\gL(\gl)=0$, 
\ie,
{\multlinegap=0pt
\begin{multline}
\left(4{\it I^3} -{\it J^2} \right) {\gl}^{6}
+ \left( -12{\it I^3}+3{\it J^2} \right) {\gl}^{5}
+ \left( -3{\it I^3}-6{\it J^2} \right) {\gl}^{4}
\\
+ \left( 26{\it I^3}+7{\it J^2} \right) {\gl}^{3}
+ \left( -3{\it I^3}-6{\it J^2} \right) {\gl}^{2}
+ \left( -12{\it I^3}+3{\it J^2} \right) \gl
+4{\it I^3} -{\it J^2}
\\\shoveleft{
=
\bigpar{4{\gl}^{6}-12{\gl}^{5}-3{\gl}^{4}+26{\gl}^{3}-3{\gl}^{2}
-12\gl+4}I^3}
\\\shoveright{
-\bigpar{
{\gl}^{6}-3{\gl}^{5}+6{\gl}^{4}-7{\gl}^{3}+6{\gl}^{2}-3\gl+1}
J^2}
\\\shoveleft{
=0,\hfill } 
\end{multline}
}
which after a rearrangement yields
\begin{equation}\label{IJgl4}
\begin{split}
  \frac{J^2}{I^3}=&
\frac
{4\,{\gl}^{6}-12\,{\gl}^{5}-3\,{\gl}^{4}+26\,{\gl}^{3}-3\,{\gl}^{2}-12\,\gl+4}
{{\gl}^{6}-3\,{\gl}^{5}+6\,{\gl}^{4}-7\,{\gl}^{3}+6\,{\gl}^{2}-3\,\gl+1}
\\
=&
{\frac { \left( \gl-2 \right) ^{2} \left( 2\,\gl-1 \right) ^{2} 
\left( \gl+1 \right) ^{2}}{ \left( {\gl}^{2}-\gl+1 \right) ^{3}}}
\end{split}
\end{equation}
and equivalently, using \eqref{DIJ4} again,
\begin{equation}\label{IDgl4}
\begin{split}
  \frac{I^3}{\gD}=&
\frac
{{\gl}^{6}-3\,{\gl}^{5}+6\,{\gl}^{4}-7\,{\gl}^{3}+6\,{\gl}^{2}-3\,\gl+1}
{\gl^4-2\,\gl^3+\gl^2}
=
\frac
{ \left( {\gl}^{2}-\gl+1 \right) ^{3}}
{\gl^2\,(\gl-1)^2}
\end{split}
\end{equation}
and
\begin{equation}\label{JDgl4}
\begin{split}
  \frac{J^2}{\gD}=&
\frac
{4\,{\gl}^{6}-12\,{\gl}^{5}-3\,{\gl}^{4}+26\,{\gl}^{3}-3\,{\gl}^{2}-12\,\gl+4}
{\gl^4-2\,\gl^3+\gl^2}
=
\frac
{\left(\gl-2 \right)^{2} \left( 2\,\gl-1 \right)^{2} \left( \gl+1 \right)^{2}}
{\gl^2\,(\gl-1)^2}.
\end{split}
\end{equation}
We have really proved these formulas for the case of four distinct roots
in $F$, but it is easy to see that they hold also in the case of one or two
double roots (in this case $\gD=0$ and $\gl\in\set{0,1,\infty}$), 
and (by projective invariance) also if there is a single or
double root at $\infty$.
Note that if two of $I$, $J$ and $\gD$ vanish, then so do all three because
of \eqref{DIJ4}; this happens if and only 
there is a triple (or quadruple) root (see \refT{Tinv0}),
and then cross ratio $\gl$ is undefined. In this case thus both sides of 
\eqref{IJgl4}--\eqref{JDgl4} are undefined. Otherwise, if there is no triple
root, 
at most one of $I$, $J$ and $\gD$ vanishes, and
both sides of \eqref{IJgl4}--\eqref{JDgl4} are defined as elements of
$\Fx$ (they may be $\infty$, \viz{} when the denominator vanishes or, for
\eqref{IDgl4}--\eqref{JDgl4}, when $\gl=\infty$), and they
are equal.

\subsection{Further examples}\label{SS4further}
A simple example of higher invariants is $A(f^\nu,f^\nu)$ for $\nu\ge1$.
This has degree $2\nu$ and weight $4\nu$.
We have $A(f,f)=4I$ by \eqref{A4} and, for example, 
\begin{align}
  A\dgx{8}(f^2,f^2)
&=
82944\,a_0^{2}\,a_4^{2}-41472\,{a_0}\,{a_1}\,{a_3}\,{a_4}
+13824\,{a_0}\,a_2^{2}\,{a_4}
\notag\\& \qquad\qquad
+5184\,a_1^{2}\,a_3^{2}-3456\,{a_1}\,a_2^{2}\,{a_3}+576\,a_2^{4}
\notag\\&
=576\, I^2,
\\
  A\dgx{12}(f^3,f^3)&=
564480\,I^3-11520\,J^2.
\end{align}

The apolar invariant of the Hessian covariant is an invariant 
given by, see \eqref{H4} and \eqref{apolar},
\begin{equation}
  \begin{split}
 A\dgx4(H(f),H(f))&=
82944\,a_0^{2}\,a_4^{2}
-41472\,{a_0}\,{a_1}\,{a_3}\,{a_4}
+13824\,{a_0}\,a_2^{2}\,{a_4}
\\&\qquad\qquad\qquad
+5184\,a_1^{2}\,a_3^{2}
-3456\,{a_1}\,a_2^{2}\,{a_3}
+576\,a_2^{4} \qquad
\\&
=576\, I^2.	
  \end{split}
\raisetag\baselineskip
\end{equation}
Equivalently, 
$ A(\tH(f),\tH(f))=64\,I^{2}$.

Similarly, omitting the details,
\begin{align}
 A(\tH(f)^2,\tH(f)^2)&=147456\,I^{4},
  \\
 A(\tH(f)^3,\tH(f)^3)&=
2123366400\,{I}^{6}+188743680\,{I}^{3}{J}^{2}-47185920\,{J}^{4}
\notag
\\&
= 47185920(5\,I^3+J^2)(9\,I^3-J^2),
\end{align}
with the coefficients $147456=2^{14}3^2$ and 
$47185920=2^{20}3^2 5$.

Equivalently,
\begin{align}
 A(H(f)^2,H(f)^2)&=11943936\,I^{4},
\\
 A(H(f)^3,H(f)^3)&
= 34398535680\, (5\,I^3+J^2)(9\,I^3-J^2),  
\end{align}
where 
$11943936=2^{14}3^6$
and
$34398535680=2^{20}3^8 5$.

The (joint) apolar invariant $A(H(f),f)$ is an invariant 
given by, see \eqref{H4} and \eqref{apolar},
\begin{equation}\label{AHf4}
  \begin{split}
A(H(f),f)
&=
1728\,{a_0}\,{a_2}\,{a_4}
-648\,{a_0}\,a_3^{2}
-648\,a_1^{2}\,{a_4}
+216\,{a_1}\,{a_2}\,{a_3}
-48\,a_2^{3}
\\&
=24 J,
  \end{split}
\raisetag{12pt}
\end{equation}
see \eqref{J4}.
This has degree 3 and weight 6.
Equivalently,
$A(\tH(f),f)=8J$.
We can also form, for example,
\begin{equation}
  A\dgx{8}\bigpar{f\tH(f),f\tH(f)}=192\,(20\,I^3+7J^2).
\end{equation}

Further invariants (\etc) of the Hessian covariant are
\begin{align}
    I(\tH(f))&=16\,I^2,
\\
    J(\tH(f))&=64\,J^2-128\,I^3=64\bigpar{J^2-2\,I^3},
\\
    \gD(\tH(f))&=2^{12} J^2\,\gD,
\\
  P(\tH(f))&=64\,a_0\,J-16\,I\,P,
\\
  Q(\tH(f))&=-64\,J\,Q,
\\
  \tH(\tH(f))&=64\,J\,f-16\,I\,\tH,
\\
  \tJH(\tH(f))&=-64\,J\,\tJH.
\end{align}

The apolar invariant $A(\tJH,\tJH)$ 
of the sextic polynomial $\tJH(f)$
is an invariant of degree 6 and weight 12
given by 
\begin{equation}\label{F6D4}
  A\dgx6(\tJH(f),\tJH(f))=960\,\gD.
\end{equation}
The discriminant $\gD(\tJH)$ is an invariant of degree 30 and
weight 60 
given by 
\begin{equation}\label{DF6D4}
  \gD\dgx6(\tJH(f))=-2^{18}\,\gD^5.
\end{equation}

We calculate also the resultants of $f$, $\tH(f)$ and $\tJH(f)$:
\begin{align}
  R\xpar{f,\tH}&=81\gD^2, \\
  R\xpar{f,\tJH}&=\gD^3, \\
  R\xpar{\tH,\tJH}&=2^{12}\gD^3J^2, 
\end{align}
where the first also follows by \refE{EResHess}.

For the seminvariants in Examples \ref{Eredb}--\ref{Eredc}, we have,
recalling $\ar_2=p=P/8a_0$, $\ar_3=q=Q/8a_0^2$ and $\ar_4=r=R/256a_0^3$, see
\eqref{red4} and \eqref{p4P}--\eqref{r4R}, and using also \eqref{RIP4},
\begin{align}
  a_0^2S_2&=-\frac14 P, \label{s42}\\
  a_0^3S_3&=-\frac38Q,\label{s43}\\
  a_0^4S_4&=-\frac1{64}R+\frac1{32}P^2
=-\frac13a_0^2I+\frac7{192}P^2
,\label{s44}
\intertext{and}
 a_0^2 \chi_2&=-\frac1{16}P,
\label{achi42}\\
a_0^3 \chi_3&=-\frac3{32}Q,
\label{achi43}\\
a_0^4  \chi_4&
=-\frac1{256}R-\frac{1}{256}P^2
=-\frac1{384}\bigpar{32a_0^2I+P^2}.
\label{achi44}
\end{align}

\subsection{Vanishing invariants and covariants}\label{SS4=0}

Let $f$ be a quartic polynomial, with roots $\xi_1,\dots,\xi_4$.
(The results extend to the case $a_0=0$ when one or several roots are
$\infty$ with no or trivial modifications.)

Since $I$ and $J$ form a basis for the invariants (\refT{T4bas}), 
\refT{Tinv0} shows that
$I=J=0$ if and only if $f$ has a triple root. Moreover, since
$\gD=\frac{4}{27}I^3-\frac1{27}J^2$ by \eqref{DIJ4}, we have the following:

\begin{theorem}\label{T4IJD=000}
Let $f$ be a quartic polynomial.
  If $f$ has
a triple (or quadruple) root then $\gD=I=J=0$.

Conversely,
if there is no triple root, then at most one of $\gD$, $I$ and $J$ vanishes.
\end{theorem}
 
If there is no triple root, the cross ratio $\crx\xi$ of the roots is
well-defined by \refSS{SS4cr}, and the vanishing of the basis invariants $I$
and $J$, as well as $\gD$, can be characterised by this cross ratio. 

\begin{theorem}\label{T4IJD=0}
Let $f$ be a quartic polynomial with roots $\xi_1,\xi_2,\xi_3,\xi_4$, and
assume that there is no triple (or quadruple) root.
\begin{romenumerate}
\item 
$I=0$ if and only if the cross-ratio
$[\xi_1,\xi_2;\xi_3,\xi_4]=\frac12\pm\frac{\sqrt3}{2}\ii$,
\ie, if and only if the roots form a equianharmonic (self-apolar) quadruple.
\item 
$J=0$ if and only if the cross-ratio
$[\xi_1,\xi_2;\xi_3,\xi_4]\in\set{-1,\frac12,2}$,
\ie, if and only if the roots form a harmonic quadruple.
\item 
$\gD=0$ if and only if the cross-ratio
$[\xi_1,\xi_2;\xi_3,\xi_4]\in\set{0,1,\infty}$,
\ie, if and only there is a double root.
\end{romenumerate}
\end{theorem}

Note that the three conditions use the three exceptional orbits of cross
ratios, see \refSS{SS4cr}.

\begin{proof}
  When there is no triple root, the cross ratio 
$\gl=\crx\xi$
is well-defined by \refSS{SS4cr} and 
at most one of $I$, $J$ and $\gD$ vanishes by \refT{T4IJD=000}; 
the results now follow from
\eqref{IJgl4}--\eqref{JDgl4}. (The result for $\gD$ is of course an
immediate consequence of \eqref{TD=0}.)
\end{proof}

\refT{T4IJD=0}(ii) 
also follows from \eqref{J4xifact}, which shows that $J=0$ if
and only if 
one of the three
factors in the brackets there vanishes, or equivalently 
that one of the three cross-ratios 
\begin{align*}
\frac{\left( {\xi_1}-{\xi_3} \right)  \left( {\xi_2}-{\xi_4} \right) }
{\left( {\xi_1}-{\xi_4} \right)  \left( {\xi_2}-{\xi_3} \right)},  
&&
\frac{ \left( {\xi_1}-{\xi_2} \right)  \left({\xi_3}-{\xi_4} \right)} 
{\left( {\xi_1}-{\xi_4} \right)  \left( {\xi_3}-{\xi_2} \right)  },
&&
\frac{\left( {\xi_1}-{\xi_2} \right)  \left( {\xi_4}-{\xi_3} \right)} 
{\left( {\xi_1}-{\xi_3} \right)  \left( {\xi_4}-{\xi_2} \right)}
&
\end{align*}
equals $-1$. 
(These are $[\xi_1,\xi_2;\xi_3,\xi_4]$,
$[\xi_1,\xi_3;\xi_2,\xi_4]$ and
$[\xi_1,\xi_4;\xi_2,\xi_3]$.)

\refC{CGund=0} gives another interpretation of $J=0$, since $J$ is a
multiple of $G_2$ (the catalecticant when $n=4$):

\begin{theorem}
  The following are equivalent for a quartic polynomial $f$:
  \begin{romenumerate}
  \item $J=0$.
  \item 
$f$ belongs to the closure $\bcP_{4,2}$ of the set
	$\cP_{4,2}\=\set{c_1(x-x_1)^2+c_2(x-x_2)^4}$.
  \item $f$ has one of the forms
$c_1(x-x_1)^4+c_2(x-x_2)^4$, 
$c_1(x-x_1)^4+c_2$, 
$c_1(x-x_1)^4+c_2(x-x_1)^3$, 
$c_1+c_2x$. (The last two comprise the cases when $f$ has a triple root,
	finite or infinite).
  \end{romenumerate}
\end{theorem}

For the covariants $H$ and $\JH$ we have the following.
The first is just an instance of the general \refT{THess=0}.
\begin{theorem}\label{T4H=0} 
The following are equivalent for a quartic polynomial $f$.
  \begin{romenumerate}
  \item $H(f)=0$. 
  \item $f$ has a single, quadruple root, \ie, $\xi_1=\xi_2=\xi_3=\xi_4$.
  \item 	$f(x)=c(x-x_0)^4$ for some $c$ and $x_0$.
  \end{romenumerate}
\end{theorem}

\begin{theorem} \label{T4JH=0}
The following are equivalent for a quartic polynomial $f$.
  \begin{romenumerate}
  \item $\JH(f)=0$.
  \item Every root is (at least) a double root.
  \item The roots coincide in two pairs $\xi_1=\xi_2$ and $\xi_3=\xi_4$
(up to labelling); this includes the case when all four roots coincide.
  \item $f=cg^2$ for some quadratic polynomial $g$.
  \end{romenumerate}
\end{theorem}
\begin{proof}
It is easy to see that (ii), (iii) and (iv) are equivalent.

Suppose now (i), \ie, $\JH(f)=0$. By \eqref{F6D4}, then $\gD=0$, so $f$ has a
double root $\xi$. By projective invariance, we may assume that $\xi=0$, so
$f(x)=a_0x^4+a_1x^3+a_2x^2$. For $f$ of this form, with $a_3=a_4=0$,
\eqref{JH4} reduces to
\begin{equation}\label{qk}
  \begin{split}
\tJH(f)&=	
 \left(-4\,{a_0}\,{a_1}\,{a_2}  +a_1^{3} \right) {x}^{6}
+ \left( 
  -8\,{a_0}\,a_2^{2}+2\,a_1^{2}\,{a_2} \right) {x}^{5}
\\&
=\bigpar{a_1^2-4\,a_0\,a_2}\bigpar{a_1x^6+2\,a_2x^5}.
  \end{split}
\end{equation}
Hence either $a_1^2-4a_0a_2=0$ or $a_1=a_2=0$; in both cases
$\gD\dgx2(a_0x^2+a_1x+a_2)=a_1^2-4a_0a_2=0$.
Hence, $a_0x^2+a_1x+a_2$ has a double root $\xi$, and $f$ has the roots
$0,0,\xi,\xi$. 

Conversely, if $f$ has only double roots, we may again by projective
invariance assume that 0 is a root, and then
$f(x)=a_0x^4+a_1x^3+a_2x^2$, where we now know that
also $a_0x^2+a_1x+a_2$ has a double root, and thus its discriminant
$a_1^2-4a_0a_2=0$. Hence, $\JH(f)=0$ by \eqref{qk}.
\end{proof}

\subsection{Roots and resolvent of a real quartic}\label{SS4rr}

Consider a real quartic $f$, with $a_0\neq0$.
Then $f$ has either 0, 2 or 4 real roots (counted with multiplicities).
The discriminant partly discriminates between these cases, by the following
simple and classic result, which is a simple consequence of \eqref{dd},
see \refR{RD}.
(In this subsection, ``complex'' means non-real.)

\begin{theorem}
  \label{T4realroots1}
  Let $f$ be a real quartic polynomial.
  \begin{romenumerate}
\item \mbox{$\gD(f)>0 \iff$}
  $f$ has either 4 distinct real roots, \orr 4 complex
  roots in two conjugate pairs.
\item $\gD(f)<0 \iff$
  $f$ has 2 real roots and 2 conjugate complex roots.
\item $\gD(f)=0 \iff$ $f$ has a double (or triple or quadruple) root.
In this case,
  $f$ has 1 quadruple real root,
\orr 2 real roots, one triple and one single,
\orr 2 double real roots,
\orr 3 real roots, one double and two single,
\orr 1 double real root and 2 conjugate complex roots,
\orr 2 conjugate complex double roots.
  \end{romenumerate}
\end{theorem}

To completely
distinguish between the different cases we employ further seminvariants
and covariants. (In the following theorem the roots are assumed to be
distinct except as explicitly stated.)
Note that $I\ge0$ when $\gD\ge0$ by \eqref{DIJ4}, so $\sqrt I\ge0$  in
this case.

\begin{theorem}
  \label{T4realroots2}
  Let $f$ be a real quartic polynomial.
  \begin{romenumerate}
\item \label{rr1111}
  $f$ has 4  real roots $\iff$
$\gD>0$, $P\le0$ and $P^2-16a_0^2I\ge0$
$\iff$
$\gD>0$ and  $P\le-4a_0\sqrt{I}$.
\item \label{rrc22}
$f$ has 2 pairs of conjugate complex  roots 
$\iff$ 
$\gD>0$ and either  $P>0$ or $P^2-16a_0^2I<0$
$\iff$
$\gD>0$ and  $P>-4a_0\sqrt{I}$.
\item \label{rr11c2}
  $f$ has 2  real roots and 2 conjugate complex roots  $\iff$
  $\gD<0$. 
\item \label{rr4} 
  $f$ has 1 quadruple real root  
$\iff \gD=I=J=P=0$
$\iff$ $H(x)\equiv0$.
In this case also $Q=0$ and $\JH(x)\equiv0$.
\item \label{rr31}
$f$ has 1 triple and 1 single real root 
$\iff \gD=I=J=0$ but $P\neq0$. In this case $P<0$, $Q\neq0$,
  $H(x)\not\equiv0$, $\JH(x)\not\equiv0$.
\item \label{rr22}
$f$ has 2 double real roots  
$\iff \gD=P^2-16a_0^2I=0$ and $P<0$
$\iff\JH(x)\equiv0$ and $P<0$.
In this case also $Q=0$.
\item \label{rrc4}
$f$ has 2 conjugate complex double roots   
$\iff \gD=P^2-16a_0^2I=Q=0$ and $P>0$
$\iff\JH(x)\equiv0$ and $P>0$.
\item \label{rr211}
$f$ has  3 real roots, one double and two single  
$\iff \gD=0$, $I>0$, $P<0$ and $P^2-16a_0^2I>0$.
\item \label{rr2c2}
$f$ has 1 double real root and 2 conjugate complex roots  
$\iff \gD=0$ and either  $P^2-16a_0^2I<0$ or $P>0$ but not 
$P^2-16a_0^2I=Q=0$.
  \end{romenumerate}
\end{theorem}

\begin{proof}
\pfitemx{\ref{rr1111},\ref{rrc22}}
By \refT{T4realroots1}, these cases are characterized by $\gD>0$.
To distinguish the two cases, we note that by \refR{Rgam}, $f$ has 4 real
roots $\iff$ $\gam_1,\gam_2,\gam_3\in\bbR$
$\iff$ $\gam_1^2,\gam_2^2,\gam_3^2\in[0,\infty)$.
Since $\gam_1^2,\gam_2^2,\gam_3^2$ are the roots of the cubic resolvent
$\Res(f)$,
it follows from \refT{T3>0} 
and \eqref{Res4}
that $f$ has 4 real roots $\iff$
$\gD(\Res(f))\ge0$, $P\le0$, $P^2-R\ge0$, and $-Q^2\le0$.
Since $\gD(\Res(f))=a_0^{-6}\gD$ by \eqref{gDres4}
and $P^2-R=\frac43(P^2-16a_0^2I)$ by \eqref{P2-R}, the results follow. 

\pfitemref{rr11c2} By \refT{T4realroots1}.

In the remaining cases $f$ has a multiple root and $\gD=0$. Note that then
$4I^3=J^2$ by \eqref{DIJ4}; in particular, $I\ge0$.
We calculate the seminvariants and covariants by \eqref{I4}, \eqref{J4},
\eqref{P4}, \eqref{Q4}, \eqref{H4}, \eqref{JH4} 
in the different cases to verify the direct parts of the assertions:

\pfitemref{rr4}
We may by invariance assume
$f=a_0x^4$, and then $I=J=P=Q=H(x)=\JH(x)=0$.

\pfitemref{rr31}
We may by invariance assume
$f=a_0x^3(x-u)$ where $u\in\bbR$ with $u\neq0$,
and then $I=J=0$, $P=-3a_0^2u^2<0$, $Q=-a_0^3u^3\neq0$,
$H(x)=-9a_0^2u^2x^2$,
$\JH(x)=-36a_0^3u^3x^6$.

\pfitemref{rr22}
We may by invariance assume
$f=a_0x^2(x-u)^2$ where $u\in\bbR$ with $u\neq0$,
and then 
$I=a_0u^2$, 
$P=-4a_0^2u^2<0$, 
$P^2-16a_0^2I=0$.
Further, $\JH(x)=0$ and thus $Q=0$ by \refT{T4JH=0}.

\pfitemref{rrc4}
We may by invariance assume 
$f=a_0(x-u-\ii v)^2(x-u+\ii v)^2$ for some real $u$ and $v\neq0$, and then
$I=16a_0^2v^4$, $P=16a_0^2v^2>0$, 
$P^2-16a_0^2I=0$, 
Further, $\JH(x)=0$ and thus $Q=0$ by \refT{T4JH=0}.

\pfitemref{rr211}
We may by invariance assume 
$f=a_0 x^2(x-u)(x-v)$ for some real $u,v\neq0$, and then
$I=a_0^2u^2v^2>0$, 
$P=-a_0^2(3u^2+3v^2-2uv)=-a_0^2(2u^2+2v^2+(u-v)^2)<0$,
$P^2-16a_0^2I=3a_0^4(u-v)^2(3u^2+3v^2+2uv)>0$. 

\pfitemref{rr2c2}
We may by invariance assume 
$f=a_0x^2(x-u-\ii v)^2(x-u+\ii v)^2$  for some real $u$ and $v\neq0$, and then
$I=a_0^2(u^2+v^2)^2$, 
$P=4a_0^2(2v^2-u^2)$, 
$P=8a_0^3v^2u$, 
$P^2-16a_0^2I=48a_0^4v^2(v^2-2u^2)$. 
If $v^2\ge u^2$, then $P>0$, and if $v^2< u^2$, then 
$P^2-16a_0^2I<0$. Further, $P^2-16a_0^2I=0 \iff v^2=2u^2$,
and $Q=0\iff u=0$, which cannot hold simultaneously. 

The converse implications in \ref{rr4},\ref{rr31}
now follow by \refT{T4IJD=000} and \refT{T4H=0}.

If $\JH(x)\equiv0$, then we have  \ref{rr4}, \ref{rr22} or \ref{rrc4}
by \refT{T4JH=0}, and they are by the calculations above distinguished by
the sign of $P$, which shows the converse implications assuming $\JH(x)\equiv0$.

It is easily verified that the other conclusions in \ref{rr22}--\ref{rr2c2}
are mutually exclusive, and also exclusive of \ref{rr4}--\ref{rr31}.
Hence the converse implications follow.
\end{proof}

The proof used some properties of the cubic resolvent. Let us study its
geometry further.
The cubic resolvent $\Res(f)$ of the real quartic $f$ has 
by \eqref{Res4}, \eqref{tRes4} and \eqref{g3y0}--\eqref{g3inf} 
an inflection point at
\begin{equation}
 (x_0,y_0)=\Bigpar{\frac{-P}{12a_0^2},\,\frac{J}{27a_0^3}}
\end{equation}
and, by 
\eqref{g3+-}--\eqref{gf3+-} and 
\eqref{Pres4}--\eqref{Qres4},
extreme points at
\begin{equation}\label{res4xy+-}
  (x_\pm,y_\pm)=\Bigpar{\frac{-P\pm 4a_0\sqrt{I}}{12a_0^2},\,\frac{J\mp2 I^{3/2}}{27a_0^3}}.
\end{equation}
For the version $\Resx(f)$ in \eqref{hRes4} we have simpler formulas:
an inflection point at
\begin{equation}
 (\hx_0,\hy_0)=\bigpar{0,J}
\end{equation}
and extreme points at
\begin{equation}\label{hres4xy+-}
  (\hx_\pm,\hy_\pm)=\bigpar{\pm\sqrt{I},J\mp2 I^{3/2}}.
\end{equation}

By \eqref{res4xy+-} or \eqref{hres4xy+-}, the resolvent has two distinct
real extreme points if and only if $I>0$, while the resolvent is
strictly increasing if $I\le0$, \cf{} \refSS{SS3geo} and \eqref{Pres4},
\eqref{Phres4}. 
We further see again that the resolvent has three distinct real roots if
and only if 
$I>0$ and $J-2I^{3/2}<0<J+2I^{3/2}$, or, equivalently, if and only if
$4I^3>J^2$, \ie, if and only if $\gD=\tfrac1{27}(4I^3-J^2)>0$, 
\cf{} \refT{Tf3roots} and  \eqref{gDres4}, \eqref{gDhres4}.  

Further, using \refR{Rgam} and \eqref{res4xy+-}, 
$f$ has 4 distinct real roots\\\hbox{\qquad} 
$\iff$ $\Res(f)$ has 3 roots in $[0,\infty)$\\\hbox{\qquad} 
$\iff$ $x_+>x_-\ge0$ and $y_->0>y_+$\\\hbox{\qquad} 
$\iff$ $I>0$, $-P\ge 4a_0\sqrt I$ and $J<2I\qqc$,\\
which by \eqref{DIJ4} yields another proof of \refT{T4realroots2}\ref{rr1111}.

Finally we note that, 
by \eqref{res4xy+-} and \eqref{P2-R},
\begin{equation}
  x_+x_-=\frac{P^2-16a_0^2I}{144a_0^4}
=\frac{P^2-R}{192a_0^4}
\end{equation}
and, using also \eqref{DIJ4} again, 
\begin{equation}
y_+y_-=\frac{J^2-4I^3}{729a_0^6}
=-\frac{\gD}{27a_0^3};
\end{equation}
hence, as observed by \citet{Nickalls:quartic}, 
the seminvariants $P^2-16a_0^2I$ and $J^2-4I^3=-\gD/27$ 
(ignoring normalizations) play a symmetric role in the geometry of the cubic
resolvent. Recall from \refT{T4realroots2} that these (together with $P$)
are the most important
seminvariants when determining the number of real roots, 
at least when the roots are simple.

\newcommand\AAP{\emph{Adv. Appl. Probab.} }
\newcommand\JAP{\emph{J. Appl. Probab.} }
\newcommand\JAMS{\emph{J. \AMS} }
\newcommand\MAMS{\emph{Memoirs \AMS} }
\newcommand\PAMS{\emph{Proc. \AMS} }
\newcommand\TAMS{\emph{Trans. \AMS} }
\newcommand\AnnMS{\emph{Ann. Math. Statist.} }
\newcommand\AnnPr{\emph{Ann. Probab.} }
\newcommand\CPC{\emph{Combin. Probab. Comput.} }
\newcommand\JMAA{\emph{J. Math. Anal. Appl.} }
\newcommand\RSA{\emph{Random Struct. Alg.} }
\newcommand\ZW{\emph{Z. Wahrsch. Verw. Gebiete} }
\newcommand\DMTCS{\jour{Discr. Math. Theor. Comput. Sci.} }

\newcommand\AMS{Amer. Math. Soc.}
\newcommand\Springer{Springer-Verlag}
\newcommand\Wiley{Wiley}

\newcommand\vol{\textbf}
\newcommand\jour{\emph}
\newcommand\book{\emph}
\newcommand\inbook{\emph}
\def\no#1#2,{\unskip#2, no. #1,} 
\newcommand\toappear{\unskip, to appear}

\newcommand\webcite[1]{
\texttt{\def~{{\tiny$\sim$}}#1}\hfill\hfill}
\newcommand\webcitesvante{\webcite{http://www.math.uu.se/~svante/papers/}}
\newcommand\arxiv[1]{\webcite{arXiv:#1.}}

\def\nobibitem#1\par{}

\end{document}